%% file: VarAmp_Journal_revision1_arXiv.tex
\documentclass[journal,onecolumn,draftcls]{IEEEtran}

\usepackage{braket}
\usepackage{array}
\usepackage[caption=false]{subfig}
\usepackage{pgfplots}

\usepackage{hyperref}       
\usepackage{url}            
\usepackage{booktabs}       
\usepackage{amsfonts}       
\usepackage{nicefrac}       
\usepackage{microtype}      
\usepackage{yfonts,dsfont}
\usepackage{setspace}

\usepackage{subfig,cite}

\usepackage{pifont}

\usepackage{mathtools}
\usepackage{mathcomp}
\usepackage{euscript,bbm,amstext,wasysym,arydshln,framed}

\usepackage{tikz}
\usepackage{amssymb}
\usetikzlibrary{positioning,shapes,arrows,calc,matrix}

\input{commands}

\renewcommand{\baselinestretch}{1.38} 

\newtheorem{mythm}{Theorem}
\newtheorem{myprop}{Proposition}
\newtheorem{mylem}{Lemma}
\newtheorem{myrem}{Remark}

\newcommand{\one}{\mathds{1}}

\newcommand{\R}{\mathbb{R}}
\newcommand{\x}{\mathbf{x}}
\newcommand{\N}{\mathbb{N}}

\newcommand{\sign}{\mathrm{sign}}

\newcommand{\bbR}{\mathbb{R}}

\newcommand{\non}{\nonumber}
\newcommand{\ds}{\displaystyle}

\newcommand{\mrd}{\mathrm{d}}

\newcommand{\tc}{\textcolor}
\newcommand{\vsp}{\vspace*{0.15cm}}

\newcommand{\bk}{\bar{\kappa}}

\newcommand{\muo}{\hat{\lambda}_{i}}
\newcommand{\mut}{\hat{\lambda}'_{i}}
\newcommand{\DefinedAs}[0]{\mathrel{\mathop:}=}

\DeclareMathOperator*{\argmin}{argmin}

\DeclareMathOperator*{\minimize}{minimize}

\DeclareMathOperator*{\subject}{subject~to}

\DeclareMathOperator*{\Null}{{\cal N}}
\newcommand{\Lap}{\mathbf{L}}

\DeclareMathOperator{\EX}{\mathbb{E}}

\newcommand{\Z}{\mathbb{Z}}
\newcommand{\T}{\mathbb{T}}
\newcommand{\gd}{\mathrm{gd}}
\newcommand{\na}{\mathrm{na}}
\newcommand{\hb}{\mathrm{hb}}
\newcommand{\binx}{i}
\newcommand{\Lf}{L}
\newcommand{\mf}{m}
\newcommand{\m}{n_0}

\tikzset{
	block/.style = {draw, rectangle, 
		minimum height=0.5cm, 
		minimum width=1cm},
	input/.style = {coordinate,node distance=2.5cm},
	output/.style = {coordinate,node distance=2.5cm},
	arrow/.style={draw, -latex,node distance=2cm},
	pinstyle/.style = {pin edge={latex-, black,node distance=2cm}},
	sum/.style = {draw, circle, node distance=1cm}
}

\begin{document}
	
		\title{Robustness of accelerated first-order algorithms 
		\\[-0.35cm]
		for strongly convex optimization problems
	}
	
	\author{
		\mbox{Hesameddin Mohammadi,
		Meisam Razaviyayn, and
		Mihailo R.\ Jovanovi\'c}
		\thanks{Financial support from the National Science Foundation under Awards ECCS~1708906 and ECCS~1809833, and the Air Force Office of Scientific Research under Award FA9550-16-1-0009 is gratefully acknowledged.
		}
		\thanks{H. Mohammadi and M. R.\ Jovanovi\'c are with the Ming Hsieh Department of Electrical and Computer Engineering, and M. Razaviyayn is with  the Daniel J. Epstein Department of Industrial and Systems Engineering, University of Southern California, Los Angeles, CA.
			E-mails: hesamedm@usc.edu,  razaviya@usc.edu, and mihailo@usc.edu.
		}
	}

	\maketitle
		
	\vspace*{-8ex}
	\begin{abstract}
			We study the robustness of accelerated first-order algorithms to stochastic uncertainties in gradient evaluation.  Specifically, for unconstrained, smooth, strongly convex optimization problems, we examine the mean-squared error in the optimization variable when the iterates are perturbed by additive white noise. This type of uncertainty may arise in situations where an approximation of the gradient is sought through measurements of a real system or in a distributed computation over a network. Even though the underlying dynamics of first-order algorithms for this class of problems are nonlinear, we establish upper bounds on the mean-squared deviation from the optimal solution that are tight up to constant factors. Our analysis quantifies fundamental trade-offs between noise amplification and convergence rates obtained via {\em any\/} acceleration scheme similar to Nesterov's or heavy-ball methods. To gain additional analytical insight, for strongly convex quadratic problems, we explicitly evaluate the steady-state variance of the optimization variable in terms of the eigenvalues of the Hessian of the objective function. We demonstrate that the entire spectrum of the Hessian, rather than just the extreme eigenvalues, influence robustness of noisy algorithms. We specialize this result to the problem of distributed averaging over undirected networks and examine the role of network size and topology on the robustness of noisy accelerated algorithms.  
	\end{abstract}
		
	\vspace*{-2ex}
	\begin{keywords}
		Accelerated first-order algorithms, consensus networks, control for optimization, convex optimization, integral quadratic constraints, linear matrix inequalities, noise amplification, second-order moments, semidefinite programming. 
	\end{keywords}
	
	\vspace*{-6ex}
	\section{Introduction}
	
First-order algorithms are well-suited for solving a broad range of optimization problems that arise in statistics, signal and image processing, control, and machine learning~\cite{botcun05,becteb09,honrazluopan15,botcurnoc18,nes13}. Among these algorithms, accelerated methods enjoy the optimal rate of convergence and they are popular because of their low per-iteration complexity. There is a large body of literature dedicated to the convergence analysis of these methods under different stepsize selection rules~\cite{susmardahhin13,pol64,nes83,nes13,becteb09,nes18book}. In many applications, however, the exact value of the gradient is not fully available, e.g., when the objective function is obtained via costly simulations (e.g., tuning of hyper-parameters in supervised/unsupervised learning~\cite{macduvada15,ben00,beirazshatar17} \tc{black}{ and model-free optimal control~\cite{fazgekakmes18,mohzarsoljovTAC19,mohsoljovACC20}}), when evaluation of the objective function relies on noisy measurements (e.g., real-time and embedded applications), or when the noise is introduced via communication between different agents (e.g., distributed computation over networks). Another related application arises in the context of (batch) stochastic gradient,  where at each iteration the gradient of the objective function is computed from a small batch of data points. Such a batch gradient is known to be a noisy unbiased estimator for the gradient of the training loss. Moreover, additive noise may be introduced deliberately in the context of non-convex optimization to help the iterates escape  saddle points and improve generalization~\cite{gehuajinyua15,jingenetkakjor17}.

In all above situations, first-order algorithms only have access to noisy estimates of the gradient. This observation has motivated the robustness analysis of first-order algorithms under different types of noisy/inexact gradient oracles\tc{black}{~\cite{luotse93,robmon51,nemjudlansha09,dev-phd13,devglines14,dvugas16}.} For example, in a deterministic noise scenario, an upper bound on the error in iterates for accelerated proximal gradient methods was established in~\cite{schroubac11}. This study showed that both proximal gradient and its accelerated variant can maintain their convergence rates provided that the noise is bounded and that it vanishes fast enough. \tc{black}{Moreover, it has been shown that in the presence of random noise,  with the proper diminishing stepsize,  acceleration  can be achieved for general convex problems. However, in this case optimal  rates are {\em sub-linear\/}~\cite{dev11}.}   

In the context of stochastic approximation, while early results suggest to use a stepsize that is inversely proportional to the iteration number~\cite{robmon51}, a more robust behavior can be obtained by combining larger stepsizes with averaging~\cite{nemjudlansha09,bac14,pol90,poljud92}. Utility of these averaging schemes and their modifications for solving quadratic optimization and manifold problems has been examined thoroughly in recent years~\cite{dieflabac17,moubac11,tripuraneni2018averaging}. Moreover, several studies have suggested that accelerated first-order algorithms are more susceptible to errors in the gradient \tc{black}{compared to} their non-accelerated counterparts\tc{black}{~\cite{bae09,das08,schroubac11,dev-phd13,devglines14,aujdos15}.}

One of the basic sources of error that arises in computing the gradient can be modeled by additive  white stochastic noise. This source of error is typical for problems in which the gradient is being sought through measurements of a real system~\cite{pol87} and it has a rich history in analysis of stochastic dynamical systems and control theory~\cite{kwasiv72}. Moreover, in many applications including distributed computing over networks~\cite{xiaboykim07,bamjovmitpat12}, coordination in vehicular formations~\cite{linfarjovTAC12platoons}, and control of power systems~\cite{dorjovchebulTPS14}, additive white noise is a convenient abstraction for the robustness analysis of distributed control strategies~\cite{bamjovmitpat12} and of first-order optimization algorithms~\cite{sim16,simkammon16}. Motivated by this observation, in this paper we consider the scenario in which a white stochastic noise with zero mean and identity covariance is added to the iterates of standard first-order algorithms: gradient descent, Polyak's heavy-ball method, and Nesterov's accelerated algorithm. By confining our attention to smooth strongly convex problems, we provide a tight quantitative characterization for the mean-squared error of the optimization variable. Since this quantity provides a measure of how noise gets amplified by the dynamics resulting from optimization algorithms, we also refer to  it as {\em noise\/} (or {\em variance\/}) {\em amplification.\/} We demonstrate that our quantitative characterization allows us to identify fundamental trade-offs between the noise amplification and the rate of convergence obtained via acceleration. 

This work builds on our recent conference papers~\cite{mohrazjovCDC18,mohrazjovACC19}. \tc{black}{In a concurrent work~\cite{aybfalgurozd19}, a similar approach was taken} to analyze the robustness of gradient descent and Nesterov's accelerated method. Therein, \tc{black}{it was shown that for a given convergence rate, one can select the algorithmic parameters such that the steady-state mean-squared error in the {\em objective value\/} of a Nesterov-like method becomes smaller than that of gradient descent.} This is not surprising because gradient descent can be viewed as a special case of Nesterov's method with a zero momentum parameter. \tc{black}{Using this argument, similar assertions have been made about the variance amplification of the {\em iterates\/}. This observation has been used to design an optimal multi-stage algorithm that does not require any information about the variance of the noise~\cite{aybfalgurozd19b}. On the contrary, we demonstrate  that there are fundamental differences between these two robustness measures, i.e., {\em \/objective values} and {\em iterates\/}, as the former does not capture the negative impact of acceleration in the presence of noise.}

\tc{black}{ Focusing on the error in the iterates, we show that} any choice of parameters for Nesterov's or heavy-ball methods that yields an accelerated convergence rate increases variance amplification relative to gradient descent. More precisely, {\em for the problem with the condition number $\kappa$, an algorithm with accelerated convergence rate of at least $1-c/\sqrt{\kappa}$, where $c$ is a positive constant, increases the variance amplification in the iterates by a factor of $\sqrt{\kappa}$.\/} The robustness problem was also studied in~\cite{yuayinsay16} where the authors show a similar behavior of Nesterov's method and gradient descent in an asymptotic regime in which the stepsize goes to zero. In contrast, we focus on the non-asymptotic stepsize regime and establish fundamental differences between gradient descent and its accelerated variants in terms of noise amplification. 

\tc{black}{More recently, the problem of finding upper bounds on the variance amplification was cast as a semidefinite program~\cite{micschebe19}. This formulation provided numerical results that are consistent with our theoretical upper bounds in terms of the condition number. In~\cite{micschebe19}, structured objective functions (e.g., diagonal Hessians) that arise in distributed optimization were also studied and the problem of designing robust algorithms were formulated as a bilinear matrix inequality (which, in general, is not convex).}

\subsubsection*{Contributions}

The effect of imperfections on the performance and robustness of first-order algorithms has been studied in~\cite{moubac11,devglines14} but the influence of acceleration on stochastic gradient perturbations has not been precisely characterized. We employ  \tc{black}{control-theoretic tools suitable for analyzing} stochastic dynamical systems to quantify such influence and \tc{black}{identify} fundamental trade-offs between acceleration and noise amplification. The main contributions of this paper are:
\begin{enumerate}
	\item 
	 We start our analysis by examining strongly convex quadratic optimization problems for which we can explicitly characterize variance amplification of first-order algorithms and obtain analytical insight. In contrast to convergence rates, which solely depend on the extreme eigenvalues of the Hessian matrix, we demonstrate that the {\em variance amplification is influenced by the entire spectrum.}
	
	\item We establish the relation between the noise amplification of accelerated algorithms and gradient descent for parameters that provide the optimal convergence rate for strongly convex quadratic problems. We also explain how the distribution of the eigenvalues of the Hessian influences these relations and provide examples to show that {\em acceleration can significantly increase \tc{black}{ the noise amplification}.} 
	
	\item We address the problem of tuning the algorithmic parameters and demonstrate the existence of a fundamental trade-off between the rate of convergence and noise amplification: for problems with condition number $\kappa$ and bounded dimension $n$, we show that any choice of parameters in accelerated methods that yields the linear convergence rate of at least $1 - {c}/\sqrt{\kappa}$, where $c$ is a positive constant, {\em increases noise amplification in the iterates relative to gradient descent\/} by a factor of at least~$\sqrt{\kappa}$.  
	
	\item We extend our analysis from quadratic objective functions to general strongly convex problems. We borrow an approach based on linear matrix inequalities from control theory to establish upper bounds on the noise amplification of both gradient descent and Nesterov's accelerated algorithm. Furthermore, for any given condition number, we demonstrate that {\em these bounds are tight up to constant factors.}
	
	\item We apply our results to distributed averaging over large-scale undirected networks. We examine the role of network size and topology on noise amplification and further illustrate \tc{black}{the subtle influence} of the entire spectrum of the Hessian matrix on the robustness of noisy optimization algorithms. In particular, {\em we identify a class of large-scale problems for which accelerated Nesterov's method achieves the same order-wise noise amplification\/} (in terms of condition number) {\em as gradient descent.} 	
\end{enumerate}

\subsubsection*{Paper structure}
The rest of our presentation is organized as follows. In Section~\ref{sec.PerilimBack}, we formulate the problem and provide background material. In Section~\ref{sec.Quadratic}, we explicitly evaluate the variance amplification (in terms of the algorithmic parameters and problem data) for strongly convex quadratic problems, derive lower and upper bounds, and provide a comparison between the accelerated methods and gradient descent. In Section~\ref{sec.general}, we extend our analysis to general strongly convex problems. In Section~\ref{sec.compTuning}, we establish fundamental trade-offs between the rate of convergence and noise amplification. In Section~\ref{sec:distributed}, we apply our results to the problem of distributed averaging over noisy undirected networks. We highlight the subtle influence of the distribution of the eigenvalues of the Laplacian matrix on variance amplification and discuss the roles of network size and topology. We provide concluding remarks in Section~\ref{sec: conclusion} and technical details in appendices. 
		
	\vspace*{-2ex}
\section{Preliminaries and background}
	\label{sec.PerilimBack}

In this paper, we quantify the effect of stochastic uncertainties in gradient evaluation on the performance of first-order algorithms for unconstrained optimization problems
\begin{align}\label{eq.f}
\minimize_{x}\quad f(x)
\end{align} 
where $f$: $\R^n\rightarrow\R$ is strongly convex with Lipschitz continuous gradient $\nabla f$. Specifically, we examine how gradient descent, 
\begin{subequations}
	\label{eq.1st}
\be
x^{t+1}
\; = \;
x^{t}
\; - \;
\tc{black}{\alpha} \nabla f(x^t)
\; + \;
\tc{black}{\sigma w^t}
\label{alg.GD}
\ee
Polyak's heavy-ball method,
\be
x^{t+2}
\; = \;
x^{t+1}
\; + \;
\tc{black}{\beta}
(
x^{t+1}
\, - \,
x^{t}
)
\; - \;
\tc{black}{\alpha }
\nabla f( x^{t+1} )
\; + \;
\tc{black}{\sigma w^t}
\label{alg.HB}
\ee	
and Nesterov's accelerated  method,
\be
\ba{rcl}
x^{t+2} 
& \!\!\! = \!\!\! &
x^{t+1}
\; + \;
\tc{black}{\beta}
(
x^{t+1}
\, - \,  
x^{t}
)
\, - \,
\tc{black}{\alpha }
\nabla 
f
\!
\left(
x^{t+1}
\, + \,
\tc{black}{\beta}
(
x^{t+1}
\, - \,  
x^{t})
\right)	
\; + \;
\tc{black}{\sigma w^t}
\ea
\label{alg.NA}
\ee
\end{subequations}
amplify the additive white stochastic noise $w^t$ with zero mean and identity covariance matrix, $
\EX 
\left[
w^t
\right]
= 
0,
$
$	
\EX 
\left[
w^t
(w^{\tau})^T
\right]
= 
I\,\delta (t - \tau)
$.
Here, $t$ is the iteration index, $x^t$ is the optimization variable, \tc{black}{ $ \alpha $ is the stepsize, $ \beta $ is an extrapolation parameter} used for acceleration, \tc{black}{$\sigma$ is the noise magnitude,} $\delta$ is the Kronecker delta, and $\EX$ is the expected value. \tc{black}{When the only source of uncertainty is a noisy gradient, we set $\sigma=\alpha$ in~\eqref{eq.1st}.}

The set of functions $f$ that are $\mf$-strongly convex and $\Lf$-smooth is denoted by $\mathcal{F}_{\mf}^{\Lf}$; $f \in \mathcal{F}_{\mf}^{\Lf}$ means that $f(x)-\frac{\mf}{2}\norm{x}^2$ is convex and that the gradient $\nabla f$ is $\Lf$-Lipschitz continuous. In particular, for a twice continuously differentiable function $f$ with the Hessian matrix $\nabla^2f$, we have
\begin{equation*}
f
\, \in \,
\mathcal{F}_{\mf}^{\Lf}
~\Leftrightarrow~
\mf I
\;\preceq\;
\nabla^2 f(x)
\;\preceq\;
\Lf I,
\quad \forall \,
x \, \in \, \R^n.
\end{equation*} 
In the absence of noise \tc{black}{(i.e., for $\sigma=0$)}, for $f\in\mathcal{F}_{\mf}^{\Lf}$, the parameters \tc{black}{$\alpha$ and $\beta$} can be selected such that gradient descent and Nesterov's accelerated method converge to the global minimum $x^\star$ of~\eqref{eq.f} with a linear rate $\rho<1$, i.e., 
	\[
	\norm{x^t \, - \, x^\star}
	\le
	c\, \rho^t \, \norm{x^0 \, - \, x^\star}
	\]
for all $t$ and some $c>0$. Table~\ref{tab:ratesGeneral} provides the conventional values of these parameters and the corresponding guaranteed convergence rates~\cite{nes18book}. Nesterov's method with the parameters provided in Table~\ref{tab:ratesGeneral} enjoys \tc{black}{the convergence rate $\rho_{\na}= \sqrt{1-{1}/{\sqrt{\kappa}}}\le 1 - 1/{(2\sqrt{\kappa})}$, where $\kappa \DefinedAs \Lf/\mf$ is the condition number associated with $\mathcal{F}_{\mf}^{\Lf}$. This rate  is {\em orderwise optimal\/} in the sense that no first-order algorithm can optimize all $f\in\mathcal{F}_{\mf}^{\Lf}$ with the rate $\rho_{\mathrm{lb}}=(\sqrt{\kappa} - 1)/(\sqrt{\kappa} + 1)$~\cite[Theorem~2.1.13]{nes18book}. Note that $1-\rho_{\mathrm{lb}}= O(1/\sqrt{\kappa})$ and $1-\rho_{\na}=\Omega(1/\sqrt{\kappa})$.} In contrast to Nesterov's method, the heavy-ball method does not offer any acceleration guarantees for all $f\in\mathcal{F}_{\mf}^{\Lf}$. However, for strongly convex quadratic $f$, parameters can be selected to guarantee linear convergence of the heavy-ball method with a rate that outperforms the one achieved by Nesterov's method~\cite{lesrecpac16}; see Table~\ref{tab:rates}.

\begin{table}
	\centering
	\begin{tabular}{ |l|l|l| } 
		\hline 
		& &
		\\[-.30cm]
		Method & Parameters & \tc{black}{Linear rate }
		\\
		\hline
		\hline
		& &
		\\[-.25cm]
		Gradient 
		& 
		$\alpha \, = \, \tfrac{1}{\Lf}$
		& 
		\tc{black}{$\rho \, = \sqrt{\, 1 \, - \, \tfrac{2}{\kappa+1}}$  } 
		\\[0.1cm] 
		Nesterov 
		& 
		$ \alpha \, = \, \tfrac{1}{\Lf}$,
		$\beta \, = \, \tfrac{\sqrt{\kappa} \, - \, 1}{\sqrt{\kappa} \, + \, 1} $ 
		& 
		\tc{black}{$\rho \, = \, \sqrt{1-\tfrac{1}{\sqrt{\kappa}}}   $}
		\\[-.35cm]
		& 
		& 
		\\
		\hline
	\end{tabular}		
	\caption{Conventional values of parameters and the corresponding \tc{black}{ rates} for $f\in\mathcal{F}_{\mf}^{\Lf}$, 
		$
		\norm{x^t - x^\star}
		\le
		c\,\rho^t\,\norm{x^0 - x^\star},
		$
where $\kappa\DefinedAs \Lf/\mf$  and $c>0$ is a constant\tc{black}{~\cite[Theorems~2.1.15,~2.2.1]{nes18book}}. The heavy-ball method does not offer acceleration guarantees for all $f\in\mathcal{F}_{\mf}^{\Lf}$.}
	\label{tab:ratesGeneral}
\end{table} 

To provide a quantitative characterization for the robustness of algorithms~\eqref{eq.1st}  to the noise $w^t$, we examine the performance measure,
	\be
	J
	\; \DefinedAs \;
	\limsup_{t \, \to \, \infty} 
	\, 
	\dfrac{1}{t}
	\sum_{k \, = \, 0}^{t}
	\EX
	\left( \norm{x^k \, - \, x^\star}^2 \right).
	\label{eq.Jnew}
	\ee
For quadratic objective functions, algorithms~\eqref{eq.1st} are linear dynamical systems. In this case, $J$ quantifies the steady-state variance amplification and it can be computed from the solution of the algebraic Lyapunov equation; see Section~\ref{sec.Quadratic}. For general strongly convex problems, there is no explicit characterization for $J$ but techniques from control theory can be utilized to compute an upper bound; see Section~\ref{sec.general}.

\subsubsection*{Notation}
	We write $g = \Omega(h)$ (or, equivalently, $h = O(g)$) to denote the existence of positive constants $c_i$ such that,  for any $x > c_2$, the functions $g$ and $h$:~$\R\rightarrow\R$ satisfy $g(x) \geq c_1 h(x)$. We write $g = \Theta(h)$, or more informally $g\approx h$, if both $g = \Omega(h)$ and $g = O(h)$. 
		
		\vspace*{-2ex}
	\section{Strongly convex quadratic problems}
	\label{sec.Quadratic}

Consider a strongly convex quadratic objective function,
	\begin{align}
		f(x)  
		\; = \;
		\tfrac{1}{2}
		\,
		x^T Q \, x
		\; - \; 
		q^T x
		\label{eq.quadraticObjective}
	\end{align}
where $ Q $ is a symmetric positive definite matrix and $q$ is a vector. Let $f\in\mathcal{F}_{\mf}^{\Lf}$ and let the eigenvalues $\lambda_i$ of $ Q $ satisfy 
\[
\Lf 
\; = \;
\lambda_1 
\; \geq \; 
\lambda_2 
\; \geq \; 
\ldots 
\; \geq \;
\lambda_n 
\; = \; 
\mf
\; > \; 
0.
\] 
In the absence of noise, the constant values of parameters $\alpha$ and $\beta$ provided in Table~\ref{tab:rates} yield linear convergence (with optimal decay rates) to the globally optimal point $x^\star = Q^{-1} q$ for all three algorithms~\cite{lesrecpac16}. In the presence of additive white noise $w^t$, we derive analytical expressions for  the variance amplification $J$ of algorithms~\eqref{eq.1st} and demonstrate that $J$ depends not only on the algorithmic parameters $\alpha$ and $\beta$ but also on all eigenvalues of the Hessian matrix $Q$. This should be compared and contrasted to the optimal rate of linear convergence which only depends on $ \kappa \DefinedAs \Lf / \mf$, i.e., the ratio of the largest and smallest eigenvalues of $Q$.
	 		
For constant $\alpha$ and $\beta$, algorithms~\eqref{eq.1st} can be described by a linear time-invariant (LTI) first-order recursion
	\be
	\ba{rcl}
	\psi^{t+1} 
	& \!\!\! = \!\!\! &
	A \,\psi^{t} 
	\; + \; 
	\tc{black}{\sigma} B \, w^t
	\\[0.cm]
	z^t
	& \!\!\! = \!\!\! &
	C \, \psi^t
	\label{eq.ss}
	\ea
	\ee
where $\psi^t$ is the state, $z^t \DefinedAs x^t-x^\star$ is the performance output, and $w^t$ is a white stochastic input. In particular, choosing 
	$
	\psi^t 
	\DefinedAs 
	x^t-x^\star
	$
for gradient descent and
	$
	\tc{black}{\psi^t}
	\DefinedAs 
	[ \, (x^t-x^\star)^T \; (x^{t+1}-x^\star)^T \, ]^T
	$	
for accelerated algorithms yields state-space model~\eqref{eq.ss} with
	\begin{align*}
	A 
	\; = \; 
	I \, - \, \alpha \, Q\,, 
	~~
	B 
	\; = \;
	C
	\;=\;
	I
	\end{align*}
for gradient descent and
\[
A
\; = \;  
\tbt{0}{I}{-\beta I}{(1+\beta)I-\alpha Q},
~~
A
\; = \;
\tbt{0}{I}{-\beta(I-\alpha Q)}{(1+\beta)(I-\alpha Q)}
\]
for the heavy-ball  and Nesterov's methods, respectively, with 
	\begin{align*}
	B^T
	\; = \;
	\obt{0}{ I},
	~~
	C
	\; = \; 
	\obt{I}{0}.
	\end{align*}
Since $w^t$ is zero mean, we have $ \EX \left( \psi^{t+1} \right) = A  \EX \left( \psi^t \right)$. Thus, $\EX \left( \psi^t \right) = A^t  \EX \left( \psi^0 \right)$ and, for any stabilizing parameters $\alpha$ and $\beta$, $\lim_{t \, \to \, \infty} \EX \left( \psi^t \right) = 0$, with the same linear rate as in the absence of noise. Furthermore, it is well-known that the covariance matrix 
	$
	P^t
	\DefinedAs  
	\EX 
	\left(
	\psi^t
	(\psi^t)^T
	\right)
	$
of the state vector satisfies the linear recursion
	\begin{subequations}
	\be
	P^{t+1} 
	\; = \; 
	A \, P^t A^T 
	\; + \; 
	\tc{black}{\sigma^2}B B^T
	\label{eq.LyapPt}
	\ee
and that its steady-state limit 
	\be
	P
	\; \DefinedAs \; 
	\lim_{t \, \to \, \infty} 
	\,
	\EX 
	\left(
	\psi^t
	(\psi^t)^T
	\right)
	\label{eq.P}
	\ee 
is the unique solution to the algebraic Lyapunov equation~\cite{kwasiv72}
\be
P
\; = \; 
A \, P A^T 
\; + \; 
\tc{black}{\sigma^2}B B^T.
\label{eq.Lyap}
\ee
For stable LTI systems, performance measure~\eqref{eq.Jnew} simplifies to the steady-state variance of the error in the optimization variable $z^t \DefinedAs x^t - x^\star$,
	\be
	J
	\; = \;
	\lim_{t \, \to \, \infty} 
	\, 
	\dfrac{1}{t}
	\sum_{k \, = \, 0}^{t}
	\EX
	\left( \norm{z^k}^2 \right)
	\; = \; 
	\lim_{t \, \to \, \infty} \EX
	\left( \norm{z^t}^2 \right)
	\label{eq.Jus}
	\ee
and it can be computed using either of the following two equivalent expressions
	\be
	J
	\; = \;
	\lim_{t \, \to \, \infty} 
	\, 
	\dfrac{1}{t}
	\sum_{k \, = \, 0}^{t}
	\trace
	\left( Z^k \right)
	\; = \;
	\trace
	\left( Z \right)
	\label{eq.Jus-P}
	\ee
\end{subequations} 	
where 
	$
	Z = C P C^T	
	$	
is the steady-state limit of the output covariance matrix 
	$
	Z^t
	\DefinedAs  
	\EX 
	\left(
	z^t
	(z^t)^T
	\right)
	=
	C P^t C^T.
	$

We next provide analytical solution $P$ to~\eqref{eq.Lyap} that depends on the parameters $\alpha$ and $\beta$ as well as on the spectrum of the Hessian matrix $Q$. This allows us to explicitly characterize the variance amplification~$J$ and quantify the impact of additive white noise on the performance of first-order optimization algorithms. 

		\begin{table}
		\centering
			\begin{tabular}{ |{l}|{l}|{l}| } 
				\hline 
				& &
				\\[-.30cm]
				Method & Optimal parameters & Rate of linear convergence
				\\
				\hline
				\hline
				& &
				\\[-.25cm]
				Gradient 
				& 
				$\alpha = \dfrac{2}{\Lf+\mf}$ 
				& 
				$\rho \, =\,\dfrac{\kappa-1}{\kappa +1}$  
				\\[0.25cm] 
				Nesterov 
				& 
				$ \alpha = \dfrac{4}{3\Lf+\mf}$,
				$\beta = \dfrac{\sqrt{3\kappa+1}-2}{\sqrt{3\kappa+1}+2}$ 
				& 
				$\rho \, =\,\dfrac{\sqrt{3\kappa+1} - 2}{\sqrt{3\kappa+1}}  \!\!\!$ 
				\\[0.25cm] 
				Heavy-ball
				& 
				$ \alpha = \dfrac{4}{(\sqrt{\Lf}+\sqrt{\mf})^2},\;\beta = \dfrac{( \sqrt{\kappa} - 1 )^2}{ ( \sqrt{\kappa} + 1 )^2} $ & 
				$ \rho \, =\,\dfrac{\sqrt{\kappa}-1}{\sqrt{\kappa} + 1}$ 
				\\[-.35cm]
				& 
				& 
				\\
				\hline
			\end{tabular}		
			\caption{Optimal parameters and the corresponding convergence rates for a strongly convex quadratic objective function $f\in\mathcal{F}_{\mf}^{\Lf}$ with $\lambda_{\max}(\nabla^2f)=\Lf$ and $\lambda_{\min}(\nabla^2f)=\mf$, and $ \kappa \DefinedAs \Lf / \mf$\tc{black}{~\cite[Proposition~1]{lesrecpac16}.}}
			\label{tab:rates}
		\end{table}
	
	\vspace*{-2ex}		
\subsection{Influence of the eigenvalues of the Hessian matrix}
	\label{subsec.Influence}

	We use the modal decomposition of the symmetric matrix $Q = V \Lambda V^T$ to bring $A$, $B$, and $C$ in~\eqref{eq.ss} into a block diagonal form,
	$
	\hat{A}
	= 
	\diag 
	\,
	(
	\hat{A}_i
	),
	$
	$
	\hat{B}
	= 
	\diag 
	\,
	(
	\hat{B}_i
	),
	$
	$
	\hat{C}
	= 
	\diag 
	\,
	(
	\hat{C}_i
	),
	$
with
	$
	i = 1, \ldots, n.
	$ Here, $\Lambda=\diag \, (\lambda_i) $ is the diagonal matrix of the eigenvalues and $V$ is the orthogonal matrix of the eigenvectors of $Q$. More specifically, the unitary coordinate transformation
	\be
	\hat{x}^t
	\,\DefinedAs\,
	V^T x^t,
	~~
	\hat{x}^\star
	\,\DefinedAs\,
	V^T x^\star,
	~~
	\hat{w}^t 
	\,\DefinedAs\,
	V^T w^t
	\label{eq:coorTrans}
	\ee
	brings the state-space model of  gradient descent into a  diagonal form with
	\begin{subequations}
	\label{eq.Ahat}
	\be
	\hat{\psi}_i^t
	\, = \,
	\hat{x}_i^t \,-\, \hat{x}_i^\star,
	~~
	\hat{A}_i
	\, = \,
	1 \, - \, \alpha \lambda_i,
	~~
	\hat{B}_i
	\,=\,
	\hat{C}_i
	\, = \,
	1.
	\label{eq.Ahat-gd}
	\ee
Similarly, for Polyak's heavy-ball and Nesterov's accelerated methods, change of coordinates~\eqref{eq:coorTrans} in conjunction with a permutation of variables, 
	$
	\tc{black}{\hat{\psi}_i^t}
	= 
	[ \, \hat{x}_i^t - \hat{x}_i^\star ~~ \hat{x}_i^{t+1} - \hat{x}_i^\star  \,]^T,
	$	
respectively yield
	\begin{align}
	\hat{A}_i
	\; = \;
	\tbt{0}{1}{-\beta}{1+\beta-\alpha \lambda_i},
	~~
	&
	\hat{B}_i
	\; = \,
	\left[
	\ba{c} {0} \\[-0.1cm] {1} \ea
	\right],
	~~	
	\hat{C}_i
	\; = \; 
	\obt{1}{0}
	\\[0.15cm]
	\hat{A}_i
	\; = \;
	\tbt{0}{1}{-\beta(1-\alpha \lambda_i)}{(1+\beta)(1-\alpha \lambda_i)},
	~~
	&
	\hat{B}_i
	\; = \,
	\left[
	\ba{c} {0} \\[-0.1cm] {1} \ea
	\right],
	~~	
	\hat{C}_i
	\; = \; 
	\obt{1}{0}.
	\end{align}
\end{subequations}	
  	This block diagonal structure allows us to explicitly solve Lyapunov equation~\eqref{eq.Lyap} for $ P $ and derive an analytical expression for $J$ in terms of the eigenvalues $\lambda_i$ of the Hessian matrix $Q$ and the algorithmic parameters $\alpha$ and $ \beta$. Namely, under coordinate transformation~\eqref{eq:coorTrans} and a suitable permutation of variables, equation~\eqref{eq.Lyap} can be brought into an equivalent set of equations,
	\be
	\hat{P}_i
	\; = \; 
	\hat{A}_i \, \hat{P}_i \, \hat{A}_i^T 
	\; + \; 
	\tc{black}{\sigma^2}\hat{B}_i \hat{B}_i^T,
	\quad 
	i \, = \, 1, \ldots, n
	\label{eq.LyapDiag}
	\ee
	where $ \hat{P}_i $ is a scalar for the gradient descent method and a $ 2\times 2 $ matrix for the accelerated algorithms. In Theorem~\ref{th.varianceJhat}, we use the solution to these decoupled Lyapunov equations to express the variance amplification~as
	\be
	J
	\; = \,
	\sum_{i \, = \, 1}^{n}
	\,
	\hat{J} ( \lambda_i )
	\; \DefinedAs \;
	\sum_{i \, = \, 1}^{n}
	\,
	\trace \, (\hat{C}_i \hat{P}_i \hat{C}_i^T )
	\non
	\ee
where $\hat{J} ( \lambda_i )$ determines the contribution of the eigenvalue $\lambda_i$ of the matrix $Q$ to the variance amplification. 
In what follows, we use subscripts $\gd$, $\hb$, and $\na$ (e.g., $J_\gd$, $J_\hb$, and $J_\na$) to denote quantities that correspond to gradient descent~\eqref{alg.GD}, heavy-ball method~\eqref{alg.HB}, and Nesterov's accelerated method~\eqref{alg.NA}.
	
	\begin{mythm}
		\label{th.varianceJhat}
		For strongly convex quadratic problems, the variance amplification of noisy first-order algorithms~\eqref{eq.1st} with any constant stabilizing parameters $\alpha$ and $\beta$ is determined by $J = \sum_{i \, = \, 1}^n \hat{J} ( \lambda_i )$, where $\lambda_i$ is the $i$th eigenvalue of $Q = Q^T \succ 0$ and the modal contribution to the variance amplification~$\hat{J}(\lambda)$ is given by
		\be
		\ba{rl}
		\mbox{Gradient:}
		&
		\hat{J}_\gd ( \lambda )
		\; = \;		
		\dfrac{\tc{black}{\sigma^2}}
		{\alpha \lambda \left(2 \, - \, \alpha \lambda \right)}
		\\[0.2cm]
		\mbox{Polyak:} 
		&
		\hat{J}_\hb(\lambda) 
		\; = \;
		\dfrac{\tc{black}{\sigma^2}(1 \, + \, \beta)}
		{\alpha \lambda\left(1\, - \, \beta\right) 
			\left(
			2(1\, + \, \beta) \, - \, \alpha\lambda
			\right)
		} 
		\\[0.25cm]
		\mbox{Nesterov:}
		&
		\hat{J}_\na ( \lambda )
		\; = \;
		\dfrac{\tc{black}{\sigma^2}(1 \, + \, \beta(1 \, - \, \alpha\lambda))}
		{
			\alpha \lambda 
			\left( 
			1 \, - \, \beta(1 \, - \, \alpha\lambda)
			\right)
			\left(
			2(1 \, + \, \beta) \, - \, (2\beta \, + \, 1)\alpha\lambda
			\right)}.
		\ea
		\non
		\ee		
	\end{mythm}
	\begin{proof}
		See Appendix~\ref{app.Quadratic}.
	\end{proof}
For strongly convex quadratic problems, Theorem~\ref{th.varianceJhat} provides {\em exact expressions\/} for variance amplification of the first-order algorithms. These expressions not only quantify the dependence of $J$ on the algorithmic parameters $\alpha$ and $\beta$ and the impact of the largest and smallest eigenvalues, but also capture the effect of all other eigenvalues of the Hessian matrix $Q$. \tc{black}{We also observe that the variance amplification $J$ is proportional to $\sigma^2$. Apart from Section~\ref{sec.compTuning}, where we examine the role of parameters $\alpha$ and $\beta$ on acceleration/robustness tradeoff and allow the dependence of $\sigma$ on $\alpha$, without loss of generality we choose $\sigma = 1$ in the rest of the paper.}

	\begin{myrem}
The performance measure $J$ in~\eqref{eq.Jus} quantifies the steady-state variance of the iterates of first-order algorithms. Robustness of noisy algorithms can be also evaluated using alternative performance measures, e.g., the mean value of the error in the objective function~\cite{aybfalgurozd19},  
	\be
J'
\;=\;
\lim_{t \, \to \, \infty} \EX
\left( ( x^t-x^\star )^T Q \, ( x^t-x^\star ) \right).
\label{eq.Jthem}
\ee
This measure of variance amplification can be characterized using our approach by defining
	$
	C = Q^{1/2}
	$
for gradient descent and
	$
	C = [ \, Q^{1/2} \;\, 0 \, ]
	$	
for accelerated algorithms in state-space model~\eqref{eq.ss}. Furthermore, repeating the above procedure for the modified performance output $z^t$ yields
	$
	J' 
	= 
	\sum_{i \, = \, 1}^n  
	\lambda_i \hat{J} ( \lambda_i ), 
	$
where the respective expressions for $\hat{J} ( \lambda_i )$ are given in Theorem~\ref{th.varianceJhat}.
\end{myrem}

	\vspace*{-2ex}
\subsection{Comparison for the parameters that optimize the convergence rate} 
	\label{subsec.ComparisionOptimal}
	
We next examine the robustness of first-order algorithms applied to strongly convex quadratic problems for the parameters that optimize the linear convergence rate; see Table~\ref{tab:rates}. For these parameters, the eigenvalues of the matrix $A$ are inside the open unit disk, implying exponential stability of system~\eqref{eq.ss}. We first use the expressions presented in Theorem~\ref{th.varianceJhat} to compare the variance amplification of the heavy-ball method to gradient descent. 

\begin{mythm}
	\label{thm.RelationJhbgd}
Let the strongly convex quadratic objective function $f$ in~\eqref{eq.quadraticObjective} satisfy $\lambda_{\max} (Q) = \Lf$, $\lambda_{\min} (Q) = \mf>0$, and let $\kappa \DefinedAs \Lf/\mf$ be the condition number. For the optimal parameters provided in Table~\ref{tab:rates}, the ratio between the variance amplification of the heavy-ball method and gradient descent \tc{black}{with equal values of $\sigma$} is given by 
	\begin{align}\label{eq.JHBtoJGD}		
	\dfrac{J_{\hb}}{J_{\gd}}
	\; = \; 
	\frac{(\sqrt{\kappa} \, + \, 1)^4}{8\sqrt{\kappa} \, (\kappa \, + \, 1)}.
	\end{align}
\end{mythm}
\begin{proof}
	For the parameters provided in Table~\ref{tab:rates} we have $\alpha_{\hb} = (1 + \beta) \alpha_{\gd}$, where $\beta = ( \sqrt{\kappa} - 1 )^2/( \sqrt{\kappa} + 1 )^2$ is the momentum parameter for the heavy-ball method. It is now straightforward to show that the modal contributions $\hat{J}_\hb$ and $\hat{J}_\gd$ to the variance amplification of the iterates given in Theorem~\ref{th.varianceJhat} satisfy
	\begin{align}\label{eq.PolToGD}		
	\dfrac{\hat{J}_{\hb}(\lambda)}{\hat{J}_{\gd}(\lambda)}
	\; = \; 
	\dfrac{1}{1 \, - \, \beta^2}
	\; = \; 
	\frac{(\sqrt{\kappa} \, + \, 1)^4}{8\sqrt{\kappa} \, (\kappa \, + \, 1)},
	\quad
	\forall \, \lambda \, \in \, [\mf,\Lf].
	\end{align}
	Thus, {\em the ratio $\hat{J}_{\hb}(\lambda)/\hat{J}_{\gd}(\lambda)$ does not depend on $\lambda$ and is only a function of the condition number $\kappa$.\/} Substitution of~\eqref{eq.PolToGD} into $J=\sum_{i} \hat{J}(\lambda_i)$ yields relation~\eqref{eq.JHBtoJGD}.
\end{proof}

Theorem~\ref{thm.RelationJhbgd} establishes the linear relation between the variance amplification of the heavy-ball algorithm $J_{\hb}$ and the gradient descent $J_{\gd}$. We observe that the ratio $J_{\hb}/J_{\gd}$ {\em only\/} depends on the condition number $\kappa$ and that {\em acceleration increases variance amplification\/}: for $\kappa \gg 1$, $J_{\hb}$ is larger than $J_{\gd}$  by a factor of $\sqrt{\kappa}$. We next study the ratio between the variance amplification of Nesterov's accelerated method and gradient descent. In contrast to the heavy-ball method, this ratio depends on the entire spectrum of the Hessian matrix $Q$. The following proposition, which examines the modal contributions $\hat{J}_{\na} (\lambda)$ and $\hat{J}_{\gd} (\lambda)$ of Nesterov's accelerated method and gradient descent, is the key technical result that allows us to establish the largest and smallest values that the ratio $J_\na/J_\gd$ can take for a given pair of extreme eigenvalues $\mf$ and $\Lf$ of $Q$ in Theorem~\ref{thm.RelationJnagd}. 
\begin{myprop}
	\label{prop.relationJhat}
	Let the strongly convex quadratic objective function $f$ in~\eqref{eq.quadraticObjective} satisfy $\lambda_{\max} (Q) = \Lf$, $\lambda_{\min} (Q) = \mf>0$, and let $\kappa \DefinedAs \Lf/\mf$ be the condition number. For the optimal parameters provided in Table~\ref{tab:rates}, the ratio $\hat{J}_{\na}(\lambda)/\hat{J}_{\gd}(\lambda)$ of modal contributions to variance amplification of Nesterov's method and gradient descent is a decreasing function of $\lambda\in[\mf,\Lf]$. Furthermore, \tc{black}{for $\sigma=1$}, the function $\hat{J}_\gd(\lambda)$ satisfies
	\begin{subequations}
	\be
	\ba{rclcl}
	\max\limits_{\lambda \, \in \, [\mf, \Lf]}
	\;
	\hat{J}_{\gd}(\lambda)
	& \!\!\! = \!\!\! &  
	\hat{J}_{\gd}(\mf) 
	& \!\!\! = \!\!\! &  
	\hat{J}_{\gd}(\Lf) 
	\; = \; 
	\dfrac{(\kappa \, + \, 1)^2}{4 \kappa}
	\\[0.25cm]
	\min\limits_{\lambda \, \in \, [\mf, \Lf]}
	\;
	\hat{J}_{\gd}(\lambda)
	& \!\!\! = \!\!\! &  
	\hat{J}_{\gd}({1}/{\alpha})  
	&\!\!\! = \!\!\! &  
	1
	\label{eq.GDExtrem}
	\ea
	\ee
	and the function $\hat{J}_\na(\lambda)$ satisfies
	\be
	\ba{rclcl}
	&&  
	\hat{J}_\na(\Lf) 
	& \!\!\! =\!\!\! &  
	\dfrac{9 \, \bk^2 \! \left(\bk \, + \, 2 \sqrt{\bk} \, -2 \, \right)}
	{32 \left(\bk \, - \, 1 \right) \! \left( \bk \, - \, \sqrt{\bk} \, + \, 1 \right) \! \left(2\sqrt{\bk} \, - \, 1\right)} 
	\\[0.35cm]
	\max\limits_{\lambda \, \in \, [\mf, \Lf]}
	\;
	\hat{J}_{\na}(\lambda)
	& \!\!\! = \!\!\! &  
	\hat{J}_{\na}(\mf)
	& \!\!\! = \!\!\! &  
	\dfrac{\bk^2 \! \left(\bk \, - \, 2 \sqrt{\bk} \, + \, 2 \right)}
	{32 \! \left(\sqrt{\bk} \, - \, 1 \right)^3}
	\\[0.15cm] 
	\min\limits_{\lambda \, \in \, [\mf, \Lf]} 
	\;
	\hat{J}_{\na}(\lambda)
	& \!\!\! = \!\!\! &  
	\hat{J}_{\na}({1}/{\alpha}) 
	& \!\!\! = \!\!\! &
	1
	\ea
	\label{eq.Jmax}
	\ee
	\end{subequations}
where $\bk \DefinedAs 3 \kappa + 1$.	
	\end{myprop}
	\begin{proof}
	See Appendix~\ref{app.Quadratic}.
	\end{proof}

For all three algorithms, Proposition~\ref{prop.relationJhat} and Theorem~\ref{thm.RelationJhbgd} demonstrate that the modal contribution to the variance amplification of the iterates at the extreme eigenvalues of the Hessian matrix $\mf$ and $\Lf$ only depends on the condition number $\kappa \DefinedAs \Lf/\mf$. For gradient descent and the heavy-ball method, $\hat{J}$ achieves its largest value at $\mf$ and $\Lf$, i.e., 
	\begin{subequations}
	\label{eq.Jhat_m_L}
	\be
	\ba{rcl}
	\max\limits_{\lambda \, \in \, [\mf, \Lf]}
	\hat{J}_\gd (\lambda) 
	& \!\!\! =  \!\!\! &
	\hat{J}_\gd(\mf) 
	\; = \;
	\hat{J}_\gd(\Lf) 
	\; = \; 
	\Theta(\kappa)
	\\[0.15cm]
	\max\limits_{\lambda \, \in \, [\mf, \Lf]}
 	\hat{J}_\hb (\lambda) 
	& \!\!\! =  \!\!\! &
	\hat{J}_\hb (\mf) 
	\; = \; \hat{J}_\hb (\Lf) 
	\; = \; 
	\Theta(\kappa \sqrt{\kappa}).
	\ea
	\ee
On the other hand, for Nesterov's method,~\eqref{eq.Jmax} implies a gap of $\Theta(\kappa)$ between the boundary values 
	\be
	\max\limits_{\lambda \, \in \, [\mf, \Lf]}
 	\hat{J}_\na (\lambda) 
	\; = \;
	\hat{J}_\na (\mf)
	\; = \; 
	\Theta(\kappa\sqrt{\kappa}),
	~~
	\hat{J}_\na(\Lf) 
	\; = \; 
	\Theta(\sqrt{\kappa}).
	\ee
	\end{subequations} 
	\vspace*{-2ex}
	\begin{myrem}
Theorem~\ref{th.varianceJhat} provides explicit formulas for variance amplification of noisy algorithms~\eqref{eq.1st} in terms of the  eigenvalues $\lambda_i$ of the Hessian matrix $Q$. Similarly, we can represent the variance amplification in terms of the eigenvalues $\hat{\lambda}_i$ of the dynamic matrices $\hat{A}_i$ in~\eqref{eq.Ahat}. For gradient descent, $\hat{\lambda}_i = 1 - \alpha \lambda_i$ and it is straightforward to verify that $J_{\gd}$ is determined by the sum of reciprocals of distances of these eigenvalues to the stability boundary, $J_{\gd} = \sum_{i \, = \, 1}^n \tc{black}{\sigma^2}/(1 - \hat{\lambda}_i^2)$. \tc{black}{Similarly, for accelerated methods we have,
	\[
	J
	\;=\;
	\sum_{i \, = \, 1}^n
	\dfrac{ \sigma^2 (1+\muo \mut)}{(1-\muo\mut) (1-\muo) (1-\mut)(1+\muo) (1+\mut)}
	\]
where $\muo$ and $\mut$ are the eigenvalues of $\hat{A}_i$. For Nesterov's method with the parameters provided in Table~\ref{tab:rates}, the matrix $\hat{A}_n$, which corresponds to $\lambda_n = m$, admits a Jordan canonical form with repeated eigenvalues $\hat{\lambda}_n = \hat{\lambda}'_n = 1 - 2/\sqrt{3 \kappa + 1}$. In this case, $\hat{J}_{\na} (m) = \sigma^2 (1 + \hat{\lambda}_n^2)/(1 - \hat{\lambda}_n^2)^3$, which should be compared and contrasted to the above expression for gradient descent. Furthermore,} for both  $\lambda_1 = L$ and $\lambda_n = m$, the matrices $\hat{A}_1$ and $\hat{A}_n$ for the heavy-ball method with the parameters provided in Table~\ref{tab:rates} have eigenvalues with algebraic multiplicity two and incomplete sets of eigenvectors.  
\end{myrem}	
	
We next establish the range of values that the ratio $J_\na/J_\gd$ can take. 
\begin{mythm}
	\label{thm.RelationJnagd}
	For the strongly convex quadratic objective function $f$ in~\eqref{eq.quadraticObjective} with $x \in \bbR^n$, $\lambda_{\max} (Q) = \Lf$, and $\lambda_{\min} (Q) = \mf>0$, the ratio between the variance amplification of Nesterov's accelerated method and gradient descent, for the optimal parameters provided in Table~\ref{tab:rates} \tc{black}{and equal values of $\sigma$} satisfies
	\begin{align}\label{eq.JNAtoJGD}
\dfrac{\hat{J}_\na(\mf)
	\, + \,
	(n\,-\,1)
	\hat{J}_\na(\Lf)}
{\hat{J}_\gd(\mf)
	\, + \,
	(n\,-\,1)
	\hat{J}_\gd(\Lf)}
\;\le\;
\dfrac{J_\na}{J_\gd}
\; \le \;
\dfrac{\hat{J}_\na(\Lf)
	\, + \,
	(n\,-\,1)
	\hat{J}_\na(\mf)}
{\hat{J}_\gd(\Lf)
	\, + \,
	(n\,-\,1)
	\hat{J}_\gd(\mf)}.
\end{align} 
\end{mythm}
\begin{proof}
	See Appendix~\ref{app.Quadratic}.
\end{proof}
Theorem~\ref{thm.RelationJnagd} provides tight upper and lower bounds on the ratio between $J_{\na}$ and $J_{\gd}$ for strongly convex quadratic problems. As shown in Appendix~\ref{app.Quadratic}, the lower bound is achieved for a quadratic function in which the Hessian matrix $Q$ has one eigenvalue at $\mf$ and $n-1$ eigenvalues at $\Lf$, and the upper bound is achieved when $Q$ has one eigenvalue at $\Lf$ and the remaining ones at $\mf$. Theorem~\ref{thm.RelationJnagd} in conjunction with Proposition~\ref{prop.relationJhat} demonstrate that {\em for a fixed problem dimension $n$, $J_{\na}$ is larger than $J_{\gd}$  by a factor of $\sqrt{\kappa}$ for $\kappa \gg 1$\/}. 
 
 This trade-off is further highlighted in Theorem~\ref{thm.JBoundsQuad} which provides tight bounds on the variance amplification of iterates in terms of the problem dimension $n$ and the condition number $\kappa$ for all three algorithms. To simplify the presentation, we first use the explicit expressions for $\hat{J}_\na(\mf)$ and $\hat{J}_\na(\Lf)$ in Proposition~\ref{prop.relationJhat} to obtain the following upper and lower bounds on $\hat{J}_\na(\mf)$ and $\hat{J}_\na(\Lf)$ (see Appendix~\ref{app.Quadratic})
 	\begin{align}\label{eq.boundsLambdaLm}
 	\dfrac{(3\kappa \, + \, 1)^{\tfrac{3}{2}}}{32}
 	\;\le\;
 	\hat{J}_\na(\mf) 
 	\;\le\;
 	\dfrac{(3\kappa \, + \, 1)^{\tfrac{3}{2}}}{8},
	\qquad
 	\dfrac{9\sqrt{3\kappa \, + \, 1} }{64}
 	\;\le\;
 	\hat{J}_\na(\Lf) 
 	\;\le\;
 	\dfrac{9\sqrt{3\kappa \, + \, 1} }{8}.
 	\end{align} 

\begin{mythm}
	\label{thm.JBoundsQuad}
	For the strongly convex quadratic objective function $f$ in~\eqref{eq.quadraticObjective} with $x \in \bbR^n$, $\lambda_{\max} (Q) = \Lf$, $\lambda_{\min} (Q) = \mf>0$, and $\kappa \DefinedAs \Lf/\mf$, the variance amplification of the first-order optimization algorithms, with the parameters provided in Table~\ref{tab:rates} \tc{black}{and $\sigma=1$}, is bounded by
	\begin{align*}
	\dfrac{(\kappa \, - \, 1)^2}{2 \kappa}
	\, + \, 
	n
	&\;\le\;
	J_\gd
	\;\le\;
	\dfrac{n (\kappa \, + \, 1)^2}{4 \kappa}
	\\[0.cm]
	\dfrac{(\sqrt{\kappa} \, + \, 1)^4}{8\sqrt{\kappa}(\kappa \, + \, 1)}
	\left(\dfrac{(\kappa \, - \, 1)^2}{2 \kappa}
	\, + \, 
	n\right)
	&
	\;\le\;
	J_\hb 
	\;\le\;
	\dfrac{n (\kappa \, + \, 1)(\sqrt{\kappa} \, + \, 1)^4}{32 \kappa\sqrt{\kappa}}
	\\[0.cm]
	\dfrac{(3 \kappa \, + \, 1)^{\tfrac{3}{2}}}{32}\, +\, \dfrac{9\sqrt{3 \kappa \, + \, 1} }{64}  \, + \, n \, - \, 2
	&\;\le\;
	J_\na
	\;\le\;
	\dfrac{(n-1)(3 \kappa \, + \, 1)^{\tfrac{3}{2}} }{8} \,+\, \dfrac{9\sqrt{3 \kappa \, + \, 1} }{8}.
	\end{align*} 
\end{mythm}
\begin{proof}
	As shown in Proposition~\ref{prop.relationJhat}, the functions $\hat{J}(\lambda)$ for gradient descent and Nesterov's algorithm attain their largest and smallest values over the interval $[\mf, \Lf]$ at $\lambda = \mf$ and $\lambda= {1}/{\alpha}$, respectively. Thus, fixing the smallest and largest eigenvalues, the variance amplification $J$ is maximized when the other $n-2$ eigenvalues are all equal to $\mf$ and is minimized when they are all equal to $1/\alpha$. This combined with the explicit expressions for $\hat{J}_\gd(\mf)$, $\hat{J}_\gd(\Lf)$, and $\hat{J}_\gd(1/\alpha)$ in~\eqref{eq.GDExtrem} leads to the tight upper and lower bounds for gradient descent. For the heavy-ball method, the bounds follow from Theorem~\ref{thm.RelationJhbgd} and for Nesterov's algorithm, the bounds follow from~\eqref{eq.boundsLambdaLm}. 	
	\end{proof}

For problems with a fixed dimension $n$ and a condition number $\kappa \gg n$, there is an $\Omega(\sqrt{\kappa})$ difference in both upper and lower bounds provided in Theorem~\ref{thm.JBoundsQuad} for the accelerated algorithms relative to gradient descent. Even though Theorem~\ref{thm.JBoundsQuad} considers only the values of $\alpha$ and $\beta$ that optimize the convergence rate, in Section~\ref{sec.compTuning} we demonstrate that this gap is fundamental in that it holds for any parameters that yield an accelerated convergence rate. It is worth noting that both the lower and upper bounds are influenced by the problem dimension $n$ and the condition number $\kappa$. For large-scale problems, there may be a subtle relation between $n$ and $\kappa$ and the established bounds may exhibit different scaling trends. In Section~\ref{sec:distributed}, we identify a class of quadratic optimization problems for which $J_\na$ scales in the same way as $J_\gd$ for $\kappa \gg 1$ and $n \gg 1$. 

Before we elaborate further on these issues, we provide two illustrative examples that highlight the importance of the choice of the performance metric in the robustness analysis of noisy algorithms. It is worth noting that an $O(\kappa)$ upper bound for gradient descent and an $O(\kappa^{2})$ upper bound for Nesterov's accelerated algorithm was established in~\cite{schroubac11}. Relative to this upper bound for Nesterov's method, the upper bound provided in Theorem~\ref{thm.JBoundsQuad} is tighter by a factor of $\sqrt{\kappa}$. Theorem~\ref{thm.JBoundsQuad} also provides lower bounds, reveals the influence of the problem dimension $n$, and identifies constants that multiply the leading terms in the condition number $\kappa$. Moreover, in Section~\ref{sec.general} we demonstrate that similar upper bounds can be obtained for general strongly convex objective functions with Lipschitz continuous gradients.
		
		\vspace*{-2ex}
	\subsection{Examples} We next provide illustrative examples to (i) demonstrate the agreement of our theoretical predictions with the results of stochastic simulations; and (ii) contrast two natural performance measures, namely the variance of the iterates $J$ in~\eqref{eq.Jus} and the mean objective error $J'$ in~\eqref{eq.Jthem}, for assessing robustness of noisy optimization algorithms.

\subsubsection*{Example~1} 

Let us consider the quadratic objective function in~\eqref{eq.quadraticObjective} with
	\be
	Q 
	\; = \;
	\tbt{\Lf}{0\\[-0.7cm]}{0}{\mf},
	~~
	q 
	\; = \;
	\tbo{0\\[-0.7cm]}{0}. 
	\label{eq.Qexample}
	\ee
For all three algorithms, the performance measures $J$ and $J'$ are given by
	\be
	\ba{rcl}
	J 
	& \!\!\! = \!\!\! &
	\hat{J} (\mf) \; + \; \hat{J} (\Lf)
	\\
	J' 
	& \!\!\! = \!\!\! &
	\mf \hat{J} (\mf) \; + \; \Lf \hat{J} (\Lf)
	\; = \;
	\Lf 
	\left(
	\tfrac{1}{\kappa} \, \hat{J} (\mf) \; + \; \hat{J} (\Lf)
	\right)
	\, = \;
	\mf 
	\left(
	\hat{J} (\mf) \; + \; \kappa \, \hat{J} (\Lf)
	\right).
	\ea
	\non
	\ee	
As shown in~\eqref{eq.Jhat_m_L}, $\hat{J} (\mf)$ and $\hat{J} (\Lf)$ only depend on the condition number $\kappa$ and the variance amplification of the iterates satisfies
	\begin{subequations}
	\label{eq.JJ'2x2}
	\be
	J_{\gd} 
	\; = \; 
	\Theta(\kappa),
	~~
	J_{\hb} 
	\; = \; 
	\Theta(\kappa \sqrt{\kappa}),
	~~
	J_{\na} 
	\; = \; 
	\Theta(\kappa \sqrt{\kappa}).
	\label{eq.J2x2}
	\ee
On the other hand, $J'$ also depends on $\mf$ and $\Lf$. In particular, it is easy to verify the following relations for two scenarios that yield $\kappa \gg 1$: 	
	\bi
	\item for $\mf \ll 1$ and $\Lf = O (1)$
	\ei	
	\be
	J'_{\gd} 
	\; = \; 
	\Theta(\kappa),
	~~
	J'_{\hb} 
	\; = \; 
	\Theta(\kappa \sqrt{\kappa}),
	~~
	J'_{\na} 
	\; = \; 
	\Theta(\sqrt{\kappa}).
	\label{eq.J'2x2a}
	\ee

	\bi
	\item for $\Lf \gg 1$ and $\mf = O (1)$	
	\ei
	\be
	J'_{\gd} 
	\; = \; 
	\Theta(\kappa^2),
	~~
	J'_{\hb} 
	\; = \; 
	\Theta(\kappa^2 \sqrt{\kappa}),
	~~
	J'_{\na} 
	\; = \; 
	\Theta(\kappa \sqrt{\kappa}).
	\label{eq.J'2x2b}
	\ee
	\end{subequations}	

Relation~\eqref{eq.J2x2} reveals the detrimental impact of acceleration on the variance of the optimization variable. On the other hand,~\eqref{eq.J'2x2a} and~\eqref{eq.J'2x2b} show that, relative to gradient descent, the heavy-ball method increases the mean error in the objective function while Nesterov's method reduces it. Thus, if the mean value of the error in the objective function is to be used to assess performance of noisy algorithms, one can conclude that Nesterov's method significantly outperforms gradient descent both in terms of convergence rate and robustness to noise. However, this performance metric fails to capture large variance of the mode associated with the smallest eigenvalue of the matrix $Q$ in Nesterov's algorithm. Theorem~\ref{thm.RelationJhbgd} and Proposition~\ref{prop.relationJhat} show that the modal contributions to the variance amplification of the iterates for gradient descent and the heavy-ball method are balanced at $\mf$ and $\Lf$, i.e., $\hat{J}_\gd(\mf) = \hat{J}_\gd(\Lf) = \Theta(\kappa)$ and $\hat{J}_\hb (\mf) = \hat{J}_\hb (\Lf) = \Theta(\kappa \sqrt{\kappa})$. On the other hand, for Nesterov's method there is a $\Theta(\kappa)$ gap between $\hat{J}_\na (\mf)=\Theta(\kappa\sqrt{\kappa})$ and $\hat{J}_\na(\Lf)=\Theta(\sqrt{\kappa})$. While the performance measure $J'$ reveals a superior performance of Nesterov's algorithm at large condition numbers, it fails to capture the negative impact of acceleration on the variance of the optimization variable; see Fig.~\ref{fig.ellips} for an illustration. 
		
	\begin{figure}[h]
		\centering
			\begin{tabular}{c}
			\begin{tikzpicture}
				\pgfmathsetmacro{\scl}{0.8}
				\pgfmathsetmacro{\x}{3.75*\scl}
				\pgfmathsetmacro{\y}{4*\scl}
	
	\pgfmathsetmacro{\h}{1.45*\scl}
	
	\pgfmathsetmacro{\r}{1*\scl}
	
	\pgfmathsetmacro{\k}{0.15}
	\pgfmathsetmacro{\kk}{0.3}
	\pgfmathsetmacro{\kkk}{1.3}
	
	\colorlet{hblue}{blue!60!white}
	\colorlet{hblack}{black!40!white}
	\colorlet{hred}{red!60!white}
	
	\draw[hblue, fill=hblue] (0, \y) ellipse (\kk*\r cm and \kk*\r cm) node [name=p11] {};
	\draw[hblack, fill=hblack] (\x, \y) ellipse (\kkk*\r cm and \kkk*\r cm) node [name=p12] {};
	\draw[hred, fill=hred] (2*\x, \y) ellipse (\k*\r cm and \kkk*\r cm) node [name=p13] {};
	\draw[hblue, fill=hblue] (0, 0) ellipse (\kk*\r cm and \k*\r cm) node [name=p21] {};
	\draw[hblack, fill=hblack](\x, 0) ellipse (\kkk*\r cm and \kk*\r cm) node [name=p22] {};
	\draw[hred, fill=hred] (2*\x,0) ellipse (\k*\r cm and \k*\r cm) node [name=p23] {};
	
	\draw[->,line width=0.3mm] (-\h, 0) -- (\h, 0) ;
	\draw[->,line width=0.3mm] (0, -\h) node[below](){Gradient descent} -- (0, \h);
	
	\draw[->,line width=0.3mm] (-\h + \x, 0) -- (\h + \x, 0);
	\draw[->,line width=0.3mm] (0 + \x, -\h) node[below](){Heavy-ball} -- (0 + \x, \h)node[above,yshift=0.cm](){Ellipsoids associated with the performance measure $J'$:};
	
	\draw[->,line width=0.3mm] (-\h + 2*\x, 0) -- (\h + 2*\x , 0);
	\draw[->,line width=0.3mm] (0 + 2*\x, -\h)node[below](){Nesterov} -- (0 + 2*\x, \h);
	
	\draw[->,line width=0.3mm] (-\h, 0 + \y ) -- (\h, 0 + \y);
	\draw[->,line width=0.3mm] (0, -\h + \y) -- (0, \h + \y);
	
	\draw[->,line width=0.3mm] (-\h + \x, 0 + \y) -- (\h + \x, 0 + \y);
	\draw[->,line width=0.3mm] (0 + \x, -\h + \y) --  (0 + \x, \h + \y)node[above,yshift=0.cm](){Ellipsoids associated with the performance measure $J$:};
	
	\draw[->,line width=0.3mm] (-\h + 2*\x, 0 + \y) -- (\h + 2*\x , 0 + \y);
				\draw[->,line width=0.3mm] (0 + 2*\x, -\h + \y) -- (0 + 2*\x, \h + \y);
	\end{tikzpicture}
\end{tabular}	
			\caption{Ellipsoids $\{z \, | \, z^T Z^{-1} z \le 1 \}$ associated with the steady-state covariance matrices $Z = C P C^T$ of the performance outputs $z^t = x^t - x^\star$ (top row) and $z^t = Q^{1/2} ( x^t - x^\star )$ (bottom row) for algorithms~\eqref{eq.1st} with the parameters provided in Table~\ref{tab:rates} for the matrix $Q$ given in~\eqref{eq.Qexample} with $\mf\ll\Lf=O(1)$. The horizontal and vertical axes show the eigenvectors $[\, 1 \;\, 0 \, ]^T$ and $[\, 0 \;\, 1 \, ]^T$ associated with the eigenvalues $\hat{J} (\Lf)$ and $\hat{J}(\mf)$ (top row) and $\hat{J}' (\Lf)$ and $\hat{J}' (\mf)$ (bottom row) of the respective output covariance matrices $Z$.}
	\label{fig.ellips}
\end{figure}
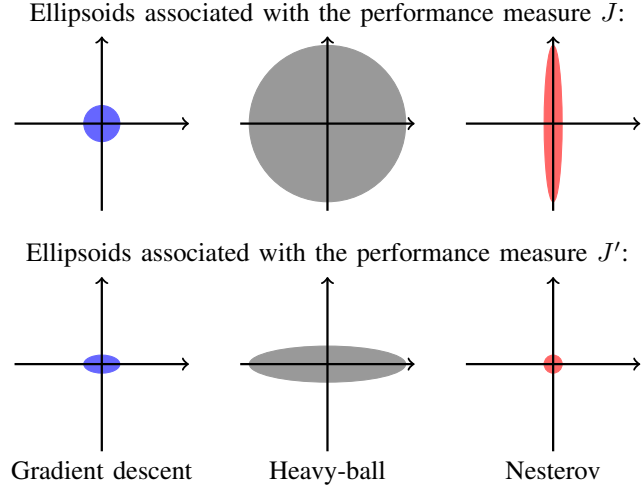
	
Figure~\ref{figScatter} shows the performance outputs $z^t=x^t$ and $z^t=Q^{1/2}x^t$ resulting from $10^5$ iterations of noisy first-order algorithms with the optimal parameters provided in Table~\ref{tab:rates} for the strongly convex objective function $f(x) = 0.5\, x_1^2 + 0.25\times10^{-4}\,x_2^2$ ($\kappa = 2\times 10^4$). Although Nesterov's method exhibits good performance with respect to the error in the objective function (performance measure $J'$), the plots in the first row illustrate detrimental impact of noise on both accelerated algorithms with respect to the variance of the iterates (performance measure $J$). In particular, we observe that: (i) for gradient descent and the heavy-ball method, the iterates $x^t$ are scattered uniformly along the eigen-directions of the Hessian matrix $Q$ and acceleration increases variance equally along all directions; and (ii) relative to gradient descent, Nesterov's method exhibits larger variance in the iterates $x^t$ along the direction that corresponds to the smallest eigenvalue $\lambda_{\min} (Q)$.

\subsubsection*{Example~2} Figure~\ref{fig.toeplitz-example} compares \tc{black}{the results} of twenty stochastic simulations for a strongly convex quadratic objective function~\eqref{eq.quadraticObjective} with $q = 0$ and a Toeplitz matrix $Q \in \bbR^{50 \times 50}$ with the first row $[ \, 2 \, - \! 1 ~\, 0 \; \cdots \; 0 \; \; 0 \, ]^T$. This figure shows the time-dependence of the variance of the performance outputs $z^t = x^t$ and $z^t = Q^{1/2} x^t$ for the algorithms subject to additive white noise with zero initial conditions. The plots further demonstrate that the mean error in the objective function does not capture detrimental impact of noise on the variance of the iterates for Nesterov's algorithm. The bottom row also compares variance obtained by averaging outcomes of twenty stochastic simulations with the corresponding theoretical values resulting from the Lyapunov equations. 

	\begin{figure*}[t!]
		\centering
		\begin{tabular}{r@{\hspace{-0.3 cm}}c@{\hspace{-0.3 cm}}c@{\hspace{-0.3 cm}}c}
			& & performance output $z^t=x^t$: &\\
			\hspace{-.0cm}
			\begin{tabular}{c}
				\rotatebox{90}{\hspace{0.75 cm} $z_2$ }
			\end{tabular}
			&
			\hspace{-.0cm}
			\begin{tabular}{c}
				\includegraphics[width=.24\textwidth]{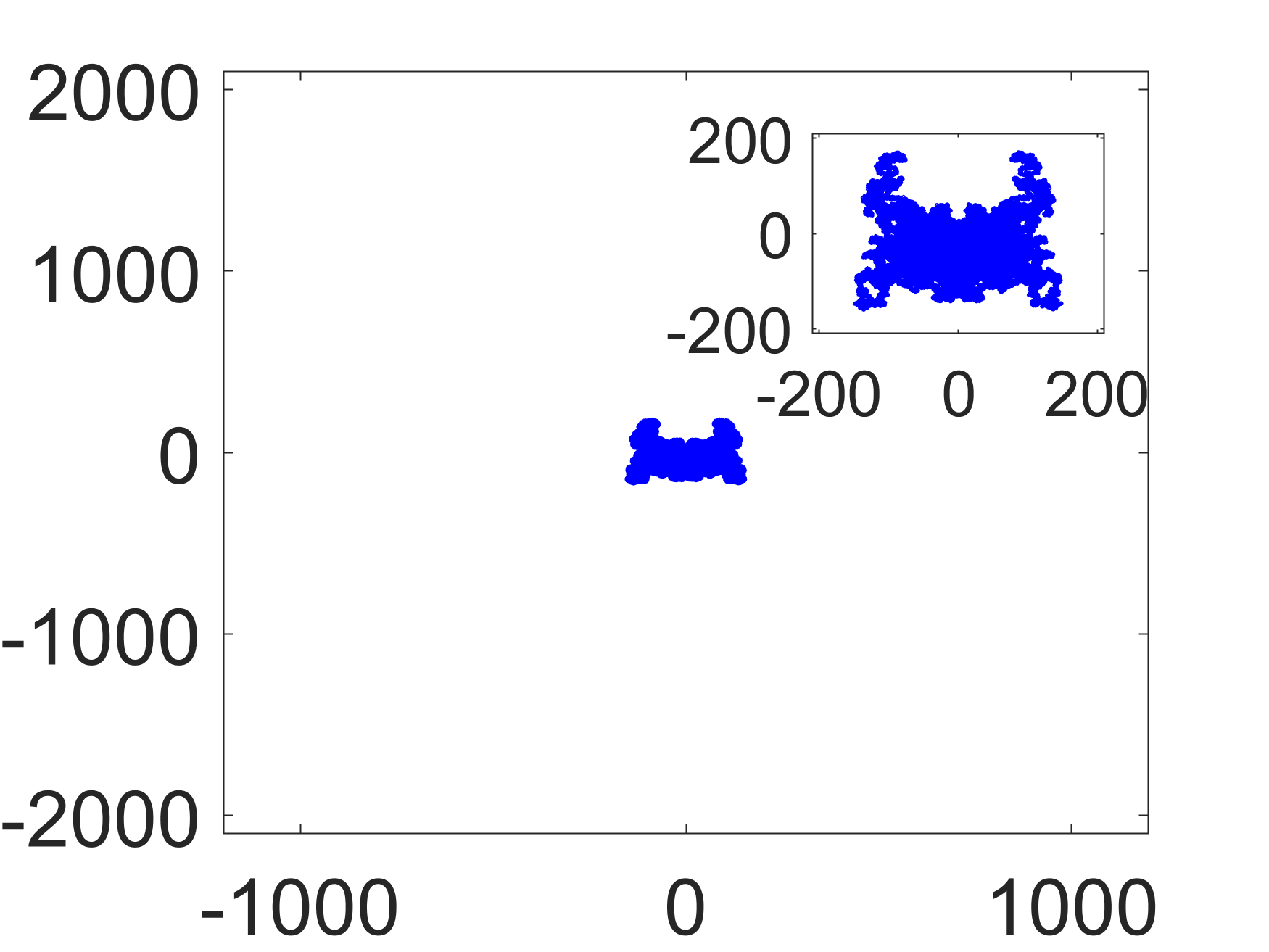}
				\\[-0cm]
				 {$z_1$}
			\end{tabular}
			&
			\hspace{0cm}
			\begin{tabular}{c}
				\includegraphics[width=.24\textwidth]{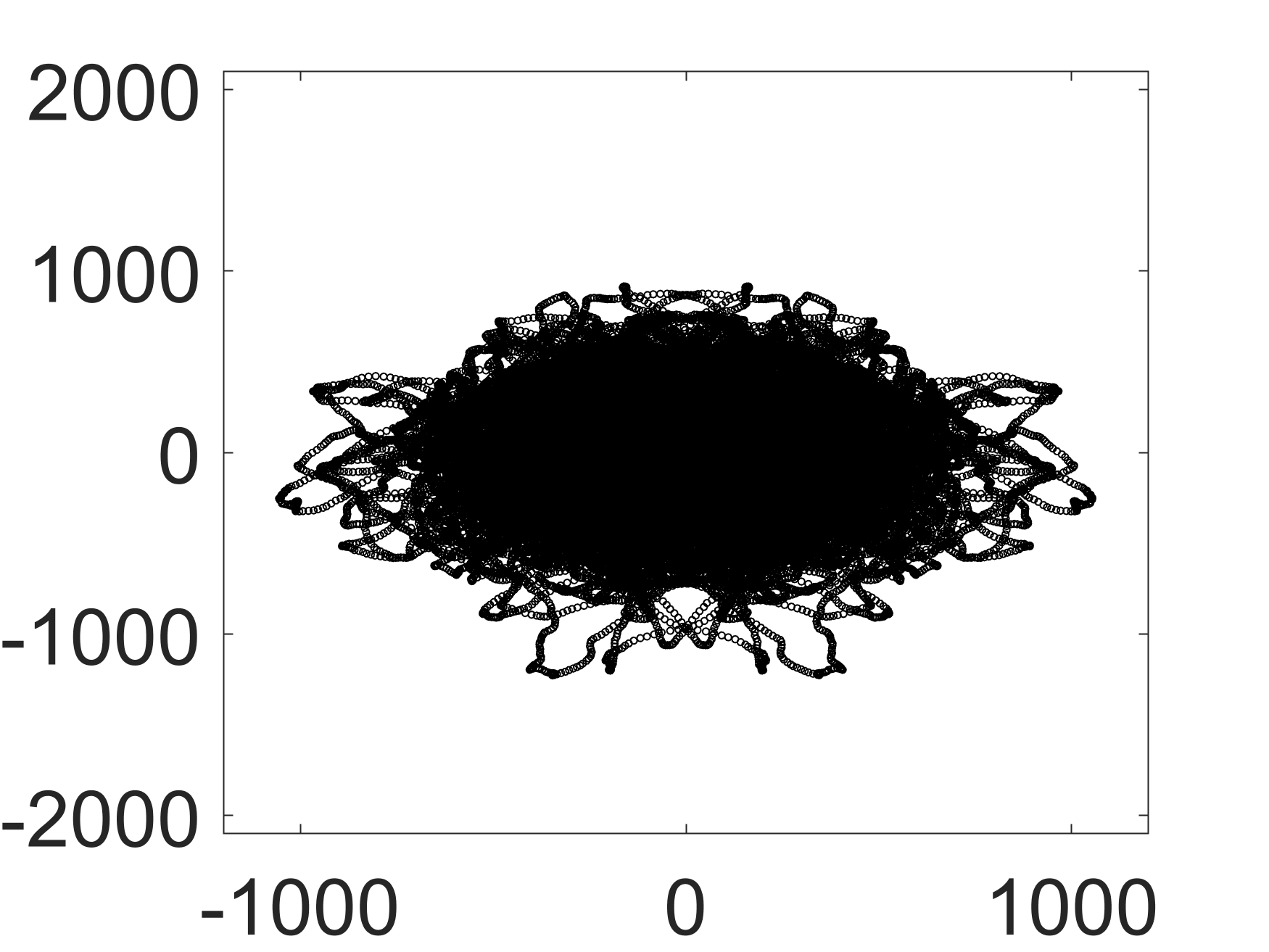}
				\\[-0cm]
				{$z_1$}
			\end{tabular}
			&
			\hspace{-0cm}
			\begin{tabular}{c}
				\includegraphics[width=.24\textwidth]{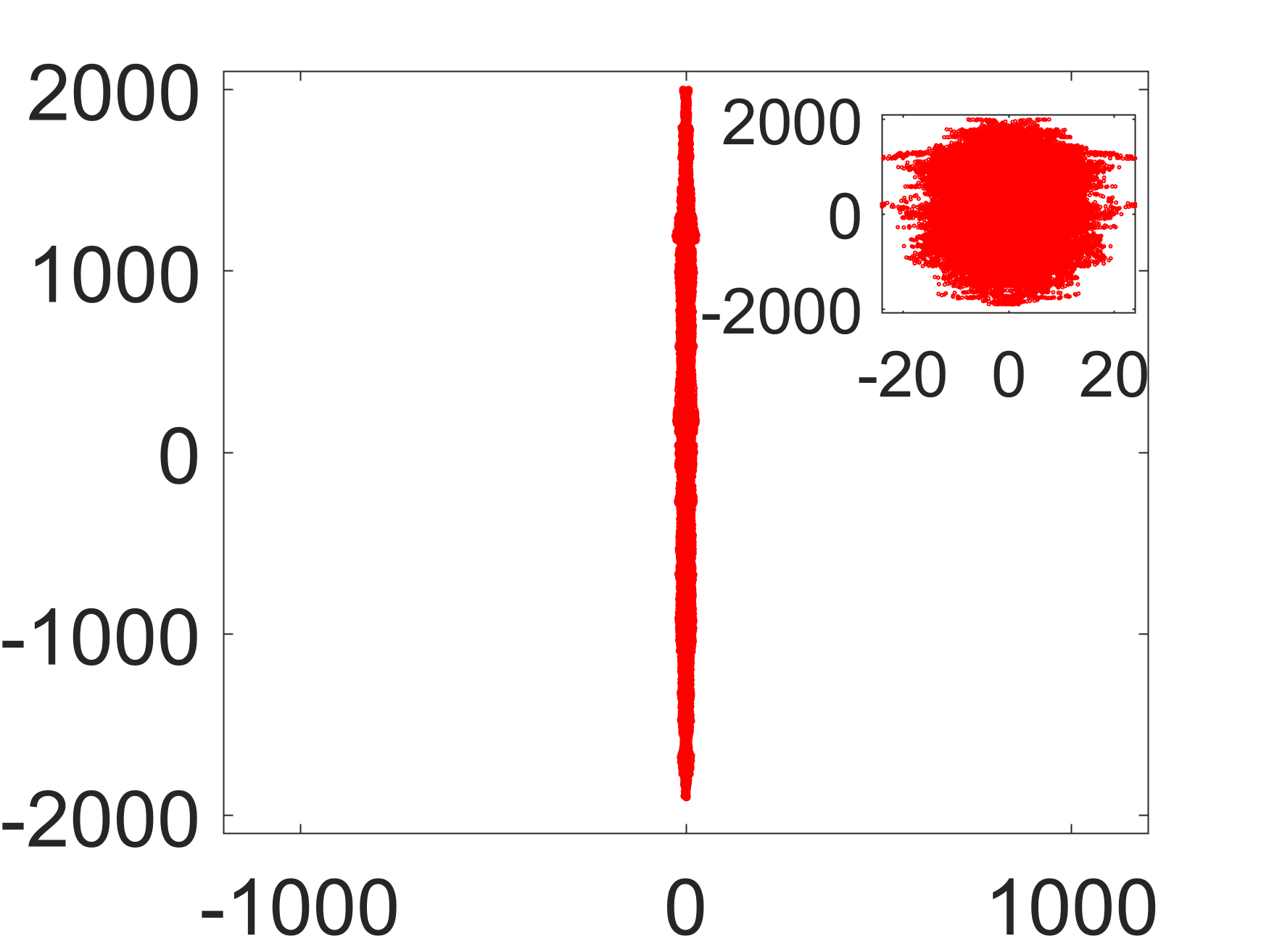}
				\\[-0cm]
				{$z_1$}
			\end{tabular}
			\vspace*{0.15cm}
			\\
						& & performance output $z^t=Q^{1/2}\,x^t$: & 
			\\
			\hspace{-.0cm}
			\begin{tabular}{c}
				\rotatebox{90}{\hspace{0.75 cm} $z_2$ }
			\end{tabular}
			&
			\hspace{-.0cm}
			\begin{tabular}{c}
				\includegraphics[width=.24\textwidth]{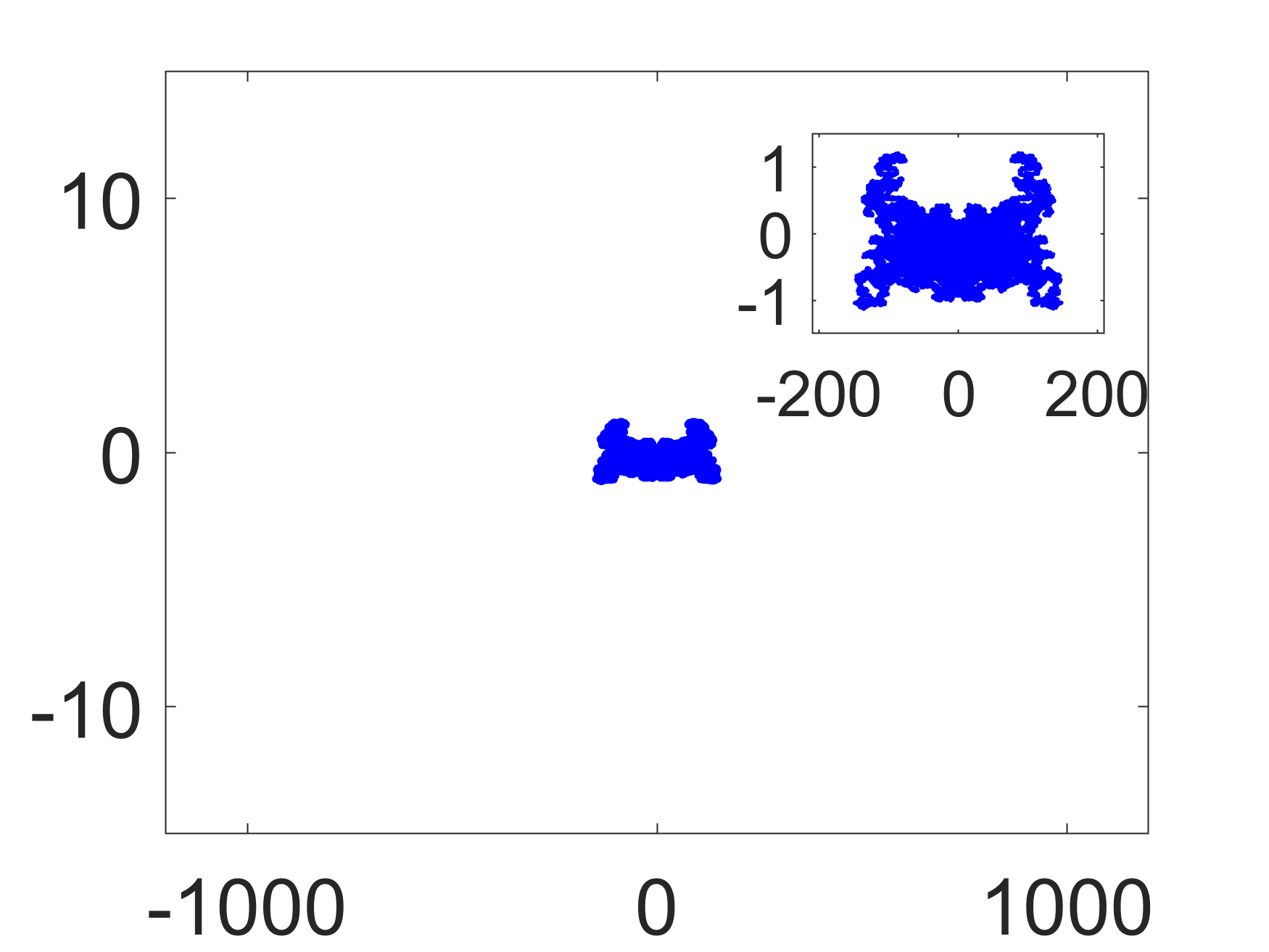}
				\\[-0cm]
				{$z_1$}
			\end{tabular}
			&
			\hspace{0cm}
			\begin{tabular}{c}
				\includegraphics[width=.24\textwidth]{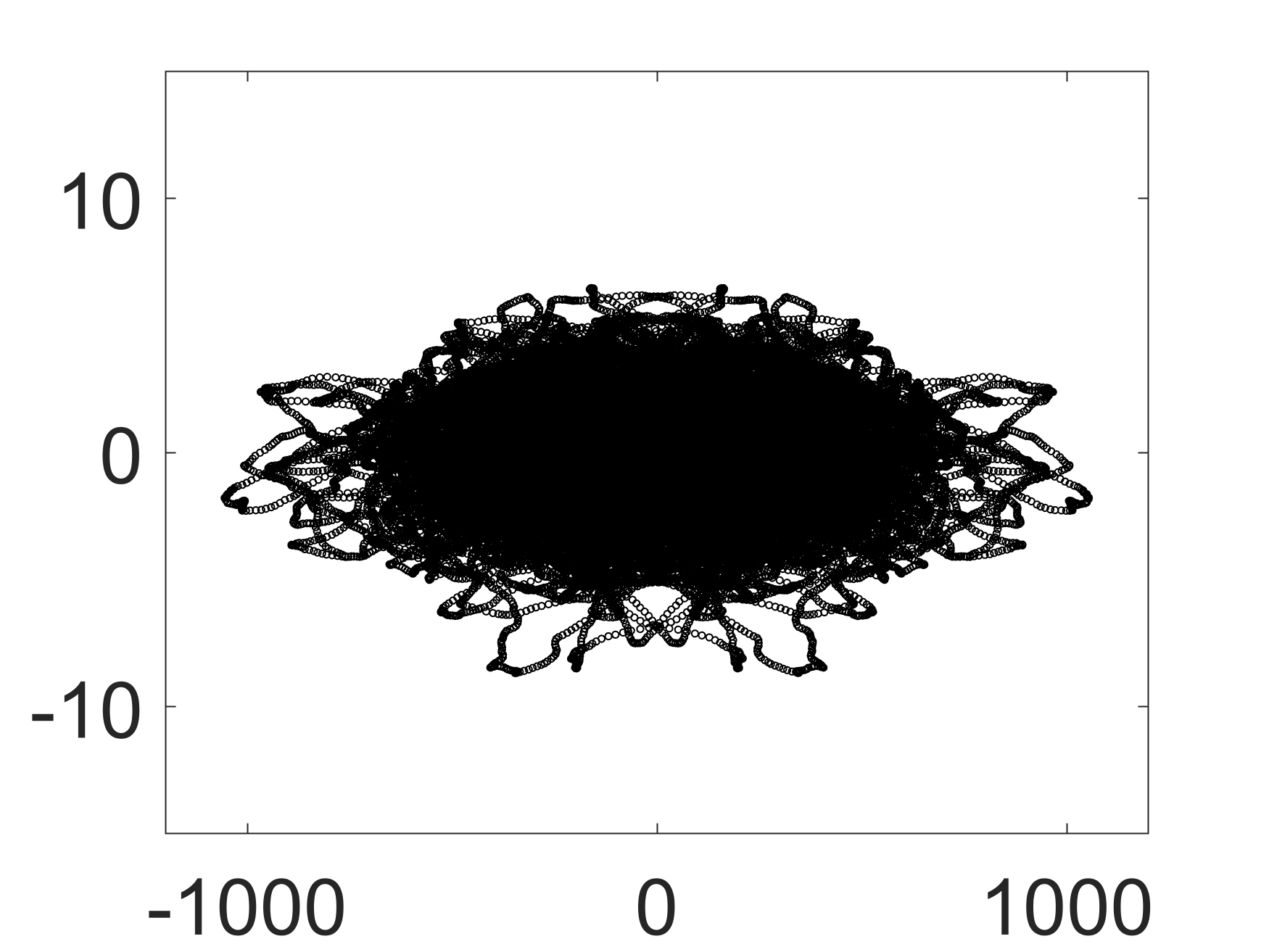}
				\\[-0cm]
				{$z_1$}
			\end{tabular}
			&
			\hspace{-0cm}
			\begin{tabular}{c}
				\includegraphics[width=.24\textwidth]{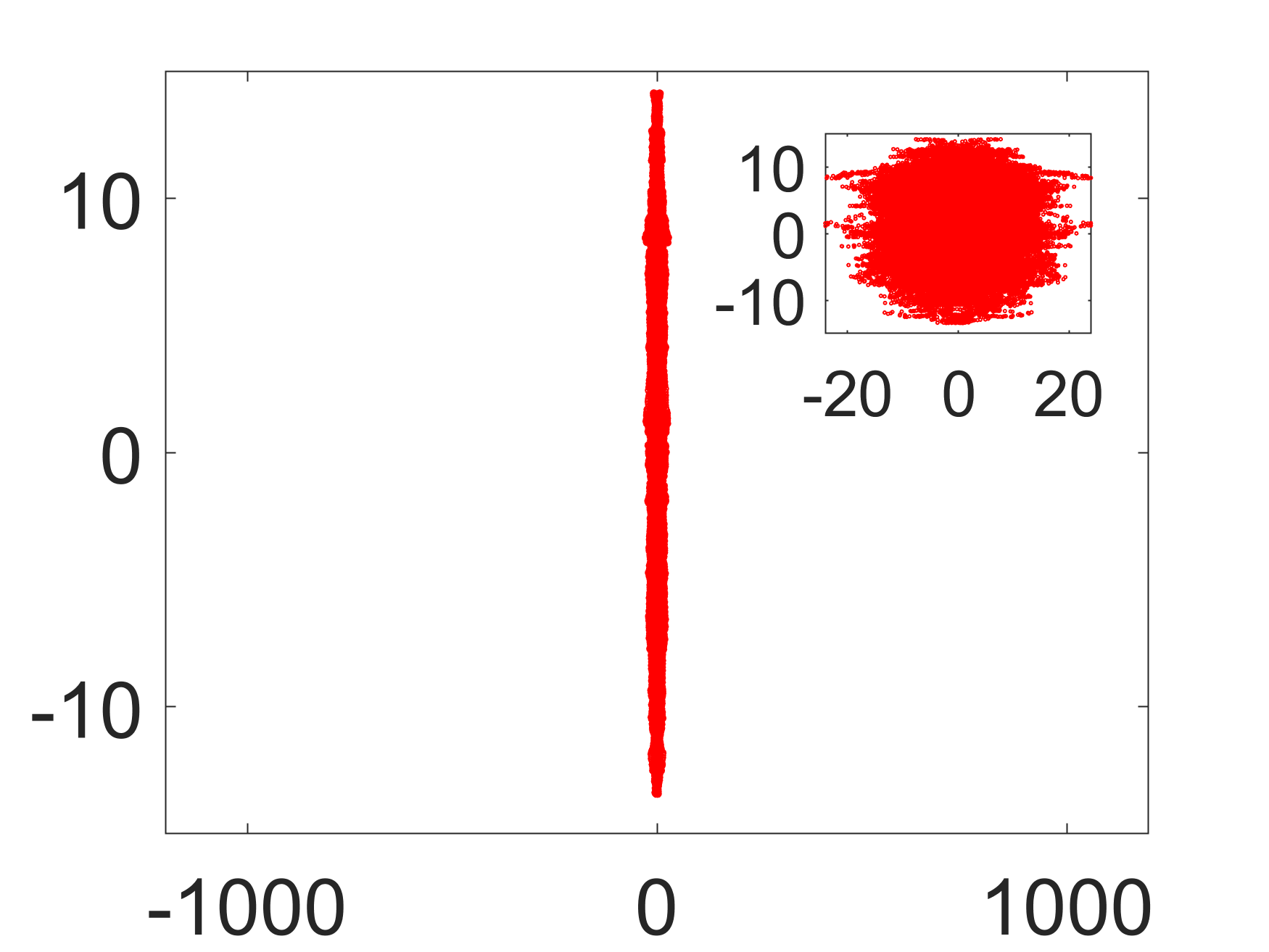}
				\\[-0cm]
				{$z_1$}
			\end{tabular}
			\\
			&
			\hspace{0cm}
			\begin{tabular}{cl}
				\begin{tabular}{c}
					\subfloat[\label{figScatter.1}]{}	

				\end{tabular} & \hspace{-0.5 cm}
				\begin{tabular}{c}\vspace{-0.43 cm}
					Gradient descent
				\end{tabular}
			\end{tabular}
			&
			\hspace{-0cm}
			\begin{tabular}{cl}
				\begin{tabular}{c}
					\subfloat[\label{figScatter.2}]{}	
				\end{tabular} & \hspace{-0.5 cm}
				\begin{tabular}{c}\vspace{-0.43 cm}
					Heavy-ball
				\end{tabular}
			\end{tabular}
			&
			\hspace{-0cm}
			\begin{tabular}{cl}
				\begin{tabular}{c}
					\subfloat[\label{figScatter.3}]{}	
				\end{tabular} & \hspace{-0.5 cm}
				\begin{tabular}{c}\vspace{-0.43 cm}
					Nesterov
				\end{tabular}
			\end{tabular}
		\end{tabular}
		\caption{Performance outputs $z^t=x^t$ (top row) and $z^t = Q^{1/2}x^t$ (bottom row) resulting from  $10^5$ iterations of  noisy first-order algorithms~\eqref{eq.1st} with the parameters provided in Table~\ref{tab:rates}. Strongly convex problem with $f(x) = 0.5\, x_1^2 + 0.25\times10^{-4}\,x_2^2$ ($\kappa =2\times 10^4$) is solved using algorithms with additive white noise and zero initial conditions.}
		\label{figScatter}
	\end{figure*}
	
	\begin{center}	
	\begin{figure*}[h]
		\centering
		\begin{tabular}{rc@{\hspace{0.75 cm}}rc}
			\hspace{-.0cm}
			\begin{tabular}{c}
				\vspace{.4cm}
				\rotatebox{90}{$\sum\limits_{k \, = \, 0}^t 
				\!
				\tfrac{1}{t} \, \norm{z^k}^2$}
			\end{tabular}
			&
			\hspace{-0.75cm}
			\begin{tabular}{c}
				\includegraphics[width=.31\textwidth]{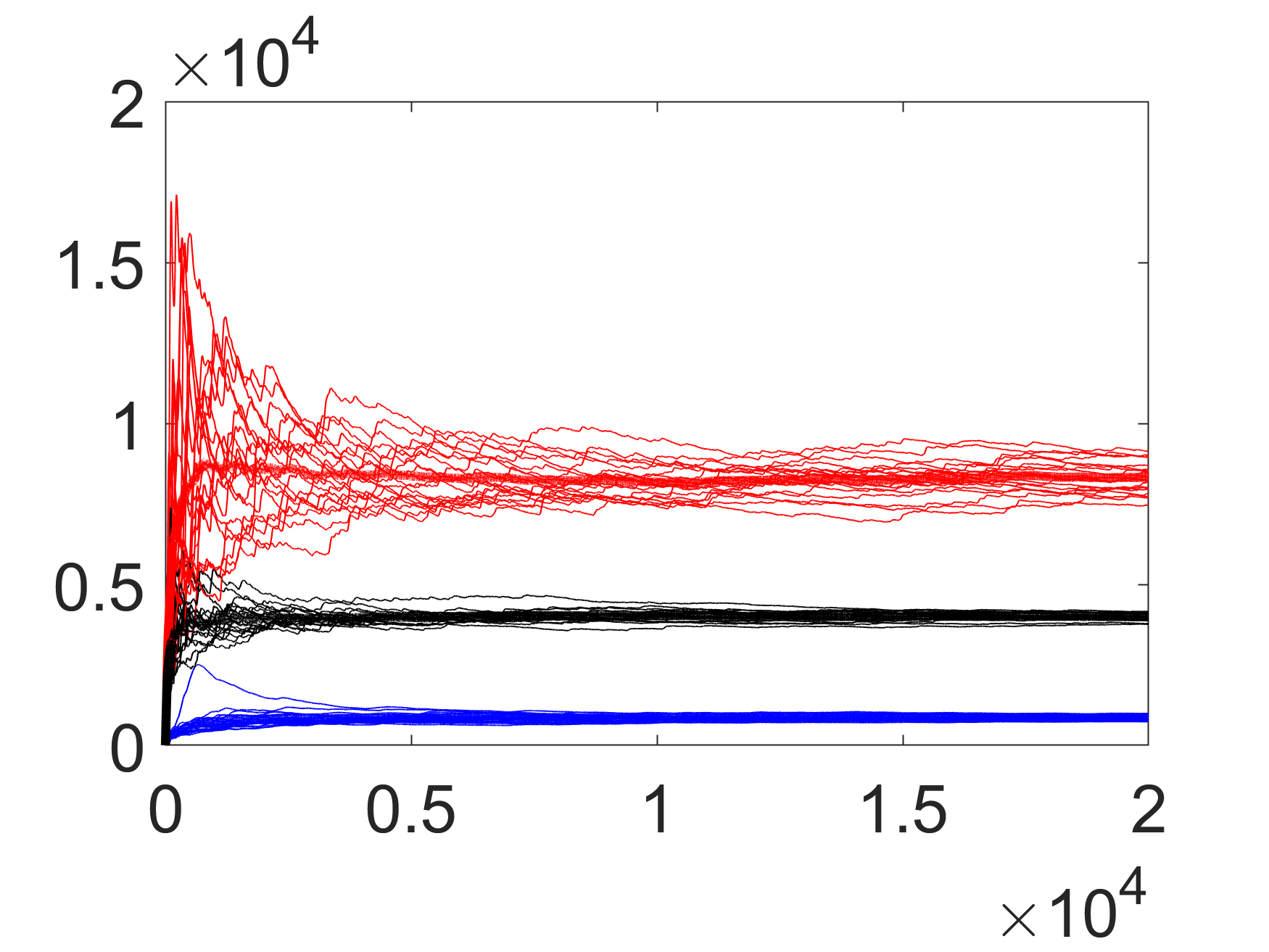}
			\end{tabular}
			&
			&
			\hspace{-0.75cm}
			\begin{tabular}{c}
				\includegraphics[width=.31\textwidth]{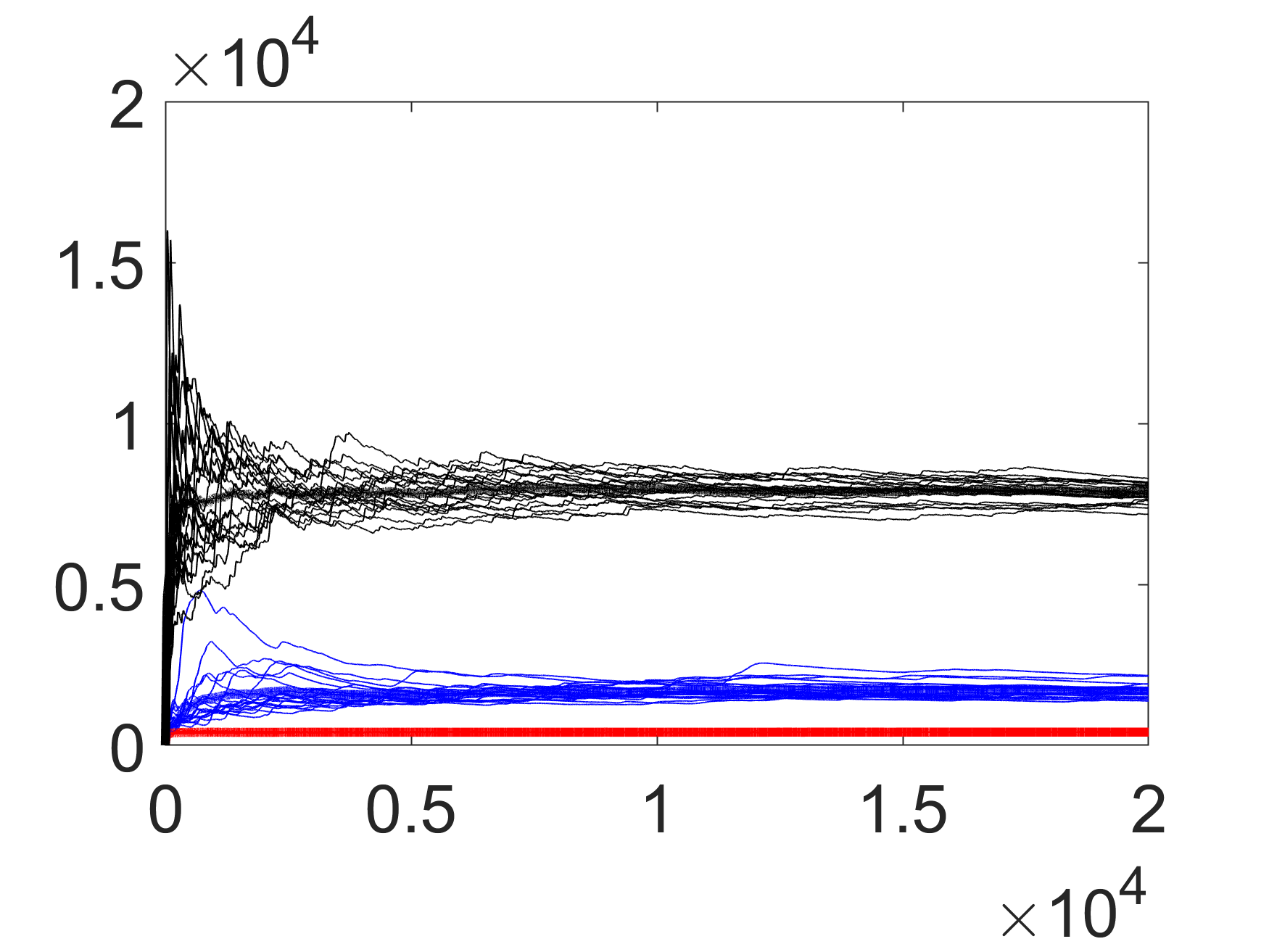}
							\end{tabular}
			\\[-.15cm]
			\begin{tabular}{c}
				\vspace{.4cm}
				\rotatebox{90}{$\sum\limits_{k \, = \, 0}^t 
					\!
					\tfrac{1}{t} \, \norm{z^k}^2$}
			\end{tabular}
			&
			\hspace{-0.75cm}
			\begin{tabular}{c}
				\includegraphics[width=.31\textwidth]{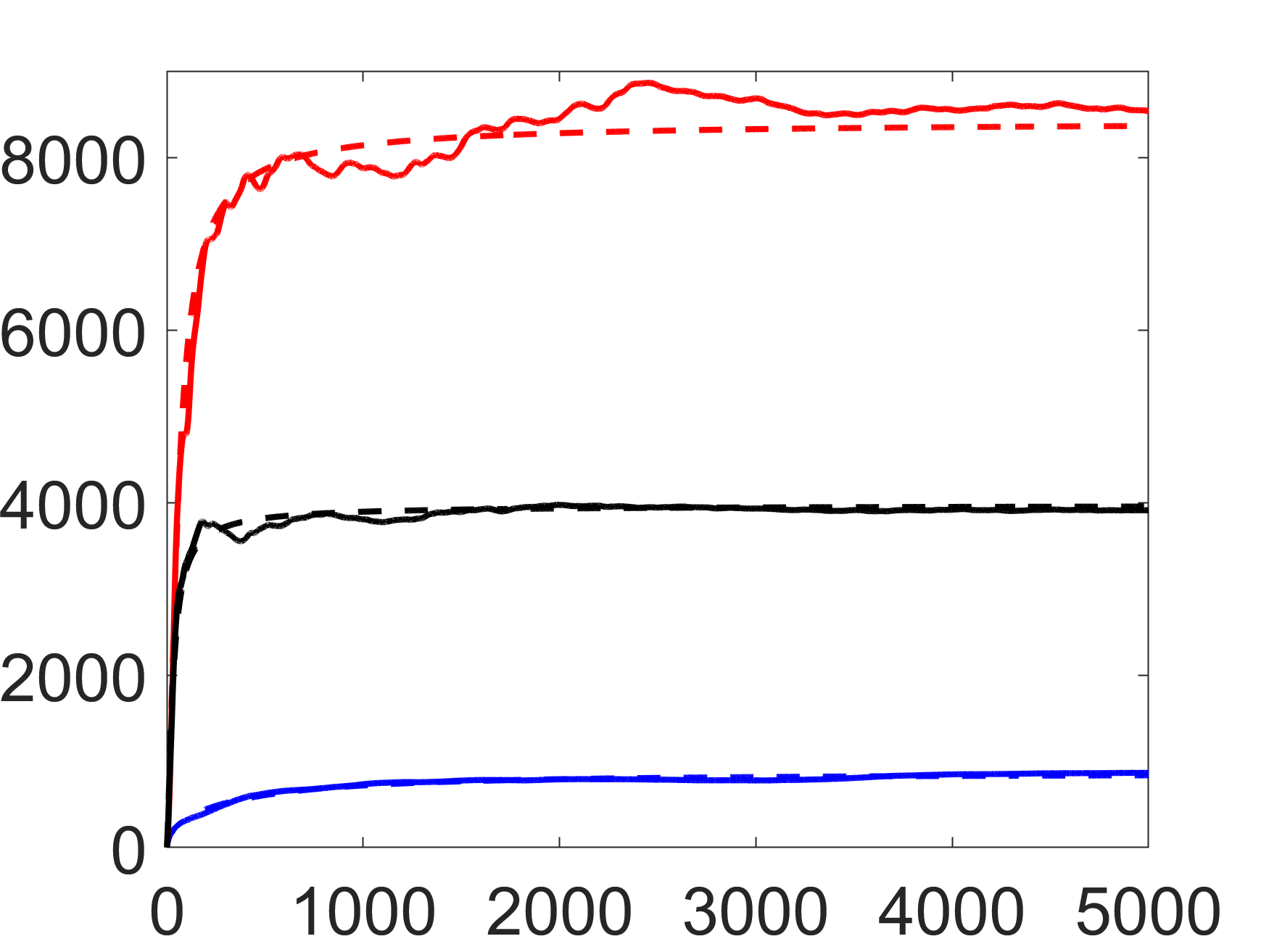}
				\\[-0.05 cm]
				{iteration number $t$}
			\end{tabular}
			&
			&
			\hspace{-0.75cm}
			\begin{tabular}{c}
				\includegraphics[width=.31\textwidth]{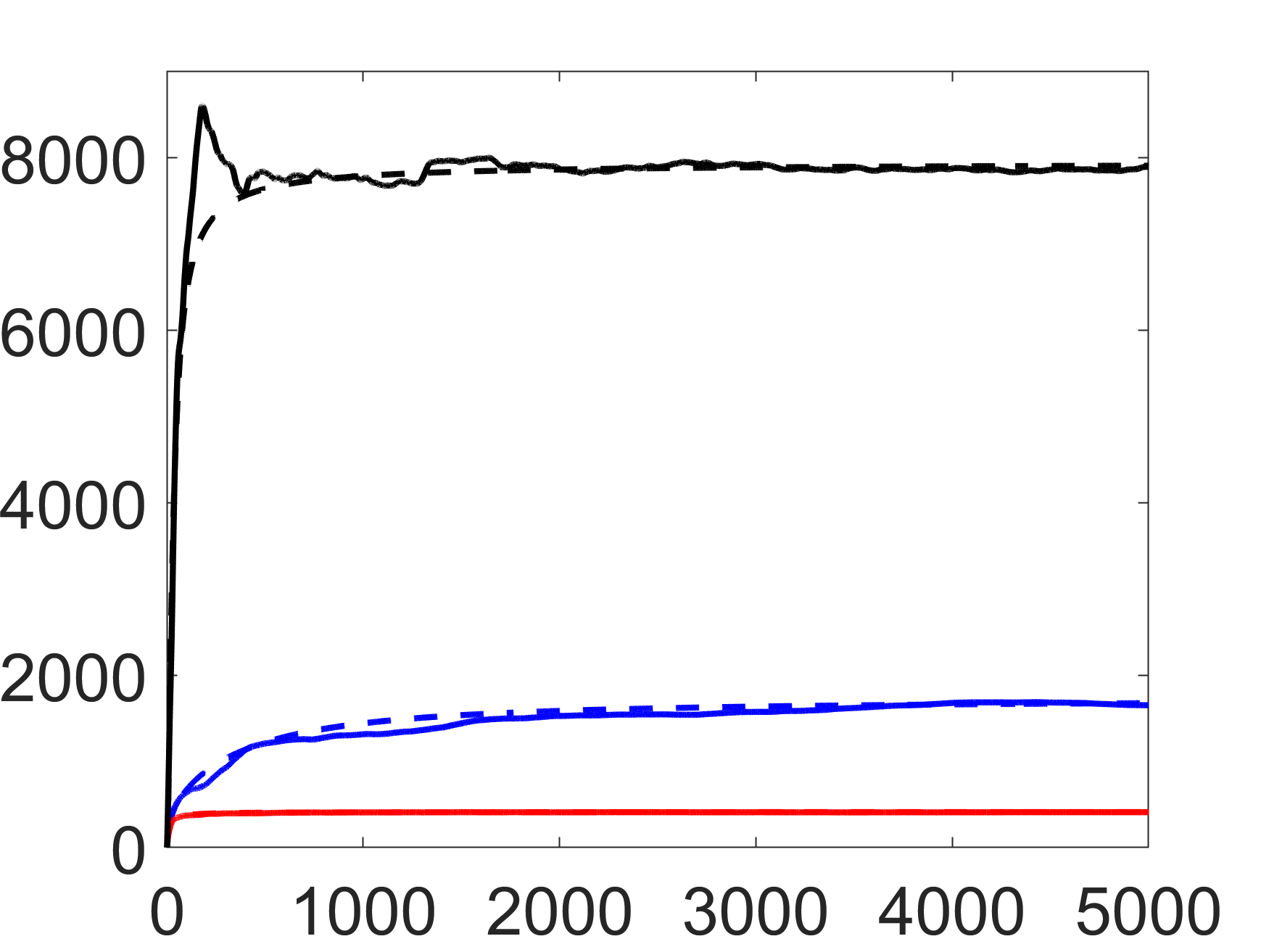}
				\\[-0.05cm]
				{iteration number $t$}
			\end{tabular}
			\\[-.3cm]
			&
			\hspace{-1.4 cm}
			\begin{tabular}{cl}
				\begin{tabular}{c}
					\subfloat[\label{figStoch.J}]{}	
				\end{tabular} & \hspace{-0.6 cm}
				\begin{tabular}{c}\vspace{-0.43 cm}
					performance output $z^t = x^t$
				\end{tabular}
			\end{tabular}
			&&
			\hspace{-1.58cm}
			\begin{tabular}{cl}
				\begin{tabular}{c}
					\subfloat[\label{figStoch.Jprim}]{}	
				\end{tabular} & \hspace{-0.6 cm}
				\begin{tabular}{c}\vspace{-0.43 cm}
					performance output $z^t = Q^{1/2} x^t$
				\end{tabular}
			\end{tabular}
		\end{tabular}
		\caption{$(1/t) \sum_{k \, = \, 0}^t 
					\norm{z^k}^2$ for the performance output $z^t$ in Example~2. Top row: the thick blue (gradient descent), black (heavy-ball), and red (Nesterov's method) lines mark variance  obtained by averaging results of twenty stochastic simulations. Bottom row: comparison between results obtained by averaging outcomes of twenty stochastic simulations (thick lines) with the corresponding theoretical values $ ({1}/{t}) \sum_{k \, = \, 0}^t \trace \, (C P^k C^T)$ (dashed lines) resulting from the Lyapunov equation~\eqref{eq.LyapPt}.}
		\label{fig.toeplitz-example}
	\end{figure*}
\end{center}

	 \vspace*{-8ex}
\section{General strongly convex problems}
\label{sec.general}

In this section, we extend our results to the class $\mathcal{F}_{\mf}^{\Lf}$ of $\mf$-strongly convex objective functions with $\Lf$-Lipschitz continuous gradients. While a precise characterization of noise amplification for general problems is challenging because of the nonlinear dynamics, we employ tools from robust control theory to obtain meaningful upper bounds. Our results utilize the theory of integral quadratic constraints~\cite{megran97}, a convex control-theoretic framework that was recently used to analyze optimization algorithms~\cite{lesrecpac16} and study convergence and robustness of the first-order methods~\cite{hules17,fazribmor18,cyrhuvanles18,dhikhojovTAC19}. We establish analytical upper bounds on the mean-squared error of the iterates~\eqref{eq.Jnew} for gradient descent~\eqref{alg.GD} and Nesterov's accelerated~\eqref{alg.NA} methods. Since there are no known accelerated convergence guarantees for the heavy-ball method when applied to general strongly convex functions, we do not consider it in this section. 

We first exploit structural properties of the gradient and employ quadratic Lyapunov functions to formulate a semidefinite programing problem (SDP) that provides upper bounds on $J$ in~\eqref{eq.Jnew}. While quadratic Lyapunov functions yield tight upper bounds for gradient descent, they fail to provide any upper bound for Nesterov's method for large condition numbers ($\kappa>100$). To overcome this challenge, we present a modified semidefinite program that uses more general Lyapunov functions which are obtained by augmenting standard quadratic terms with the objective function. This type of generalized Lyapunov functions has been introduced in~\cite{polshc17,fazribmor18} and used to study convergence of optimization algorithms for non-strongly convex problems. We employ a modified SDP to derive meaningful upper bounds on $J$ in~\eqref{eq.Jnew} for Nesterov's method as~well.  

We note that algorithms~\eqref{eq.1st} are invariant under translation, i.e., if we let $\tilde{x} \DefinedAs x - \bar{x}$ and $g(\tilde{x}) \DefinedAs f(\tilde{x} + \bar{x})$, then~\eqref{alg.NA}, for example, satisfies
\[
\ba{rcl}
\tilde{x}^{t+2} 
& \!\!\! = \!\!\! &
\tilde{x}^{t+1}
\; +\;
\tc{black}{\beta}
(
\tilde{x}^{t+1}
\, - \,  
\tilde{x}^{t}
)
\; -\;
\tc{black}{\alpha}
\nabla 
g
\!
\left(
\tilde{x}^{t+1}
\, + \,
\tc{black}{\beta}
(
\tilde{x}^{t+1}
\, - \,  
\tilde{x}^{t})
\right)	
+ \;
\tc{black}{\sigma w^t}.
\ea
\]
Thus, in what follows, without loss of generality, we assume that $x^\star = 0$ is the unique minimizer of~\eqref{eq.f}.

{\color{black}
	\subsection{An approach based on contraction mappings}
	Before we present our approach based on Linear Matrix Inequalities (LMIs), we provide a more intuitive approach that can be used to examine noise amplification of gradient descent. Let $\varphi$$:\R^n\rightarrow\R^n$ be a contraction mapping, i.e., there exists a positive scalar $\eta<1$ such that
	$
	\norm{ \varphi(x) - \varphi(y) } \le \eta \norm{ x - y }
	$
	for all $x$, $y\in\R^n$, and let $x^\star=0$ be the unique fixed point of $\varphi$, i.e, $\varphi(0)=0$. For the noisy recursion
	$
	x^{t+1} =\varphi(x^t) + \sigma w^t,
	$
	where $w^t$ is a zero-mean white noise with identity covariance and $\EX((w^t)^T \varphi (x^t)) = 0$, the contractiveness of $\varphi$ implies
	\begin{align*}
	\EX(\norm{x^{t+1}}^2)
	\, = \,
	\EX(\norm{\varphi(x^t) + \sigma w^t}^2)
	\, \le \,
	\eta^2\EX(\norm{x^t}^2) \,+\,n \sigma^2.
	\end{align*}
Since $\eta<1$, this relation yields
	$$
	\lim_{t \, \to \, \infty}\EX(\norm{x^t}^2)
	\;\le\;
	\dfrac{n\sigma^2}{1 \, - \, \eta^2}.
	$$ 
	If $\eta\DefinedAs\max\{|1-\alpha \mf|,|1-\alpha\Lf|\} < 1$, the map $\varphi(x) \DefinedAs x-\alpha \nabla f(x)$ is a contraction~\cite{ber15book}. Thus, for the conventional stepsize $\alpha=1/\Lf$ we have $\eta=1-1/\kappa$, and the bound becomes
	\begin{align*}
	\lim_{t \, \to \, \infty}\EX(\norm{x^t}^2)
	\, \le \,
	\dfrac{n\sigma^2}{1 \, - \, \eta^2}
	\, = \,
	\dfrac{n\sigma^2 \kappa^2}{2 \kappa \, - \, 1}
	\, = \,
	n \Theta(\kappa).
	\end{align*}
	In the next section, we show that this upper bound is indeed tight for the class of functions $\mathcal{F}_{\mf}^{\Lf}$. While this approach yields a tight upper bound for gradient descent, it cannot be used for Nesterov's method (because it is not a contraction).} 
	
\subsection{\tc{black}{An approach based on linear matrix inequalities}}
For any function $f\in\mathcal{F}_\mf^\Lf$, the nonlinear mapping $\Delta$$:\R^n\rightarrow\R^n$ 
{\color{black}
\beq
\Delta (y)
\; \DefinedAs \;
\nabla f(y) \; - \; \mf \, y
\non
\eeq }
satisfies the  quadratic inequality~\cite[Lemma 6]{lesrecpac16}
\be
\tbo{\!\!y \,-\, y_0\!\!}{\!\!\!\Delta (y) \,-\, \Delta(y_0)\!\!\!}^T
\!\Pi
\tbo{\!\! y \,-\, y_0\!\!}{\!\!\!\Delta (y) \,-\, \Delta(y_0)\!\!\!}
\;\geq\; 
0\label{eq.IQC}
\eeq
for all $y$, $y_0 \in \bbR^n$, where the matrix $\Pi$ is given by
\begin{align}\label{eq.PI}
\Pi 
\;\DefinedAs\;
\tbt{0}{(\Lf \,-\, \mf)I }{(\Lf \,-\, \mf)I}{-2I}.
\end{align}
We can bring algorithms~\eqref{eq.1st} with constant parameters into a time-invariant state-space form
\begin{subequations}
	\beq
	\ba{rcl}
	\psi^{t+1}
	& \!\!\! = \!\!\! &
	A \, \psi^t
	\; + \;
	\tc{black}{\sigma}B_w w^t
	\; + \;
	B_u u^t
	\\[0.1cm]
	\left[ 
	\ba{c}{z^{t}} \\[-0.15cm] {y^{t}} \ea
	\right]
	& \!\!\! = \!\!\! &
	\left[
	\ba{c}
	{C_z}
	\\[-0.15cm]
	{C_y} 
	\ea
	\right]
	\psi^t	
	\\[0.35cm]
	u^{t}
	& \!\!\! = \!\!\! &
	\Delta ( y^t )
	\ea
	\label{eq.ss-ns}
	\eeq
that contains a feedback interconnection of linear and nonlinear components. Figure~\ref{fig.blockDiagram} illustrates the block diagram of system~\eqref{eq.ss-ns}, where $\psi^t$ is the state, $w^t$ is a white stochastic noise, $z^t$ is the performance output, and $u^t$ is the output of the nonlinear term $\Delta (y^t)$. In particular, if we let
	\beq
	\psi^t 
	\, \DefinedAs \,
	\left[ 
	\ba{c} {x^t} \\[-0.cm]{x^{t+1}} \ea
	\right],
	\quad
	z^t
	\, \DefinedAs \,
	x^t,
	\quad
	y^t
	\, \DefinedAs \,
	- \beta x^t
	+
	(1 + \beta)  x^{t+1}
	\non
	\eeq
	and define the corresponding matrices  as
	\be
	\ba{rclrcl}
	A
	& \!\!\! = \!\!\! &
	\tbt{0}{I}{-\beta(1-\alpha\, \mf) I}{(1+\beta)(1-\alpha\, \mf) I},
	\quad
	B_w
	& \!\!\! = \!\!\! &
	\tbo{0}{ I},
	\quad
	B_u
	\, = \, 
	\tbo{0}{-\alpha\, I}
	\\[0.5cm]
	C_z
	& \!\!\! = \!\!\! &
	\obt{I}{0},
	\quad	
	C_y
	\, = \,
	\obt{-\beta \,I}{(1 + \beta)I} &&
	\label{eq.ss-na}
	\ea
	\ee
	then~\eqref{eq.ss-ns} represents Nesterov's method~\eqref{alg.NA}. For gradient descent~\eqref{alg.GD}, we can alternatively use 
	$ 
	\psi^t 
	=
	z^t 
	=
	y^t
	\DefinedAs
	x^t
	$
	with the corresponding matrices 
	\beq
	A 
	\;=\;  
	(1 \, - \, \alpha \,\mf) I,
	\quad
	B_w 
	\;=\; 
	I,
	\quad 
	B_u
	\;=\; 
	- \alpha I,
	\quad
	C_z
	\;=\; 
	C_y
	\;=\; 
	I.
	\label{eq.ss-gd}
		\eeq
\end{subequations}

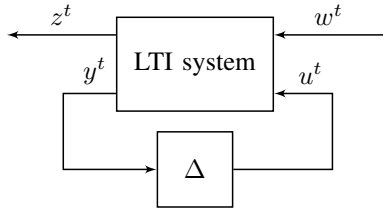
\begin{figure}[ht!]\begin{center}
		\begin{tikzpicture}[auto, node distance=2.5cm,>=latex',line width=0.2mm]
		\node[draw,rectangle,inner sep=0.25cm,minimum height = 1 cm,minimum width = 1cm] (nonlin) at (0,0) {$\Delta$};
		\matrix[matrix of nodes,nodes in empty cells,rectangle,draw] (sys) at (0,1.4)
		{
			$$\\
			LTI system	 \\
			$$\\
		};
		\draw[->] (nonlin.east)  -| ([xshift=1.7 cm] sys-3-1.east |- sys-3-1.east) node [name=u,above left] {$u^t$} |- (sys.east |- sys-3-1.east);
		\draw[->] (sys.west |- sys-3-1.west) node [name=y,above left] {$y^t$}-| ([xshift=-0.7 cm] sys.west |- sys-3-1.west) |- (nonlin.west);
		\node [input, name=input, right of=sys-1-1] {};
		\draw[->](input)--node [name=w,above] {$w^t$}(sys-1-1-|sys.east);
		\node[output,name=output, left of=sys-1-1]{};
		\draw[->](sys-1-1-|sys.west)--node [name=,above] {$z^t$}(output);
		\end{tikzpicture}
	\end{center}
	\caption{Block diagram of system~\eqref{eq.ss-ns}. }\label{fig.blockDiagram}
\end{figure}

In what follows, we demonstrate how property~\eqref{eq.IQC} of the nonlinear mapping $ \Delta $  allows us to obtain upper bounds on $J$ when  system~\eqref{eq.ss-ns} is driven by the white stochastic input $w^t$  with zero mean and identity covariance.  Lemma~\ref{lem:LMI-GD} uses a quadratic Lyapunov function of the form $V(\psi) =\psi^TX\psi$ and provides  upper bounds on the steady-state second-order moment of the performance output $z^t$ in terms of solutions to a certain LMI. This approach yields a tight upper bound for gradient descent.

\begin{mylem} \label{lem:LMI-GD}
	Let the nonlinear function $ u=\Delta(y)$ satisfy the quadratic inequality
	\be
	\label{eq:jointIneqApp11}
	\tbo{y}{u}^T\Pi\,\tbo{y}{u}
	\; \ge \; 
	0
	\ee
	for some matrix $ \Pi $, let $X$ be a positive semidefinite matrix, and let $\lambda$ be a nonnegative scalar such that system~\eqref{eq.ss-ns} satisfies
	\be
	\ba{rcl}
	\tbt{A^T X \,A - X + C_z^T\, C_z}{A^T X\, B_u}{B_u^T \,X\, A}{B_u^T\, X\, B_u}
	\; + \;
	\lambda 
	\tbt{C_y^T}{0}{0}{I} 
	\Pi
	\tbt{C_y}{0}{0}{I}
	& \!\!\! \preceq \!\!\! &
	0. \label{LMI}
	\ea
	\ee
	Then the steady-state second-order moment $J$ of the performance output $ z^t $ in~\eqref{eq.ss-ns} is bounded by 
	\[
	J
	\; \le \;
	\tc{black}{\sigma^2} \, \trace \, (B_w^T \, X \, B_w).
	\] 
\end{mylem} 
\begin{proof}
	See Appendix~\ref{app.General}.
\end{proof} 

For Nesterov's accelerated method with the parameters provided in Table~\ref{tab:ratesGeneral}, computational experiments show that LMI~\eqref{LMI} becomes infeasible for large values of the condition number $ \kappa $. Thus, Lemma~\ref{lem:LMI-GD} does not provide sensible upper bounds on $J$ for Nesterov's algorithm. This observation is consistent with the results of~\cite{lesrecpac16}, where it was suggested that analyzing the convergence rate requires the use of additional quadratic inequalities, apart from~\eqref{eq.IQC}, to further tighten the constraints on the gradient $\nabla f$ and reduce conservativeness. In what follows, we build on the results of~\cite{fazribmor18}  and present an alternative LMI in Lemma~\ref{lem.generalNester} that is obtained using  a Lyapunov function of the form $V(\psi)$$=\psi^TX\psi+f([\,0~\,I\,] \psi)$, where $X$ is a positive semidefinite matrix and $f$ is the objective function in~\eqref{eq.f}. Such Lyapunov functions have been used to study convergence of optimization algorithms in~\cite{polshc17}. The resulting approach allows us to establish an order-wise tight analytical upper bound on $J$ for Nesterov's accelerated method.

\begin{mylem}
	\label{lem.generalNester}
	Let the matrix $M(\mf,\Lf;\alpha,\beta) $ be defined as 
	\begin{align*}		
	M
	\,\DefinedAs\,
	N_1^T\tbt{\Lf\,I}{I}{I}{0}N_1 
	\,+\,
	N_2^T\tbt{-\mf\,I}{I}{I}{0}N_2
	\end{align*}
	where
	\be
	\ba{rclrcl}
	N_1
	&\!\!\! \DefinedAs \!\!\! &
	\tbth{\alpha\,\mf\,\beta\,I}{-\alpha\,\mf(1+\beta)\,I}{-\alpha\,I}
	{-\mf\,\beta\,I}{\mf(1+\beta)\,I}{I},
	\quad
	N_2
	&\!\!\! \DefinedAs \!\!\! &
	\tbth{-\beta\,I}{\beta\,I}{0}
	{-\mf\,\beta\,I}{\mf(1+\beta)\,I}{I}.
	\ea
	\non
	\ee
	Consider state-space model~\eqref{eq.ss-ns}-\eqref{eq.ss-na} for algorithm~\eqref{alg.NA} and let $\Pi$ be given by~\eqref{eq.PI}. Then, for any positive semidefinite matrix $ X $ and scalars $ \lambda_1\ge0 $ and $\lambda_2\ge0$ that satisfy
	\be
	\ba{rcl}
	\tbt{A^T X\, A - X + C_z^T\, C_z}{A^T X\, B_u}{B_u^T\, X\, A}{B_u^T\, X\, B_u}
	\; + \;
	\lambda_1 
	\tbt{C_y^T}{0}{0}{I} 
	\Pi
	\tbt{C_y}{0}{0}{I}
	\;+\;
	\lambda_2\,
	M
	&\!\!\! \preceq \!\!\! &
	0
	\ea
	\label{LMI2}
	\ee
	the steady-state second-order moment $J$ of the performance output $ z^t $ in~\eqref{eq.ss-ns} is bounded by 
		\be
		\label{eq.UpperBoundNesterov}
		J
		\; \le \;
		\tc{black}{\sigma^2}
		\left(n\, \Lf\,\lambda_2
		\, + \,
		\trace \, (B_w^T \, X \, B_w)\right).
		\ee	
\end{mylem} 
\begin{proof}
	See Appendix~\ref{app.General}.
\end{proof} 
\begin{myrem}
	Since LMI~\eqref{LMI2} simplifies to~\eqref{LMI} by setting $\lambda_2=0$, Lemma~\ref{lem.generalNester} represents a relaxed version of Lemma~\ref{lem:LMI-GD}. This modification is the key enabler to establishing tight upper bound on $J$ for Nesterov's method. 
\end{myrem}

{\color{black}
		The upper bounds provided in Lemmas~\ref{lem:LMI-GD} and~\ref{lem.generalNester} are proportional to $\sigma^2$. In what follows, to make a  connection between these bounds and  our analytical expressions for the variance amplification in the quadratic case (Section~\ref{sec.Quadratic}), we again set $\sigma =1$. } 
The best upper bound on $J$ that can be obtained using Lemma~\ref{lem.generalNester} is given by the optimal objective value of the semidefinite program
\begin{align}\label{prob.LMI2}
\ds{\minimize_{X, \, \lambda_1, \, \lambda_2}}
& 
\quad
n\, \Lf\,\lambda_2
\;+\;
\trace \, (B_w^T \, X \, B_w)
\\[.0cm]
	\nonumber
\subject
&
\quad
\mbox{LMI}~\eqref{LMI2},
~~
X \, \succeq \, 0,
~
\lambda_1 \, \ge \, 0,
~
\lambda_2 \, \ge \, 0.	
\end{align}
	For system matrices~\eqref{eq.ss-na}, LMI~\eqref{LMI2} is of size $3 n \times 3 n$ where $x^t \in \bbR^n$. However, if we impose the additional constraint that the matrix $X$ has the same block structure as $A$, 
	\[
	X \; = \, \tbt{x_1 I}{x_0 I}{x_0 I}{x_2 I}
	\] 
for some scalars $x_1$, $x_2$, and $x_0$, then using appropriate permutation matrices, we can simplify~\eqref{LMI} into an LMI of size $3 \times 3$. Furthermore, imposing this constraint comes without loss of generality. In particular, the optimal objective value of problem~\eqref{prob.LMI2} does not change if we require $X$ to have this structure; see~\cite[Section 4.2]{lesrecpac16} for a discussion of this lossless dimensionality reduction for LMI constraints with similar structure. 
	
In Theorem~\ref{thm.summary}, we use Lemmas~\ref{lem:LMI-GD} and~\ref{lem.generalNester} to establish tight upper bounds on $J_\gd$ and $J_\na$ for all $f \in \mathcal{F}_{\mf}^{\Lf}$. 
\begin{mythm} 
	\label{thm.summary}
	For gradient descent and Nesterov's accelerated method with the parameters provided in Table~\ref{tab:ratesGeneral} \tc{black}{and $\sigma=1$},  the performance measures $J_\gd$ and $J_\na$ of the error $x^t-x^\star\in\R^n$ satisfy
	\begin{align*}
	\sup_{f \, \in \, \mathcal{F}_\mf^\Lf}  J_\gd
	\;=\;
	q_\gd,
	\qquad
	q_\na 
	\; \le \;
	\sup_{f \, \in \, \mathcal{F}_\mf^\Lf}  J_\na
	\; \le \;
	4.08\, q_\na
	\end{align*}
	where 
	\[
		q_\gd
		\; = \;
		\dfrac{n \kappa^2}{2 \kappa-1}
		\; = \;
		n\,\Theta(\kappa),
		\quad
		q_\na
		\; = \;
		\dfrac{n \kappa^2 \! \left( 2\kappa - 2\sqrt{\kappa} + 1 \right)}{{\left(2 \sqrt{\kappa}-1\right)}^3}
		\;=\; 
		n\,\Theta(\kappa^{\frac{3}{2}})
		\]
	and $\kappa \DefinedAs \Lf/\mf$ is the condition number of the set $\mathcal{F}_\mf^\Lf$.
	\end{mythm}
	
\begin{proof}
	See Appendix~\ref{app.General}.
\end{proof}

The variance amplification of gradient descent and Nesterov's method for $f(x) = \tfrac{m}{2} \, x^T x$ in $\mathcal{F}_{\mf}^{\Lf}$ is determined by $q_\gd$ and $q_\na$, respectively, and these two quantities can be obtained using Theorem~\ref{th.varianceJhat}. In Theorem~\ref{thm.summary}, we use this strongly convex quadratic objective function to certify the accuracy of the upper bounds on $\sup J$ for all $f\in\mathcal{F}_{\mf}^{\Lf}$. In particular, we observe that the upper bound is exact for gradient descent and that  it is within a $4.08$ factor of the optimal for Nesterov's method.

For strongly convex objective functions with the condition number $\kappa$, Theorem~\ref{thm.summary} proves that gradient descent outperforms Nesterov's accelerated method in terms of the largest noise amplification by a factor of $\sqrt{\kappa}$. This uncovers the fundamental performance limitation of Nesterov's accelerated method when the gradient evaluation is subject to additive stochastic uncertainties.

	\vspace*{-2ex}
	\section{Tuning of algorithmic parameters}
	\label{sec.compTuning}
					
	The parameters provided in Table~\ref{tab:rates} yield the optimal convergence rate for strongly convex quadratic problems. For these specific values, Theorem~\ref{thm.JBoundsQuad} establishes upper and lower bounds on \tc{black}{the variance amplification} that reveal the negative impact of acceleration. However, it is relevant to examine whether the parameters can be designed to provide acceleration while reducing the variance amplification. 
	
	 While the convergence rate solely depends on the extreme eigenvalues $\mf=\lambda_{\min} (Q)$ and $\Lf=\lambda_{\max} (Q)$ of the Hessian matrix $Q$, variance amplification is influenced by the entire spectrum of $Q$ and its minimization is challenging as it requires the use of all eigenvalues. In this section, we first consider the special case of eigenvalues being symmetrically distributed over the interval $[\mf,\Lf]$ and demonstrate that for gradient descent and the heavy-ball method, the parameters provided in Table~\ref{tab:rates} yield a variance amplification that is within a constant factor of the optimal value. As we demonstrate in Section~\ref{sec:distributed}, symmetric distribution of the eigenvalues is encountered in distributed consensus over undirected torus networks. We also consider the problem of designing parameters for objective functions in which the problem size satisfies $n \ll \kappa$ and establish a trade-off between convergence rate and variance amplification. More specifically, we show that for any accelerating pair of parameters $\alpha$ and $\beta$ and bounded problem dimension $n$, the variance amplification of accelerated methods is larger than that of gradient descent by a factor of $\Omega(\sqrt{\kappa})$.
	 	
		\vspace*{-2ex}
\subsection{Tuning of parameters using the whole spectrum} Let $\Lf = \lambda_1 \ge \lambda_2\,\ge \cdots \ge \lambda_n = \mf > 0$ be the eigenvalues of the Hessian matrix $Q$ of the strongly convex quadratic objective function in~\eqref{eq.quadraticObjective}.  Algorithms~\eqref{eq.1st} converge linearly in expected value  to the optimizer $x^\star$ with the rate
\begin{align}\label{eq.rhoForm}
	\rho \; \DefinedAs \; \max_{i} \; \hat{\rho}(\lambda_i)
\end{align}
where $\hat{\rho}(\lambda_i)$ is the spectral radius of the matrix $\hat{A}_i$ given by~\eqref{eq.Ahat}. For any scalar $c>0$ \tc{black}{and fixed $\sigma$, let}
\begin{subequations}
	\label{eq.optimalSymDist}
 \be
 \label{eq.alphaStarBetaStarAcceleration}
 	\ba{rcll}
 	(\alpha^\star_{\hb}(c),\beta^\star_{\hb}(c))
 	&\!\!\! \DefinedAs \!\!\!&
 	\displaystyle{\argmin_{\alpha,\,\beta}}
	&
	J_\hb(\alpha,\beta)
	\\[0.15cm]
	&&
 	\subject 
	&
	\rho_\hb 
	\, \leq \,
	1 \, - \, \dfrac{c}{\sqrt{\kappa} }
	\ea
 	\ee
 for the heavy-ball method, and 
 	\be
	\label{eq.alphaStarBetaStarGD}
	\ba{rcll}
	 \alpha^\star_{\gd}(c)
 	&\!\!\!\DefinedAs \!\!\!&
 	\displaystyle{\argmin_{\alpha}} 
	&
	J_\gd(\alpha)
	\\[0.15cm]
	&&
	\subject
	& 
	\rho_\gd
	\, \leq \,
	1 \, - \, \dfrac{c}{\kappa} 
	\ea
 	\ee
	\end{subequations}
for gradient descent, where the expression for the variance amplification  $J$ is provided in Theorem~\ref{th.varianceJhat}. Here, the constraints enforce a standard rate of linear convergence for gradient descent and an accelerated rate of linear convergence for the heavy-ball method parametrized with the constant $c$. Obtaining a closed form solution to~\eqref{eq.optimalSymDist} is challenging because $J$ depends on all eigenvalues of the Hessian matrix $Q$. Herein, we focus on objective functions for which the spectrum of $Q$ is symmetric, i.e., for any eigenvalue $\lambda$, the corresponding mirror image $\lambda' := \Lf+\mf-\lambda$ with respect to $\frac{1}{2}(\Lf+\mf)$ is also an eigenvalue with the same algebraic multiplicity. For this class of problems, Theorem~\ref{thm.lowerboundBetaHeavyball} demonstrates that the parameters provided in Table~\ref{tab:rates} for gradient descent and the heavy-ball method yield variance amplification that is within a constant factor of the optimal.

\begin{mythm}
	\label{thm.lowerboundBetaHeavyball}	
	For any scalar $c>0$ \tc{black}{and fixed $\sigma$}, there exist constants $c_1\ge1$ and $c_2>0$ such that for any strongly convex quadratic objective function in which the spectrum of the Hessian matrix $Q$ is symmetrically distributed over the interval $[\mf,\Lf]$ with $\kappa \DefinedAs \Lf/\mf > c_1$, we have
	\begin{align*}
	J_\gd(\alpha^\star_\gd(c)) 
	\; \ge \; 
	\dfrac{1}{2} \,J_\gd(\alpha_\gd),
	\quad	
	J_\hb(\alpha^\star_\hb(c),\beta^\star_\hb(c)) 
	\; \ge \; 
	c_2\, J_\hb(\alpha_\hb,\beta_\hb)
	\end{align*} 
	where  parameters $\alpha_\gd$ and {\em ($\alpha_\hb,\beta_\hb$)} are provided in Table~\ref{tab:rates}, whereas $\alpha_\gd^\star(c)$ and {\em ($\alpha^\star_\hb(c),\beta^\star_\hb(c)$)} solve~\eqref{eq.optimalSymDist}.
\end{mythm}
\begin{proof}
	See Appendix~\ref{sec.Tuning parameters using the whole spectrumApp}.
\end{proof}

For strongly convex quadratic objective functions with symmetric spectrum of the Hessian matrix over the interval $[\mf,\Lf]$, Theorem~\ref{thm.lowerboundBetaHeavyball} shows that the variance amplifications of gradient descent and the heavy-ball method with the parameters provided in Table~\ref{tab:rates} are within a constant factors of the optimal values. As we illustrate in Section~\ref{sec:distributed}, this class of problems is encountered in distributed averaging over noisy undirected networks. Combining this result with the lower bound on $J_{\hb}(\alpha_{\hb},\beta_{\hb})$ and the upper bound on $J_{\gd}(\alpha_{\gd})$ established in Theorem~\ref{thm.JBoundsQuad}, we see that regardless of the choice of parameters, there is a fundamental gap of $\Omega(\sqrt{\kappa})$ between $J_{\hb}$ and $J_{\gd}$ as long as we require an accelerated rate of convergence.  
		
			\vspace*{-2ex}
\subsection{Fundamental lower bounds}

We next establish lower bounds on the variance amplification of accelerated methods that hold for any pair of $\alpha$ and $\beta $ for strongly convex quadratic problems with $\kappa \gg 1$. In particular, we show that the variance amplification of accelerated algorithms is lower bounded by $\Omega(\kappa^{3/2})$ irrespective of the choice of $\alpha$ and $\beta$.

The next theorem establishes a fundamental tradeoff between the convergence rate and variance amplification for the heavy-ball method.
	\begin{mythm}\label{thm.tradeOffHB}
		For strongly convex quadratic problems with any stabilizing parameters $\alpha>0$ and $0 < \beta < 1$ \tc{black}{ and with a fixed noise magnitude $\sigma$,} the heavy-ball method with \tc{black}{the linear convergence rate $\rho$ satisfies}
		\begin{subequations}
			{\color{black}
		\begin{align}\label{eq.tradOffHBa}
		\dfrac{J_\hb}{1 \, - \, \rho} \; \ge \; \sigma^2\left(\dfrac{\kappa+1}{8}\right)^2.
		\end{align}
			Furthermore, if $\sigma=\alpha$, i.e., when the only source of uncertainty is a noisy gradient, we have
	\begin{align}\label{eq.tradOffHBb}
	\dfrac{J_\hb}{1 \, - \, \rho} \; \ge \;  \left(\dfrac{\kappa}{8 L}\right)^2.
\end{align}	}	
\end{subequations}
	\end{mythm}
	\begin{proof}
		See Appendix~\ref{sec.proofOflowerBoundJstar}.
	\end{proof}
To gain additional insight, let us consider two special cases: (i) for $\alpha=1/L$ and $\beta \rightarrow 0^{+}$, we obtain gradient descent algorithm for which $1-\rho  = \Theta({1}/{\kappa})$ and $J = \Theta(\kappa)$; (ii) for the heavy-ball method with the parameters provided in Table~\ref{tab:rates}, we have $1-\rho  = \Theta({1}/{\sqrt{\kappa}})$ and $J = \Theta(\kappa\sqrt{\kappa})$. Thus, in both cases, ${J_\hb}/(1-\rho) = \Omega(\kappa^2)$. Theorem~\ref{thm.tradeOffHB} shows that this lower bound is fundamental and it therefore quantifies the tradeoff between the convergence rate and the variance amplification of the heavy-ball method for any choice of parameters $\alpha$ and $\beta$. \tc{black}{It is also worth noting that the lower bound for $\sigma=\alpha$ depends on the largest eigenvalue $L$ of the Hessian matrix $Q$. Thus, this bound is meaningful when the value of $L$ is uniformly upper bounded. This scenario occurs in many applications including consensus over undirected tori networks; see Section~\ref{sec:distributed}. }
	
	While we are not able to show a similar lower bound for Nesterov's method, in the next theorem, we establish an asymptotic lower bound on the variance amplification that holds for any pair of accelerating parameters ($\alpha,\beta$) for both Nesterov's and heavy-ball methods.
	\begin{mythm}\label{th.lowerBoundJstar} 
		For a strongly convex quadratic objective function with condition number $\kappa$, let $c>0$ be a constant such that either Nesterov's algorithm or the heavy-ball method with some (possibly problem dependent) parameters $\alpha>0$ and $0 < \beta < 1$ converges linearly with a rate $\rho \le 1 - c/\sqrt{\kappa}$. Then, \tc{black}{for any fixed noise magnitude $\sigma$, the variance amplification  satisfies 
			\begin{subequations}
			\begin{align}\label{eq.lowerboundJStara}
			\dfrac{J}{\sigma^2} 
			\;=\;  \Omega(\kappa^{\frac{3}{2}}).
			\end{align}
Furthermore, if $\sigma=\alpha$, i.e., when the only source of uncertainty is a noisy gradient, we have			
	\begin{align}
	\label{eq.lowerboundJStarb}
			J
			\;=\;
			\Omega(\dfrac{\kappa^{\frac{3}{2}}}{L^2}).
			\end{align}
		\end{subequations}
		}
	\end{mythm}
\begin{proof}
	For the heavy-ball method, the result follows from combining Theorem~\ref{thm.tradeOffHB} with the inequality $1-\rho\ge{c}/{\sqrt{\kappa}}$. For Nesterov's method, the proof is provided in Appendix~\ref{sec.proofOflowerBoundJstar}.
\end{proof}

For problems with $n \ll \kappa$, we recall that the variance amplification of gradient descent with conventional values of parameters scales as $O(\kappa)$; see Theorem~\ref{thm.summary}. Irrespective of the choice of parameters $\alpha$ and $\beta$, this result in conjunction with Theorem~\ref{th.lowerBoundJstar} demonstrates that acceleration cannot be achieved without increasing the variance amplification $J$ by a factor of $\Omega(\sqrt{\kappa})$.

		\vspace*{-2ex}
\section{Application to distributed computation over undirected networks}
	\label{sec:distributed}
	
	Distributed computation over networks has received significant attention in optimization, control systems, signal processing, communications, and machine learning communities. In this problem, the goal is to optimize an objective function (e.g., for the purpose of training a model) using multiple processing units that are connected over a network. Clearly, the structure of the network (e.g., node dynamics and network topology) may impact the performance (e.g., convergence rate and noise amplification) of any optimization algorithm. As a first step toward understanding the impact of the network structure on performance of noisy first-order optimization algorithms, in this section, we examine the standard distributed consensus problem.

The consensus problem arises in applications ranging from social networks, to distributed computing networks, to cooperative control in multi-agent systems. In the simplest setup, each node updates a scalar value using the values of its neighbors such that they all agree on a single consensus value. Simple updating strategies of this kind can be obtained by applying a first-order algorithm to the convex quadratic problem
	\be 	
	\minimize\limits_x
	~
	\dfrac{1}{2}
	\, 
	x^T \Lap \, x
	\label{eq.Lmin}
	\ee 
where $ \Lap= \Lap^T\in\R^{n\times n}$ is the Laplacian matrix of the graph associated with the underlying undirected network and $x \in \bbR^n$ is the vector of node values. 

The graph Laplacian matrix $\Lap\succeq 0$  has a nontrivial null space that consists of the minimizers of problem~\eqref{eq.Lmin}. In the absence of noise, for gradient descent and both of its accelerated variants, it is straightforward to verify that the projections $v^t$ of the iterates $x^t$ onto the null space of $\Lap$ remain constant ($v^t=v^0$, for all $t$) and also that $x^t$ converges linearly to $v^0$. In the presence of additive noise, however, $v^t$ experiences a random walk which leads to an unbounded variance of $x^t$ as $t \to \infty$. Instead, as described in~\cite{xiaboykim07}, the performance of algorithms in this case can be quantified by examining
	$
	\bar{J}\DefinedAs\lim_{t \, \to \, \infty}\EX \left( \norm{x^t-v^t}^2 \right).
	$
For connected networks, the null space of $\Lap$  is given by $\Null(\Lap) =  \{c\one\,|\, c \in \R\}$ and 
	\be
	\bar{J} \; = \; \lim_{t \, \to \, \infty}\EX\left( \norm{x^t-(\one^Tx^t/n)\one}^2 \right)
	\label{eq.Jbar}
	\ee
quantifies the mean-squared deviation from the network average, where $\one$ denotes the vector of all ones, i.e., $ \one \DefinedAs [ \, 1 \; \cdots \;1 \, ]^T $. Finally, it is straightforward to show that $\bar{J}$ can also be computed using the formulae in Theorem~\ref{th.varianceJhat}  by summing over the non-zero eigenvalues of $\Lap$.  

In what follows, we consider a class of networks whose structure allows for the explicit evaluation of the eigenvalues of the Laplacian matrix $\Lap$. For $d$-dimensional torus networks, fundamental performance limitations of standard consensus algorithms in continuous time were established in~\cite{bamjovmitpat12}, but it remains an open question if gradient descent and its accelerated variants suffer from these limitations. We utilize such torus networks to demonstrate that standard gradient descent exhibits the same scaling trends as consensus algorithms studied in~\cite{bamjovmitpat12} and that, in lower spatial dimensions, acceleration always increases variance amplification. 	
	
	\vspace*{-2ex}
	\subsection{Explicit formulae for $d$-dimensional torus networks}
	
	We next examine the asymptotic scaling trends of the performance metric $\bar{J}$ given by~\eqref{eq.Jbar} for large problem dimensions $n \gg 1$ and highlight the subtle influence of the distribution of the eigenvalues of $\Lap$ on the variance amplification for $d$-dimensional torus networks. Tori with nearest neighbor interactions generalize one-dimensional rings to higher spatial dimensions. Let $ \Z_{{\m}} $ denote the group of integers modulo $ {\m} $. A $d$-dimensional torus $ \T^d_{\m} $ consists of $ n \DefinedAs \m^d $ nodes denoted by $ v_a $ where $ a \in \Z_{\m}^d$ and the set of edges 	
	$ 
	\{\{v_a ~ v_b\} \, | \, \norm{ a - b } = 1 \!\! \mod {\m} \};
	$
nodes $ v_a $ and $ v_b $ are neighbors if and only if $ a $ and $ b $ differ exactly at a single entry by one. For example, $ \T_{\m}^1 $ denotes a ring with $ n = {\m} $ nodes and $ \T_{\m}^5 $ denotes a five dimensional torus with $ n = \m^5 $ nodes.
	
	The multidimensional discrete Fourier transform can be used to determine the eigenvalues of the Laplacian matrix $\Lap$ of a $d$-dimensional torus $ \T^d_{\m}$,
	\begin{align}
	\label{eq:eigTorusExpression}
	\lambda_i
	\; = \; 
	\sum_{l \, = \, 1}^{d}
	\, 
	2 \left(
	1 \, - \, \cos \tfrac{2 \pi i_l}{{\m}} 
	\right),
	~~
	i_l \, \in \, \Z_{\m}
	\end{align}
	where $i \DefinedAs (i_1, \ldots, i_d ) \in  \Z_{\m}^d$. We note that $ \lambda_0=0 $ is the only zero eigenvalue of $\Lap$ with the eigenvector $\one$ and that all other eigenvalues are positive. Let $\kappa \DefinedAs \lambda_{\max}/\lambda_{\min}$ be the ratio of the largest and smallest nonzero eigenvalues of $\Lap$. A key observation is that, for ${\m} \gg 1$,  
	\be
\kappa 
\; = \; 
\Theta( \dfrac{2}{1 \, - \, \cos \tfrac{2 \pi}{{\m}}} )
\; = \;
\Theta (\m^2)
\; = \;
\Theta (n^{2/d}). 
\label{eq.kappa_d}
\ee
  This is because  
	$
	\lambda_{\min} 
	= 
	2 d
	\,
	(
	1 - \cos \, ( {2\pi}/{{\m}} )
	)
	$
goes to zero as $ {\m} \to \infty $, and the largest eigenvalue of $ \Lap $,
	$
	\lambda_{\max} 
	= 
	2d
	\,
	(
	1 - \cos\, ( 2\pi \lfloor \tfrac{{\m}}{2} \rfloor/{{\m}} )
	),
	$
	is equal to $4\,d$ for even ${\m}$ and it approaches $ 4\,d $ from below for odd ${\m}$.  
	
	As aforementioned, the performance metric $\bar{J}$ can be obtained by 
	\[
	\bar{J}
	\; = \;
	\sum_{0 \, \neq \, i \, \in \, \Z_{\m}^d} 
	\;
	\hat{J} (\lambda_i) 
	\]
	where $\hat{J}(\lambda)$ for each algorithm is  determined in Theorem~\ref{th.varianceJhat} and $\lambda_i$ are the non-zero eigenvalues of $\Lap$. The next theorem characterizes the asymptotic value of the network-size normalized mean-squared deviation from the network average, $\bar{J}/n$, for a fixed spatial dimension $d$ and condition number $\kappa \gg 1$. This result is obtained using analytical expression~\eqref{eq:eigTorusExpression} for the eigenvalues of the Laplacian matrix $ \Lap$. 
	
	\begin{mythm}
	\label{cor:orders}
		Let $ \Lap \in \R^{n\times n} $ be the graph Laplacian of the $d$-dimensional undirected torus $ \T_{\m}^d $ with $n = \m^d \gg 1$ nodes. For convex quadratic optimization problem~\eqref{eq.Lmin}, the network-size normalized performance metric $\bar{J}/n$ of noisy first-order algorithms with the parameters provided in Table~\ref{tab:rates} \tc{black}{ and $\sigma=1$,} is determined by
		\vsp 
		\begin{center}
			\begin{tabular}{l|@{\;}cccccc}
				& & {$ d=1 $} & {$ d=2 $} & {$ d=3 $} & {$ d=4 $} & {$ d=5 $}
				\\[0.0cm]\hline\\&&&&&&\\[-1.15 cm]
				{Gradient}  & & {$ \Theta (\sqrt{\kappa}) $}& {$ \Theta (\log \, \kappa ) $} & {$ \Theta ( 1 ) $} & {$ \Theta ( 1 ) $} & {$ \Theta ( 1 ) $}
				\\[0.15cm]
				{Nesterov} & & {$ \Theta ( \kappa )$}& {$ \Theta ( \sqrt{\kappa}\log\,\kappa )$} & {$ \Theta ( \kappa^{\frac{1}{4}} )$} & {$ \Theta ( \log \, \kappa ) $} & {$ \Theta ( 1 ) $}
                \\[0.15cm]
				{Polyak} & & {$ \Theta ( \kappa  )$}& {$ \Theta ( \sqrt{\kappa}\,\log\,\kappa )$} & {$ \Theta ( \sqrt{\kappa} )$} & {$ \Theta ( \sqrt{\kappa} ) $} & {$ \Theta ( \sqrt{\kappa} ) $}
			\end{tabular}
		\end{center}
		\vsp 
		where $ \kappa = \Theta (n^{2/d})$ is the condition number of $\Lap$ given in \eqref{eq.kappa_d}.
	\end{mythm}
	\begin{proof}
		See Appendix~\ref{app.networks}.
	\end{proof}

Theorem~\ref{cor:orders} demonstrates that the variance amplification of  gradient descent is equivalent to that of the standard consensus algorithm studied in~\cite{bamjovmitpat12} and that, in lower spatial dimensions, acceleration always negatively impacts the performance of noisy algorithms. Our results also highlight the subtle influence of the distribution of the eigenvalues of $\Lap$ on the variance amplification. For rings (i.e., $d=1$), lower bounds provided in Theorem~\ref{thm.JBoundsQuad} capture the trends that our detailed analysis based on the distribution of the entire spectrum of $\Lap$ reveals. In higher spatial dimensions, however, the lower bounds that are obtained using only the extreme eigenvalues of $\Lap$ are conservative. Similar conclusion can be made about the upper bounds provided in Theorem~\ref{thm.JBoundsQuad}. This observation demonstrates that the na\"{i}ve bounds that result only from the use of the extreme eigenvalues can be overly conservative.

We also note that gradient descent significantly outperforms Nesterov's accelerated algorithm in lower spatial dimensions. In particular, while $\bar{J}/n$ becomes network-size-independent for $d = 3$ for gradient descent, Nesterov's algorithm reaches ``critical connectivity'' only for $d = 5$. On the other hand, in any spatial dimension, there is no network-size independent upper bound on $\bar{J}/n$  for the heavy-ball method. These conclusions could not have been reached without performing an in-depth analysis of the impact of all eigenvalues on performance of noisy networks with $n \gg 1$ and $\kappa \gg 1$. 	
	
		\vspace*{-3ex}
	\section{Concluding remarks}
	\label{sec: conclusion}
	
	We study the robustness of noisy first-order algorithms for smooth, unconstrained, strongly convex optimization problems. Even though the underlying dynamics of these algorithms are in general nonlinear, we establish upper bounds on noise amplification that are accurate up to constant factors. For quadratic objective functions, we provide analytical expressions that quantify the effect of all eigenvalues of the Hessian matrix on variance amplification. We use these expressions to establish lower bounds demonstrating that although the acceleration techniques improve the convergence rate they significantly amplify noise for problems with large condition numbers. In problems of bounded dimension $n \ll \kappa$, the noise amplification increases from $O(\kappa)$ to $\Omega(\kappa^{{3}/{2}})$ when moving from standard gradient descent to accelerated algorithms. We specialize our results to the problem of distributed averaging over noisy undirected networks and also study the role of network size and topology on robustness of accelerated algorithms. Future research directions include (i) extension of our analysis to multiplicative and correlated noise; and (ii) robustness analysis of broader classes of optimization algorithms.  

	\vspace*{-2ex}	
	\appendix
	
	\vspace*{-2ex}
	\subsection{Quadratic problems}
	\label{app.Quadratic}
	\begin{proof}[Proof of Theorem~\ref{th.varianceJhat}]
		For gradient descent,  $\hat{A}_i = 1 - \alpha \lambda_i$ and $\hat{B}_i =1$ are  scalars and the solution to~\eqref{eq.LyapDiag} is given by
			\[
			\hat{P}_i
			\, \DefinedAs \,
			\tc{black}{\sigma^2} p_i
			\, = \,
			\dfrac{\tc{black}{\sigma^2}}{1 \, - \, (1 \, - \, \alpha \lambda_i)^2}
			\, = \, 
			\dfrac{\tc{black}{\sigma^2}}
			{\alpha \lambda_i (2 \, - \, \alpha \lambda_i )}.
			\]
		For the accelerated methods, we note that for any $ \hat{A}_i $ and $ \hat{B}_i $ of the form 
			\[
			\hat{A}_i
			\, = \,
			\tbt{0}{1}{a_i}{b_i},
			~
			\hat{B}_i
			\,=\,
			\tbo{0}{1}
			\]
		the solution $\hat{P}_i$ to Lyapunov equation~\eqref{eq.LyapDiag} is given by
			\[
			\hat{P}_i
			\, = \,
			\tc{black}{\sigma^2}\tbt{p_i}{b_i p_i/(1 - a_i)}{b_i p_i/(1 - a_i)}{p_i}
			\]	
		where 
		\be
		p_i
		\, \DefinedAs \,
		\dfrac{a_i \, - \, 1}
		{
			(a_i \, + \, 1)
			(b_i \, + \, a_i \, - \, 1)
			(b_i \, - \, a_i \, + \, 1)
		}.
		\label{eq:pi}
		\ee
		The parameters $a_i$ and $b_i$ for Nesterov's algorithm are 
		$
		\{
		a_i 
		= 
		-\beta( 1 - \alpha \lambda_i)
		$;
		$
		b_i 
		= 
		(1 + \beta)(1 - \alpha \lambda_i)
		\}
		$        
		and for the heavy-ball method we have
		$
		\{
		a_i 
		=
		-\beta
		$;
		$
		b_i 
		= 
		1 + \beta - \alpha \lambda_i
		\}.
		$
		Now, since $ \hat{C_i}=1 $ for gradient descent and 
		$
		\hat{C}_i
		=
		[ \, 1 \;\, 0 \, ]
		$
		for the accelerated algorithms, it follows that  for all three algorithms we have
			$
			\hat{J}(\lambda_i)
			\DefinedAs  
			\trace \, (\hat{C}_i \hat{P}_i \hat{C}_i^T ) 
			= 
			\tc{black}{\sigma^2} p_i.
			$		
		Finally, if we use the expression for $p_i$ for gradient descent and substitute for $ a_i $ and $ b_i $ in~\eqref{eq:pi} for the accelerated algorithms, we obtain the expressions for $\hat{J}$ in the statement of the theorem. 
	\end{proof}
	\begin{proof}[Proof of Proposition~\ref{prop.relationJhat}]
		To show that ${\hat{J}_{\na}(\lambda)}/{\hat{J}_\gd(\lambda)}$ is a decreasing function of $\lambda\in[\mf,\Lf]$, we split this ratio into the sum of two homographic functions
		$
		{\hat{J}_{\na}(\lambda)}/{\hat{J}_\gd(\lambda)}= \sigma_1(\lambda) + \sigma_2(\lambda),
		$
where
		\begin{align}\label{eq.Sigma1Sigma2}
		\sigma_1(\lambda) 
		\;\DefinedAs\;
		\dfrac{
			4\alpha_\gd\beta	
		}{\alpha_\na(3\beta+1)(1-\beta)}
		\dfrac{1-\tfrac{\alpha_\gd}{2}\lambda}{1 + \tfrac{\alpha_\na\beta}{1-\beta}\lambda},
		\quad
		\sigma_2(\lambda) 
		\;\DefinedAs\;
		\dfrac{\alpha_\gd}{\alpha_\na(3\beta+1)}
		\dfrac{1-\tfrac{\alpha_\gd}{2}\lambda}{1- \tfrac{\alpha_\na(2\beta+1)}{2+2\beta}\lambda}.
		\end{align}
		Now, if we substitute the parameters provided in Table~\ref{tab:rates} into~\eqref{eq.Sigma1Sigma2}, it follows that the signs of the derivatives $\mrd\sigma_1/\mrd\lambda$ and $\mrd\sigma_2/\mrd\lambda$ satisfy 
		\begin{align*}
		\sign \, (\dfrac{\mrd\,\sigma_1}{\mrd\lambda}) 
		&\;=\;
		\sign \, (-\tfrac{\alpha_\na\beta}{1-\beta}-\tfrac{\alpha_\gd}{2})
		\;=\;
		\sign \, (-\frac{\kappa+\kappa\,\sqrt{3\,\kappa+1}+\sqrt{3\,\kappa+1}-1}{\mf \left(3\,\kappa+1\right) \left(\kappa+1\right)}
		)
		\; < \;
		0, 
		~~
		\forall \, \kappa > 1\\
		\sign \, (\dfrac{\mrd\,\sigma_2}{\mrd\lambda}) 
		&\;=\;
		\sign \, (\tfrac{\alpha_\na(2\beta+1)}{2+2\beta}-\tfrac{\alpha_\gd}{2})
		\;=\;
		\sign \, (
		-\frac{2 \left(\kappa-\sqrt{3\,\kappa+1}+1\right)}{\mf \left(3\,\kappa+1\right)^{3/2} \left(\kappa+1\right)})
		\; < \;
		0, 
		~~
		\forall \, \kappa > 1.
		\end{align*}
Furthermore, since the critical points of the functions $\sigma_1 (\lambda)$ and $\sigma_2 (\lambda)$ are outside the interval $[\mf, \Lf]$,
		\begin{align*}
		\lambda_{\text{crt}1}
		&\;=\;
		-\frac{\mf(3\kappa+1)}{\sqrt{3\kappa+1}-2}
		\; < \;
		0
		\; < \;
		\mf,
		~~
		\lambda_{\text{crt}2}
		\;=\;
		\frac{\mf\left(3\kappa+1\right) \sqrt{3\kappa+1}}{3\,\sqrt{3\kappa+1}-2}
		\; > \;
		\mf \,\kappa
		\; =\;
		\Lf
		\end{align*}
we conclude that both $\sigma_1$ and $\sigma_2$ are decreasing functions over the interval $[\mf,\Lf]$.
We next prove~\eqref{eq.GDExtrem} and~\eqref{eq.Jmax}.

It is straightforward to verify that both $\hat{J}_\gd(\lambda)$ and $\hat{J}_\na(\lambda)$ are quasi-convex functions over the interval $[\mf,\Lf]$ and that the respective minima are attained at the critical point $\lambda=1/\alpha$. Quasi-convexity also implies
\begin{align}\label{eq.quasiConvexity}
	\max\limits_{\lambda \, \in \, [\mf, \Lf]} \hat{J}(\lambda) 
	\;=\;
	 \max \, \{\hat{J}(\mf), \hat{J}(\Lf)\}.
\end{align}
Now, letting $\alpha=2/(\Lf+\mf)$ in the expression for $\hat{J}_\gd$ gives $\hat{J}_\gd(\mf)=\hat{J}_\gd(\Lf)=(\kappa+1)^2/(4\kappa)$ which in conjunction with~\eqref{eq.quasiConvexity} complete the proof for~\eqref{eq.GDExtrem}. Finally, since the ratio $\hat{J}_\na(\lambda)/\hat{J}_\gd(\lambda)$ is decreasing, we have 
	$
	{\hat{J}_\na(\Lf)}/{\hat{J}_\gd(\Lf)} \le {\hat{J}_\na(\mf)}/{\hat{J}_\gd(\mf)}.
	$
Combining this inequality with $\hat{J}_\gd(\mf)=\hat{J}_\gd(\Lf)$ and~\eqref{eq.quasiConvexity} completes the proof of~\eqref{eq.Jmax}. 
\end{proof}

\begin{proof}[Proof of Theorem~\ref{thm.RelationJnagd}]
	From Proposition~\ref{prop.relationJhat}, it follows that 
	\begin{subequations}
		\begin{align}\label{eq.NesToGD}	
		\dfrac{\hat{J}_{\na}(\Lf)}{\hat{J}_{\gd}(\Lf) 	}
		\;\le\;
		\dfrac{\hat{J}_{\na}(\lambda_i)}{\hat{J}_{\gd} (\lambda_i)}
		\;\le \;
		\dfrac{\hat{J}_{\na}(\mf)}{\hat{J}_{\gd}(\mf) }
		\end{align}
		for all $\lambda_i$ and 
		\begin{align}\label{eq.sumGDLes}	
		\sum_{ i=1 }^{n-1} \hat{J}_\gd(\lambda_i)
		\;\le\;
		(n-1
		)\hat{J}_\gd(\mf) = (n-1
		)\hat{J}_\gd(\Lf).
		\end{align}
	\end{subequations}
	For the upper bound, we have
	\be
	\begin{array}{rcccccl}
		\dfrac{J_\na}{J_\gd}
		&\!\!\!=\!\!\!&
		\dfrac{\sum_{ i=1 }^n \hat{J}_\na(\lambda_i)}{\sum_{ i=1 }^n\hat{J}_\gd(\lambda_i)}
		&\!\!\!\le\!\!\!&
		\dfrac{\hat{J}_\na(\Lf) 
			\;+\;
			\tfrac{\hat{J}_\na(\mf)}{\hat{J}_\gd(\mf)}
			\sum_{ i=1 }^{n-1} \hat{J}_\gd(\lambda_i)}{\hat{J}_\gd(\Lf) \;+\;\sum_{ i=1 }^{n-1}\hat{J}_\gd(\lambda_i)}
		&\!\!\!\le\!\!\!&
		\dfrac{\hat{J}_\na(\Lf)
			\; +\;
			(n\,-\,1)\hat{J}_\na(\mf)}
		{\hat{J}_\gd(\Lf)
			\; +\;
			(n\,-\,1)\hat{J}_\gd(\mf) }
		\non
	\end{array}
	\ee
	where the first inequality follows from~\eqref{eq.NesToGD}. The second inequality can be verified by multiplying both sides with the product of the denominators and using $\hat{J}_\gd(\mf)=\hat{J}_\gd(\Lf)$, $\hat{J}_\na(\mf)\ge \hat{J}_\na(\Lf)$, and~\eqref{eq.sumGDLes}.
	Similarly, for the lower bound we can write
	\be
	\begin{array}{rcccccl}
		\dfrac{J_\na}{J_\gd}
		&\!\!\!=\!\!\!&
		\dfrac{\sum_{ i=1 }^n \hat{J}_\na(\lambda_i)}{\sum_{ i=1 }^n\hat{J}_\gd(\lambda_i)}
		&\!\!\!\ge\!\!\!&
		\dfrac{\hat{J}_\na(\mf) 
			\;+\;
			\tfrac{\hat{J}_\na(\Lf)}{\hat{J}_\gd(\Lf)}
			\sum_{ i=2 }^n \hat{J}_\gd(\lambda_i)}{\hat{J}_\gd(\mf) \;+\;\sum_{ i=2 }^n\hat{J}_\gd(\lambda_i)}
		&\!\!\!\ge\!\!\!&
		\dfrac{\hat{J}_\na(\mf)
			\; +\;
			(n\,-\,1)\hat{J}_\na(\Lf)}
		{\hat{J}_\gd(\mf)
			\; +\;
			(n\,-\,1)\hat{J}_\gd(\Lf) }.
			\non
	\end{array}
	\ee
	Again, the first inequality follows from~\eqref{eq.NesToGD} and  the second inequality can be verified by multiplying both sides with the product of the denominators and using $\hat{J}_\gd(\mf)=\hat{J}_\gd(\Lf)$, $\hat{J}_\na(\mf)\ge \hat{J}_\na(\Lf)$, and~\eqref{eq.sumGDLes}.
\end{proof}

\begin{proof}[Proof of the bounds in~\eqref{eq.boundsLambdaLm}]
	From Proposition~\ref{prop.relationJhat}, we have
\begin{align*}
\hat{J}_\na(\mf) 
\;=\;
\frac{b^4 \! \left(b^2-2\,b+2\right)}{32\,{\left(b-1\right)}^3}
,\quad
\hat{J}_\na(\Lf) 
\;=\; 
\frac{9\,b^4 \! \left(b^2+2\,b-2\right)}{32\,\left(b^2-1\right)\left(2\,b-1\right)\left(b^2-b+1\right)}
\end{align*}
where $b\DefinedAs \sqrt{3\,\kappa+1} > 2$. The upper and lower bounds on $\hat{J}_\na(\mf)$ are obtained as follows
\begin{align*}
\dfrac{b^3}{32} 
\; \leq \;
\dfrac{b^4 ((b - 1)^2+1)}{32 \left( b - 1\right)^3} 
\;=\;
\hat{J}_\na(\mf)
\;\le\;
\frac{b^3 \! \left( b+c_1(b) \right) \! \left(b^2-2\,b+2+c_2 (b)\right)}{32\,{\left(b-1\right)}^3}= \dfrac{b^3}{8}
\end{align*}
where the positive quantities $c_1(b) \DefinedAs b-2 $ and $c_2 (b) \DefinedAs b^2-2b$ are added to yield a simple upper bound. Similarly, for $\hat{J}_\na(\Lf)$ we have
\begin{align*}
\dfrac{9b}{64}
&\;=\;
\dfrac{(9/32)\,b^4 (b^2+2\,b-2)}{ ((b^2-1)+1)\,((2\,b-1)+1)\left(b^2-b+1 +c_3 (b) \right)}
\;\le\; 
\hat{J}_\na(\Lf)\\[0.15 cm]
\dfrac{9b}{8}
&\;=\;
\frac{(9/32)\,b^4 \! \left(b^2+2\,b-2 + c_4 (b) \right)}{\left(b^2-1\right) \! \left(2\,b-1 -c_5(b) \right) \! \left(b^2-b+1 - c_6 (b) \right)}
\;\ge\;
\hat{J}_\na(\Lf)
\end{align*}
where the positive quantities $c_3 (b) \DefinedAs 3b-3$, $c_4 (b) \DefinedAs b^2-2 b$, $c_5 (b) \DefinedAs b-1$, and $c_6 (b)\DefinedAs (1/2) b^2-b+1$ are introduced to obtain tractable bounds. 
\end{proof}
	
	\vspace*{-2ex}
	\subsection{General strongly convex problems}
	\label{app.General}
	
	\begin{proof}[Proof of Lemma~\ref{lem:LMI-GD}]
		Let us define the positive semidefinite function $ V(\psi) \DefinedAs \psi^TX\psi $ and let $\eta \DefinedAs [ \, \psi^T \; u^T \, ]^T$. Using LMI~\eqref{LMI} and~\eqref{eq:jointIneqApp11}, we can write
		\begin{align*}
		\norm{z^t}^2
		&\;=\;
		(\eta^t)^{T}\tbt{C_z^T C_z}{0}{0}{0}\eta^t
		\\[0.cm]
		&\;\le\;
		-(\eta^t)^{T}
		\tbt{A^T X\, A - X}{A^T X\, B_u}{B_u^T\, X\, A}{B_u^T\, X \,B_u}\eta^t
		\;-\;
		\lambda\,
		(\eta^t)^{T} 
		\tbt{C_y^T}{0}{0}{I} 
		\Pi
		\tbt{C_y}{0}{0}{I}\eta^t
		\\[0.cm]
		&\;=\; (\eta^t)^{T}\bigg(\!\tbt{X}{0}{0}{0}-\tbo{A^T}{B_u^T}X \tbo{ A^T}{B_u^T}^T\bigg)\eta^t
		\;-\;
		\lambda \tbo{y^t}{u^t}^T\!\Pi\tbo{y^t}{u^t}
		\\[.cm]
		&\;\le\;
		V(\psi^t)\;-\;V(\psi^{t+1}) \;+\; 2\tc{black}{\sigma}(\psi^t)^TA^TX\,B_w\,w^t
		\;+\;
		\tc{black}{\sigma^2}(w^t)^TB_w^T\,X\,B_w\,w^t \;+\; 2\tc{black}{\sigma}(u^t)^TB_u^T\,X\,B_w\,w^t.
		\end{align*}
Since $w^t $ is a zero-mean white input with identity covariance which is independent of $ u^t $ and $ x^t $, if we take the average of the above inequality over $ t $ and expectation over different realizations of $w^t$, we obtain
		\be
		\ba{rcl}
		\ds \dfrac{1}{\bar{T}} \,  \sum_{t \, = \, 1}^{\bar{T}}\EX \left(\norm{z^t}^2 \right)
		& \!\!\! \le \!\!\! &
		\dfrac{1}{\bar{T}} 
		\,
		\EX 
		\left( 
		V(\psi^1) \,-\, V(\psi^{\bar{T}+1}) 
		\right)
		\, + \;
		\tc{black}{\sigma^2}\trace \, (B_w^TXB_w)
		\ea
		\non
		\ee
Therefore, letting $ \bar{T} \rightarrow \infty $ and using $ X\succeq0 $ lead to
		$
		J
		\le
		\tc{black}{\sigma^2}\trace  \,(B_w^T\,X\,B_w)$,
		which completes the proof.
	\end{proof}
	
	In order to prove Lemma~\ref{lem.generalNester}, we present a technical lemma which along the lines of results of~\cite{fazribmor18} provides us with an upper bound on the difference between the objective value at two consecutive iterations.
	\begin{mylem}
		\label{lem.Mahyar}
		Let $f\in\mathcal{F}_{\mf}^{\Lf}$ and $\kappa\DefinedAs\Lf/\mf.$ Then, Nesterov's accelerated method, with the notation introduced in Section~\ref{sec.general}, satisfies
		\begin{align*}
		f(x^{t+2}) \; - \; f(x^{t+1}) 
		\; \le \;
		&
		~
		\dfrac{1}{2}\left(  N_1
		\tbo{\psi^t}{u^t} \!+\! \tc{black}{\tbo{\sigma w^t}{0}} \right)^T\tbt{\Lf\,I}{I}{I}{0}\left(  N_1
		\tbo{\psi^t}{u^t} \!+\! \tc{black}{\tbo{\sigma w^t}{0}} \right) 
		\; +
		\\
		&
		~
		\dfrac{1}{2}\left(   N_2 \tbo{\psi^t}{u^t} \right)^T \,\tbt{-\mf\,I}{I}{I}{0}\, \left(   N_2 \tbo{\psi^t}{u^t} \right)
		\end{align*}
		where $N_1$ and $N_2$ are defined in Lemma~\ref{lem.generalNester}.
		
	\end{mylem}
	\begin{proof}
		For any $f\in\mathcal{F}_\mf^\Lf$, the Lipschitz continuity of $\nabla f$ implies
		\be
		\label{ineqTemp1}
		\ba{l}
		f(x^{t+2}) \,-\, f(y^t) 
		\; \le\;
		\dfrac{1}{2}\tbo{x^{t+2}\,-\,y^t}{\nabla f(y^t)}^T
		\tbt{\Lf\,I}{I}{I}{0}
		\tbo{x^{t+2}\,-\,y^t}{\nabla f(y^t)}
		\ea
		\ee
and the strong convexity of $f$ yields
		\be
		\label{ineqTemp2}
		\ba{l}
		f(y^{t}) \,-\, f(x^{t+1})
		\;\le\;
		\dfrac{1}{2}\tbo{y^{t}\,-\,x^{t+1}}{\nabla f(y^t)}^T
		\tbt{-\mf\,I}{I}{I}{0} \tbo{y^{t}\,-\,x^{t+1}}{\nabla f(y^t)}.
		\ea
		\ee
		Moreover,  the state and output equations in~\eqref{eq.ss}  lead to
			\begin{align}\label{eq.SubstutyteTemp3}
			\tbo{x^{t+2}\,-\,y^t}{\nabla f(y^t)} 
			&\;=\;
			 N_1
			\tbo{\psi^t}{u^t} \,+\, \tc{black}{\tbo{\sigma w^t}{0}}, 
			~~
			\tbo{y^{t}\,-\,x^{t+1}}{\nabla f(y^t)} 
			\;=\;
			 N_2 \tbo{\psi^t}{u^t}.
			\end{align}
		Summing up inequalities~\eqref{ineqTemp1} and~\eqref{ineqTemp2} and substituting for $\tbo{x^{t+2}\,-\,y^t}{\nabla f(y^t)}$ and $\tbo{x^{t+2}\,-\,y^t}{\nabla f(y^t)}$ from~\eqref{eq.SubstutyteTemp3} completes the proof.
	\end{proof}
	\begin{proof}[Proof of Lemma~\ref{lem.generalNester}]
		Let us define the positive semidefinite function $ V(\psi) \DefinedAs \psi^TX\psi $ and let  $\eta \DefinedAs [ \, \psi^T \;\, u^T \, ]^T$. Similar to the first part of the proof of Lemma~\ref{lem:LMI-GD}, we can use LMI~\eqref{LMI2} and inequality~\eqref{eq.IQC} to write
		\begin{align}\nonumber
		\norm{z^t}^2
		\;\le\;
		& 
		~
		V(\psi^t)\,-\,V(\psi^{t+1}) \,+\, 2\tc{black}{\sigma}(\psi^t)^TA^TX\,B_w\,w^t
		\;+\;
		\tc{black}{\sigma^2}(w^t)^TB_w^T\,X\,B_w\,w^t 
		\;+\; 2\tc{black}{\sigma}(u^t)^TB_u^T\,X\,B_w\,w^t
		~ - 
		\\[0.15cm]\label{ineqTemp5}
		&
		~
		(\eta^t)^T M \, \eta^t.
		\end{align}
From Lemma~\ref{lem.Mahyar}, it follows that
		\begin{equation}
		\label{eq.lowerBoundOnetaTMeta}
		(\eta^t)^T M \, \eta^t 
		\;\ge\;
		2\left(f(x^{t+2}) \,-\, f(x^{t+1})\right) \,-\, \tc{black}{\sigma^2} \Lf\,\norm{w^t}^2
		\;-\;
		2 \tbo{\tc{black}{\sigma w^t}}{0}^T\tbt{\Lf\,I}{I}{I}{0}N_1 \eta^t.
		\end{equation}	
Now, combining inequalities~\eqref{ineqTemp5} and~\eqref{eq.lowerBoundOnetaTMeta} yields
		\begin{align}\nonumber
		\norm{z^t}^2
		\;\le\;&
		~
		V(\psi^t)\,-\,V(\psi^{t+1}) \,+\, 2 \tc{black}{\sigma}(\psi^t)^TA^TX\,B_w\,w^t
		\,+\, \tc{black}{\sigma^2}(w^t)^TB_w^T\,X\,B_w\,w^t \, + \, 2 \tc{black}{\sigma}(u^t)^TB_u^T\,X\,B_w\,w^t
		\; -
		\\[-0.05cm]\label{eq.scalability}
		&
		~
		2\,\lambda_2\left(f(x^{t+2}) \,-\, f(x^{t+1})\right) 
		\,+\, 
		\lambda_2 \tc{black}{\sigma^2} \Lf\norm{w^t}^2
		\,+\, 
		2 \lambda_2 \tbo{\tc{black}{\sigma w^t}}{0}^T\tbt{\Lf\,I}{I}{I}{0}N_1 \eta^t.
		\end{align}
Since $w^t $ is a zero-mean white input with identity covariance which is independent of $ u^t $ and $ x^t $, taking the expectation of the last inequality yields
		\begin{equation*}
		\EX
		\left( \norm{z^t}^2 \right)
		\;\le\;
		\EX \left( V(\psi^t)-V(\psi^{t+1}) \right) 
		+ \, 
		 \tc{black}{\sigma^2}\trace \, (B_w^T\,X\,B_w) 
		\, + \, 
		2 \, \lambda_2 
		\EX \left( f(x^{t+1}) \,-\, f(x^{t+2}) \right) 
		+ \,
		n\,\tc{black}{\sigma^2}\Lf\,\lambda_2
		\end{equation*}
and taking the average over the first $\bar{T}$ iterations results in
		\begin{equation*}
		\dfrac{1}{\bar{T}}
		\sum_{t \, = \, 1}^{\bar{T}} \EX
		\left( \norm{z^t}^2 \right)
		\, \le \; 
		\dfrac{1}{\bar{T}} \EX 
		\!
		\left( V(\psi^1)\,-\,V(\psi^{\bar{T}+1}) \right)
		\;+\;
		 \tc{black}{\sigma^2}\trace \, (B_w^T\,X\,B_w)
		\;+\;
		\dfrac{2\,\lambda_2}{\bar{T}}
		\EX
		\left(
		f(x^2)\,-\,f(x^{\bar{T}+2})
		\right)
		\,+\;
		\tc{black}{n\,\sigma^2\Lf\,\lambda_2.}
		\end{equation*}
		Finally, using positive definiteness of the function $V$, strong convexity of the function $f$, and letting $\bar{T}\rightarrow\infty$,  it follows that
			$
			J
			\le \tc{black}{\sigma^2}( 
			n  \Lf \lambda_2
			+ 
			\trace \, (B_w^TX\,B_w) )
			$
as required.
	\end{proof}

	\begin{proof}[Proof of Theorem~\ref{thm.summary}]	
		Using Theorem~\eqref{th.varianceJhat}, it is straightforward to show  that for gradient descent and Nesterov's method with the parameters provided in Table~\ref{tab:ratesGeneral}, the function $f(x)\DefinedAs \tfrac{m}{2} \norm{x}^2$ leads to the largest variance amplification $J$ among the quadratic objective functions within $\mathcal{F}_{\mf}^{\Lf}$. This yields the lower bounds
		\begin{align*}
			q_\gd \, = \, J_\gd \;\le\; J^\star_\gd, \quad q_\na \, = \, J_\na \; \le \; J^\star_\na
		\end{align*}
with $J_\gd$ and $J_\na$ corresponding to $f(x)= \tfrac{m}{2} \norm{x}^2$. We next show that $J_\gd \le q_\gd$.
	 
	  To obtain the best upper bound on $J_\gd$ using Lemma~\ref{lem:LMI-GD}, we minimize $ \trace \, (B_w^TXB_w) $ subject to LMI~\eqref{LMI}, $ X \succeq 0$, and $ \lambda \geq 0$. For gradient descent, if we use representation~\eqref{eq.ss-gd}, then the negative definiteness of the $(1,1)$-block of LMI~\eqref{LMI} implies that
		\be
		X
		\; \succeq \; 
		\dfrac{1}{\alpha\, \mf (2  \, - \, \alpha\, \mf )}
		\, 
		I 
		\; = \; 
		\frac{\kappa^2}{2\kappa-1} \, 
		I.
		\label{eq.Temp1}
		\ee
It is straightforward to show that  the pair 
		\be
		X
		\;=\; 
		\dfrac{\kappa^2}{2\kappa-1} \, I,
		~~
		\lambda
		\;=\; 
		\frac{1-\alpha\,\mf}{\mf(2-\alpha\,\mf)(\Lf-\mf)}
		\label{eq.Xlambda}
		\ee
		is feasible as the LMI~\eqref{LMI} becomes
		\begin{equation*}
		\tbt{0}{0}{0}{\tfrac{- 1}{m^2(2\kappa \, - \, 1)} \, I}
		\;\preceq\;
		0.
		\end{equation*}
		Thus, $X$ and $\lambda$ given by~\eqref{eq.Xlambda} provide a solution to LMI~\eqref{LMI}. Therefore, inequality~\eqref{eq.Temp1} is tight and it provides the best achievable upper bound 
		$$
		J_\gd
		\;\le\;
		\trace \, (B_w^T \, X \, B_w) 
		\;=\;
		\frac{n \kappa^2}{2\,\kappa \, - \, 1}.
		$$		
		Finally, we show
		$
		J_\na
		\le
		4.08 q_\na$
		by finding a sub-optimal feasible point for~\eqref{prob.LMI2}. Let
		$
		X
		\DefinedAs
		\tbt{x_1 I}{x_0 I}{x_0 I}{x_2 I}
		$
		with
		\begin{align*}
		x_{1} 
		&\;\DefinedAs\;
		\dfrac{1}{s(\kappa)}
		\left(
		2\,\kappa^{3.5}
		-
		8\,\kappa^{3}
		+
		11\,\kappa^{2.5}
		+
		5\,\kappa^{2}
		-
		14\,\kappa^{1.5}
		+
		8\,\kappa
		-
		2\,\kappa^{0.5}
		\right)
		\\[.1cm]
		x_{0}
		&\;\DefinedAs\;
		\dfrac{-1}{s(\kappa)}
		\left( 2\,\kappa^{1.5}
		\left( \kappa^{0.5}-1\right)^3
		\left( \kappa^{0.5}+1\right)
		\right)
		\\[.1cm]
		x_{2} 
		&\;\DefinedAs\;
		\dfrac{\kappa^{1.5}}{s(\kappa)}
		\left(
		2\,\kappa^{2}
		-3\,\kappa
		+5\,\kappa^{0.5}
		-2
		\right),
		~~
		s(\kappa) 
		\; \DefinedAs \;
		8\,\kappa^{2}
		-6\,\kappa^{1.5}
		-2\,\kappa
		+3\,\kappa^{0.5}
		-1
		\end{align*}
		and let
		$
		\lambda_1
		\DefinedAs
			(\kappa/\Lf)^2 / 
			(2\kappa - 1 ) 
		$
and
		$
		\lambda_2
		\DefinedAs
		{-x_{0}}/(\Lf s(\kappa)).
		$
		We first show  that  $(\lambda_1,\lambda_2,X)$ is feasible for problem~\eqref{prob.LMI2}. It is straightforward to verify that $s(\kappa)$, $x_{1} s(\kappa)$, $x_{2} s(\kappa)$, and $-x_{0}s(\kappa)$  (which are polynomials of degree less than $7$ in $\sqrt{\kappa}$) are all positive for any $\kappa\ge 1$. Hence, $x_{1}>0$, $x_{2}>0$ and $\lambda_2>0$. It is also easy to see that $\lambda_1>0$ and that the determinant of $X$ satisfies
		\begin{equation*}
		\det(X) 
		\;=\;
		\dfrac{\kappa^{2n}}{s^{2n}(\kappa)}\big(
		28\,\kappa^{3.5}
		- 65\,\kappa^3
		+ 56\,\kappa^{2.5}
		+ 25\,\kappa^{2}
		- 88\,\kappa^{1.5}
		+ 70\,\kappa
		- 26\,\kappa^{0.5}
		+ 4
		\big)^{n}
		> 0,
		~  
		\forall \, \kappa \, \ge \, 1
		\end{equation*}
which yields $X\succeq 0$. Moreover, it can be shown that the left-hand-side of LMI~\eqref{LMI2} becomes
		\begin{equation*}
		\begin{bmatrix}
		\;0 \,&\, 0 \,&\, 0 \\[-0.15cm]
		\;0 \,&\, 0 \,&\, 0 \\[-0.15cm]
		\;0 \,&\, 0 \,&\, -\lambda_1 I 
		\end{bmatrix}
		\;\preceq\;
		0.
		\end{equation*}
		Therefore, the point $(\lambda_1,\lambda_2,X)$ is feasible to problem~\eqref{prob.LMI2}
		and 
		\begin{equation*}
		J_\na
		\;\le\;
		p(\kappa)
		\;\DefinedAs\;
		n\,\Lf\lambda_2 + n\,x_{2} \\[.15 cm]
		\;=\; \dfrac{n}{s(\kappa)}\left(
		4\,\kappa^{3.5}
		- 4\,\kappa^{3}
		- 3\,\kappa^{2.5}
		+ 9\,\kappa^{2}
		- 4\,\kappa^{1.5}
		\right).
		\end{equation*}
		Comparing  $p$ with $q_\na$, it can be verified that, for all $\kappa\ge1$, $
		4.08 q_\na(\kappa) \ge p(\kappa),
		$
which completes the proof.
	\end{proof}

	\subsection{Proof of Theorem~\ref{thm.lowerboundBetaHeavyball}}
\label{sec.Tuning parameters using the whole spectrumApp}
\tc{black}{Without loss of generality, let $\sigma=1$ and }
\begin{align}\label{eq.Gdefintion}
	G
	\,\DefinedAs\,
	\ds{\sum_{i \, = \, 1}^n} \, \max\{\hat{J}(\lambda_i) , \hat{J}(\lambda_i')\}
\end{align} 
where $\lambda_i$  are the eigenvalues of the Hessian of the objective function $f$ and  $\lambda_i'= \mf+\Lf-\lambda_i$ is the mirror image of $\lambda_i$ with respect to $(\mf+\Lf)/2$. Since $J = \sum_i \hat{J}(\lambda_i)$, if $\lambda_i$ are symmetrically distributed over the interval $[\mf,\Lf]$ i.e., $(\lambda_1,\cdots,\lambda_n) = (\lambda_n',\cdots,\lambda_1')$, then for any parameters $\alpha$ and $\beta$ we have
\begin{align}\label{eq.JGJ}
	J 
	\;\le\; G
	\;\le\;
	2J.
\end{align}
Equation~\eqref{eq.JGJ} implies that any bound on $G$ simply carries over to $J$ within an accuracy of constant factors. Thus, we focus on $G$ and establish one of its useful properties in the next lemma   that allows us to prove Theorem~\ref{thm.lowerboundBetaHeavyball}.
\begin{mylem}\label{lem.TuningWholeSpec1}
 	The heavy-ball method with any stabilizing parameter $\beta$ satisfies
	\begin{align}
	\dfrac{2(1+\beta)}{\Lf+\mf}
	\;=\;
	\argmin_\alpha \, \rho(\alpha,\beta)
	\end{align} 
	where $\rho$ is the rate of linear convergence. Furthermore, if the Hessian of the quadratic objective function $f$ has a symmetric spectrum over the interval $[\lambda_1,\lambda_n] = [\mf,\Lf]$, then 
	\begin{align*}
	\dfrac{2(1+\beta)}{\Lf+\mf}
	\;= \;
	\argmin_\alpha \,G(\alpha,\,\beta).
	\end{align*}
\end{mylem}

\begin{proof}
	The linear convergence rate $\rho$  is given by 
	$\rho = \max_{1 \, \le \, i \, \le \, n} \hat{\rho}(\lambda_i)$, 
	where $\hat{\rho}(\lambda)$ is the largest absolute value of the roots of the characteristic polynomial 
	\begin{align*}
	\det (zI-\hat{A}) = z^2+(\alpha\lambda-1-\beta)z + \beta
	\end{align*}
	 associated with the heavy-ball method and the eigenvalue $\lambda$ of the Hessian of the objective function $f$; \tc{black}{See~\eqref{eq.Ahat} for the form of $\hat{A}$. Thus, we have}
	\begin{align*}
	\hat{\rho}(\lambda) 
	\; = \; 
	\left\{
	\ba{ll}
	\sqrt{\beta} & \text{if} \; \Delta<0
	\\[0.1cm]
	\frac{1}{2}|1+\beta-\alpha\lambda|+\frac{1}{2}\sqrt{\Delta} & \text{otherwise}
	\ea
	\right.
	\end{align*}
where $\Delta \DefinedAs (1+\beta-\alpha\lambda)^2-4\beta.$ This can be simplified to
	\begin{align*}
	\hat{\rho} \; = \; 
	\left\{
	\ba{ll}
	\sqrt{\beta}
	&\text{if}\; (1-\sqrt{\beta})^2\le\alpha\lambda\le(1+\sqrt{\beta})^2
	\\[0.1cm]
	\frac{1}{2}|1+\beta-\alpha\lambda|+\frac{1}{2}\sqrt{\Delta}&\text{otherwise.}
	\ea
	\right.
	\end{align*}
	
	It is straightforward to show that $\hat{\rho}$ and $\hat{J}$ \tc{black}{with $\sigma=1$} are explicit quasi-convex functions of $\mu\DefinedAs\alpha\lambda$ which are symmetric with respect to $\mu=1+\beta$. Quasi-convexity of $\hat{\rho}$ yields 
	\begin{align*}
		\rho 
		\; = \; 
		\max \, \{\hat{\rho}(\lambda_1),\hat{\rho}(\lambda_n)\} 
		\; = \; 
		\max \, \{\hat{\rho}(\lambda_1),\hat{\rho}(\lambda_1') \}.
	\end{align*}
	 Let $\alpha^\sharp(\beta) = 2(1+\beta)/(\Lf+\mf)$. For any eigenvalue $\lambda_i$, from the symmetry of the spectrum, we have
	$$
	\alpha^\sharp(\beta)\lambda_i 
	\, - \, 
	(1+\beta) 
	\; = \; 
	(1+\beta) 
	\, - \, 
	\alpha^\sharp(\beta)\lambda_i'
$$
	meaning that $\alpha^\sharp(\beta)\lambda_i$ and $\alpha^\sharp(\beta)\lambda'_i$ are the mirror images with respect to the middle point $1+\beta$. Thus, from the quasi-convexity and symmetry of the functions $\hat{\rho}$ and $\hat{J}$, it follows that $\alpha^\sharp(\beta)$ minimizes $\rho$ as well as $\max \, \{\hat{J}(\lambda_i),\hat{J}(\lambda_i')\}$ for all $i$, which completes the proof.
\end{proof}

Since gradient descent is obtained from the heavy-ball method by letting $\beta=0$, from Lemma~\ref{lem.TuningWholeSpec1} it immediately follows that $\alpha_\gd = 2/(\Lf+\mf)$ given in Table~\ref{tab:rates} optimizes both $G_\gd$ and the convergence rate $\rho_\gd$. This fact combined with~\eqref{eq.JGJ} yields
	\begin{align}
	2\,J_\gd(\alpha^\star_\gd(c)) 
	\,\ge\, 
	G_\gd(\alpha^\star_\gd(c))
	\,\ge\, 
	G_\gd(\alpha_\gd) 
	\,\ge\,
 	J_\gd(\alpha_\gd)
	\end{align}
where $\alpha^\star_\gd(c)$ is given by~\eqref{eq.alphaStarBetaStarGD}. This completes the proof for gradient descent. 
	
We next use Lemma~\ref{lem.TuningWholeSpec1} to establish a bound on the parameter $\beta^\star_\hb(c)$ that allows us to prove the result for the heavy-ball method as well.
\begin{mylem}\label{lem.lowerboundBetaHB}
	There exists a positive constant $a$ such that 
	\begin{align}\label{eq.lowerboundBetaHB}
	\beta_\hb^\star(c) 
	\; \ge \; 
	1 \, - \, \frac{a}{\sqrt{\kappa}}
	\end{align} 
	where $\beta^\star_\hb(c)$ is given by~\eqref{eq.alphaStarBetaStarAcceleration}.
\end{mylem}
\begin{proof}
	We first show that for any parameters $\alpha$ and $\beta$, the convergence rate $\rho$ of the heavy-ball method given by~\eqref{eq.rhoForm} is lower bounded by 
	\begin{align}\label{eq.lowerBforRho}
	\rho 
	\;\ge\;
	\left\{
	\ba{ll}
	\sqrt{\beta}
	&\text{if}\;                           \beta\ge(\frac{\sqrt{\kappa}-1}{\sqrt{\kappa}+1})^2\\
	\frac{(1+\beta)(\Lf-\mf)+\sqrt{(1+\beta)^2(\Lf-\mf)^2-4\beta(\Lf+\mf)^2}}{2(\Lf+\mf)}&\text{otherwise.}
	\ea
	\right.
	\end{align}
	The convergence rate satisfies
	\begin{align*}
	\rho
	\;=\;
	\max_{1 \, \le \, i \, \le \, n}\; \hat{\rho}(\lambda_i)
	\;=\;
	\max_{\lambda \, \in \, \{\mf,\Lf\}} \hat{\rho}(\lambda)
	\end{align*}
	where the function $\hat{\rho}(\lambda)$ is given by (see proof of Lemma~\ref{lem.TuningWholeSpec1} for the proof of this statement) 
	\begin{align*}
	\hat{\rho}(\lambda) = 
	\left\{
	\ba{ll}
	\sqrt{\beta}
	&\text{if}\; (1-\sqrt{\beta})^2\le\alpha\lambda\le(1+\sqrt{\beta})^2
	\\[0.1cm]
	\frac{1}{2}|1+\beta-\alpha\lambda|+\frac{1}{2}\sqrt{\Delta}&\text{otherwise}
	\ea
	\right.
	\end{align*}
	\tc{black}{ and $\Delta \DefinedAs (1+\beta-\alpha\lambda)^2-4\beta.$}
	According to Lemma~\ref{lem.TuningWholeSpec1},  $\alpha  = 2(1+\beta)/(\Lf+\mf)$ optimizes  the rate $\rho$. This value of $\alpha$ yields
	\begin{align*}
	\hat{\rho}(\mf) \; = \; \hat{\rho}(\Lf) \; = \; 
	\left\{
	\ba{ll}
	\sqrt{\beta}
	&\text{if}\; \kappa\le\frac{(1+\sqrt{\beta})^2}{(1                                                                                                                                                                                                                                                                                                                                                                                                                                                       -\sqrt{\beta})^2}
	\\[-0.15cm]
	\frac{1}{2}|1+\beta-\alpha^\star\lambda|+\frac{1}{2}\sqrt{\Delta}\biggr\rvert_{\tc{black}{\lambda=\mf}}&\text{otherwise}
	\ea
	\right.
	\end{align*}
	or equivalently
	\begin{align}\label{eq.ratePartialOptimized}
	\hat{\rho}(\mf) \; = \; \hat{\rho}(\Lf) \; = \; 
	\left\{
	\ba{ll}
	\sqrt{\beta}
	&\text{if}\;                                        \beta\ge(\frac{\sqrt{\kappa}-1}{\sqrt{\kappa}+1})^2
	\\[0.1cm]
	\frac{(1+\beta)(\Lf-\mf)+\sqrt{(1+\beta)^2(\Lf-\mf)^2-4\beta(\Lf+\mf)^2}}{2(\Lf+\mf)}&\text{otherwise}
	\ea
	\right.
	\end{align}
\tc{black}{ which completes} the proof of inequality~\eqref{eq.lowerBforRho}. Now, if $\beta \ge (\sqrt{\kappa}-1)^2/(\sqrt{\kappa}+1)^2$, then~\eqref{eq.lowerboundBetaHB} with $a=2$ follows immediately. Otherwise, from~\eqref{eq.lowerBforRho} we obtain 
	\begin{align*}
	\rho\ge\frac{(1+\beta)(\Lf-\mf)+\sqrt{(1+\beta)^2(\Lf-\mf)^2-4\beta(\Lf+\mf)^2}}{2(\Lf+\mf)}
	\end{align*}
	which yields
	\begin{align}
	\beta
	 \;\ge\;
	v(\rho)	
	\;\DefinedAs\;
	{\rho\,(\tfrac{\Lf-\mf}{\Lf+\mf}\,-\,\rho)}/({1 \,-\, \tfrac{\Lf-\mf}{\Lf+\mf}\,\rho}).
	\end{align}
The convergence rate $\rho$ satisfies $(\sqrt{\kappa}-1)^2/(\sqrt{\kappa}+1)^2\le\rho\le 1-c/\sqrt{\kappa}$, where the lower bound follows from the optimal rate provided in Table~\ref{tab:rates} and the upper bound follows from the definition in~\eqref{eq.alphaStarBetaStarAcceleration}. Moreover, the derivative $\frac{\mrd v}{\mrd \rho} = 0$ vanishes only at $ \rho = (\sqrt{\kappa}-1)/(\sqrt{\kappa}+1)$. Thus, we obtain a lower bound on $\beta$ as
	\begin{align}\label{eq.lowerboundBeta2}
	\beta
	\;\ge\;
	v(\rho)
	\;\ge\;
	\min 
	\, 
	\{v((\frac{\sqrt{\kappa}-1}{\sqrt{\kappa}+1})^2),\;v(1-c/\sqrt{\kappa}),\;v(\frac{\sqrt{\kappa}-1}{\sqrt{\kappa}+1})\}.
	\end{align}
	A simple manipulation of~\eqref{eq.lowerboundBeta2} allows us to find a constant $a$ that satisfies~\eqref{eq.lowerboundBetaHB}, which completes the proof.
\end{proof}
	Let $(\hat{\alpha},\hat{\beta} )$ be the optimal solution of the optimization problem
	\[
	\ba{rl}
	\minimize\limits_{\alpha,\,\beta} & G(\alpha,\beta)
	\\[0.15cm]
	\subject & \rho \, \le \, 1 \, - \, c/\sqrt{\kappa}
	\non
	\ea
	\]
where $G$ is defined in~\eqref{eq.Gdefintion}. We next show that there exists a scalar $c'>0$ such that
	\begin{align}\label{eq.lowerboundBetaHeavyballObjectiveG}
	G(\hat{\alpha},\hat{\beta}) \; \ge \; c'J(\alpha_\hb,\beta_\hb)
	\end{align}
	where $\alpha_\hb$ and $\beta_\hb$ are provided in Table~\ref{tab:rates}. 
	Let $\hat{\alpha}(\beta)\DefinedAs 2(1+\beta)/(\Lf+\mf)$. It is straightforward to verify that 
	\begin{align}\label{eq.helperSimp1}
		J(\hat{\alpha}(\beta),\beta) \; = \; \dfrac{1-\beta_\hb^2}{1-\beta^2} \, J(\alpha_\hb,\beta_\hb)
	\end{align}
which allows us to write 
	\begin{align}\label{eq.chainineq1}
	G(\hat{\alpha},\hat{\beta})
	\;\overset{\text{(i)}}{=}\;
	&\min_\beta\;\; G(\hat{\alpha}(\beta),\beta) \\[-0.15cm]
	&\subject\;\;\; \rho\le 1-c/\sqrt{\kappa}\nonumber \\[-0.15cm]
	\;\overset{\text{(ii)}}{\ge}\;
	&\min_\beta\;\; J(\hat{\alpha}(\beta),\beta)\nonumber \\[-0.15cm]
	&\subject\;\;\; \rho \le 1-c/\sqrt{\kappa}\nonumber \\[-0.15cm]
	\;\overset{\text{(iii)}}{=}\;
	&\min_\beta\;\; \dfrac{1-\beta_\hb^2}{1-\beta^2} J(\alpha_\hb,\beta_\hb)\nonumber \\[-0.15cm]
	&\subject\;\;\; \rho \le 1-c/\sqrt{\kappa}\nonumber \\[-0.05cm]
	\;\overset{\text{(iv)}}{\ge}\;
	&\dfrac{1-\beta_\hb^2}{1-(1-\frac{a}{\sqrt{\kappa}})^2} \, J(\alpha_\hb,\beta_\hb).\nonumber
	\end{align}
	Here, (i) determines partial minimization with respect to $\alpha$ which follows from Lemma~\ref{lem.TuningWholeSpec1};
	(ii) follows from~\eqref{eq.JGJ}; (iii) follows from~\eqref{eq.helperSimp1}, and (iv) follows from Lemma~\ref{lem.lowerboundBetaHB}. Furthermore, it is easy to show the existence of a constant scalar $c'$ such that 
	\begin{align}\label{eq.lowrboundTemp1}
		\dfrac{1-\beta_\hb^2}{1-(1-\frac{a}{\sqrt{\kappa}})^2} 
		\;\ge\;
		c'.
	\end{align}
	Inequality~\eqref{eq.lowerboundBetaHeavyballObjectiveG} follows from combining~\eqref{eq.lowrboundTemp1} and~\eqref{eq.chainineq1}.		
	Finally, we obtain that
\begin{align*}
J(\alpha^\star_\gd,\beta^\star_\gd)
\; \ge \; 
\frac{1}{2} G(\alpha^\star_\gd,\beta^\star_\gd)
\; \ge \; 
\frac{1}{2} G(\hat{\alpha},\hat{\beta}) \; \ge \;
\dfrac{c'}{2}J(\alpha_\gd,\beta_\gd) 
\end{align*}
where  the first inequality follows from~\eqref{eq.JGJ}, the second inequality follows from the definition of $(\hat{\alpha},\hat{\beta})$, and the last inequality is given by~\eqref{eq.lowerboundBetaHeavyballObjectiveG}.
This completes the proof of Theorem~\ref{thm.lowerboundBetaHeavyball} for the heavy-ball method. 

\subsection{Fundamental lower bounds}
\label{sec.proofOflowerBoundJstar}

\begin{proof}[Proof of Theorem~\ref{thm.tradeOffHB}]
	{\color{black}
	We first prove~\eqref{eq.tradOffHBa}. Without loss of generality, let the noise magnitude  $\sigma=1$. We define the trivial lower bound
	\begin{align}\label{eq.jhatstar}
	J 
	\;\ge\;
	\hat{J}^\star 
	\;\DefinedAs\;
	\max \, \{\hat{J}(\mf),\hat{J}(\Lf)\}
	\end{align}
	 and  show that
		$
		\dfrac{\hat{J}^\star}{1-\rho} \;\ge\; (\dfrac{\kappa+1}{8})^2.
		$}
Let $\tilde{f}(x_1,x_2) \DefinedAs \tfrac{1}{2} \, ( \mf\, x_1^2 + \Lf \, x_2^2 )$. The eigenvalues of the Hessian matrix $\nabla^2 \tilde{f}$ are given by $\mf$ and $\Lf$ which are clearly symmetric over the interval $[\mf,\Lf]$. Thus,  for any given value of $\beta$, $\mf$, and $\Lf$, we can use Lemma~\ref{lem.TuningWholeSpec1} with the objective function $\tilde{f}$ to obtain 
	\begin{align*}
		\hat{\alpha}(\beta) \;\DefinedAs\;  \dfrac{2(1+\beta)}{\Lf+\mf} \;=\; \argmin_\alpha \,\hat{J}^\star(\alpha,\,\beta) =  \argmin_\alpha \,\rho(\alpha,\,\beta).
	\end{align*} 
	For the stepsize $\hat{\alpha}(\beta)$, the rate of convergence $\rho$ is given by~\eqref{eq.ratePartialOptimized}, i.e.,
	\begin{align}\label{eq.rhotemp1}
 \rho \;=\; 
 \left\{
 \ba{ll}
 \sqrt{\beta}
 &\text{if}\;                                        \beta\ge(\frac{\sqrt{\kappa}-1}{\sqrt{\kappa}+1})^2\\
 \frac{(1+\beta)(\Lf-\mf)+\sqrt{(1+\beta)^2(\Lf-\mf)^2-4\beta(\Lf+\mf)^2}}{2(\Lf+\mf)}&\text{otherwise}
 \ea
 \right.
 \end{align}
 and the lower bound $\hat{J}^\star$ is given by
 \begin{align}\label{eq.jhatstartemp1}
 \hat{J}^\star = \hat{J}(\mf)=\hat{J}(\Lf) 
 =\dfrac{(\Lf+\mf)^2}{4\,\Lf\,\mf(1-\beta^2)}.
 \end{align}
 Therefore, we obtain a lower bound on $\hat{J}^\star/(1-\rho)$ as
 \begin{align}\nonumber
 	\dfrac{\hat{J}^\star(\alpha,\beta)}{1-\rho(\alpha,\beta)}
 	&\;\ge\; 
 	\nu(\beta) 
 	\;\DefinedAs\;
 	\dfrac{\hat{J}^\star(\hat{\alpha}(\beta),\beta)}{1-\rho(\hat{\alpha}(\beta),\beta)}
	\\[0.cm]
	\label{eq.vBetaDef}
 	&\;=\;
 	\left\{
 	\ba{ll}
 \frac{(\Lf+\mf)^2}{4\,\Lf\,\mf(1-\beta^2)(1-\sqrt{\beta})}
 &\text{if}\;                                    \beta\ge(\frac{\sqrt{\kappa}-1}{\sqrt{\kappa}+1})^2\\
 \frac{(\Lf+\mf)^3}{2\,\Lf\,\mf(1-\beta^2)\left((1-\beta)\Lf+(3+\beta)\mf-\sqrt{(1+\beta)^2(\Lf-\mf)^2-4\beta(\Lf+\mf)^2}\right)}&\text{otherwise}
 \ea
 \right.
 \end{align}
 where the last equality follows from~\eqref{eq.rhotemp1} and~\eqref{eq.jhatstartemp1}. It can be shown that $v(\beta)$ attains its minimum at $\beta =(\sqrt{\kappa}-1)^2 / (\sqrt{\kappa}+1)^2 $; see Figure~\ref{fig.vbeta} for an illustration.
 \begin{figure}[h!]
 	\centering
 	\begin{tabular}{rc}
 		\hspace{0 cm}
 		\begin{tabular}{c}
 			\vspace{.45cm}
 			\rotatebox{90}{$v$}
 		\end{tabular}
 		&
 		\hspace{-0.5cm}
 		\begin{tabular}{c}			\includegraphics[width=.3\textwidth]{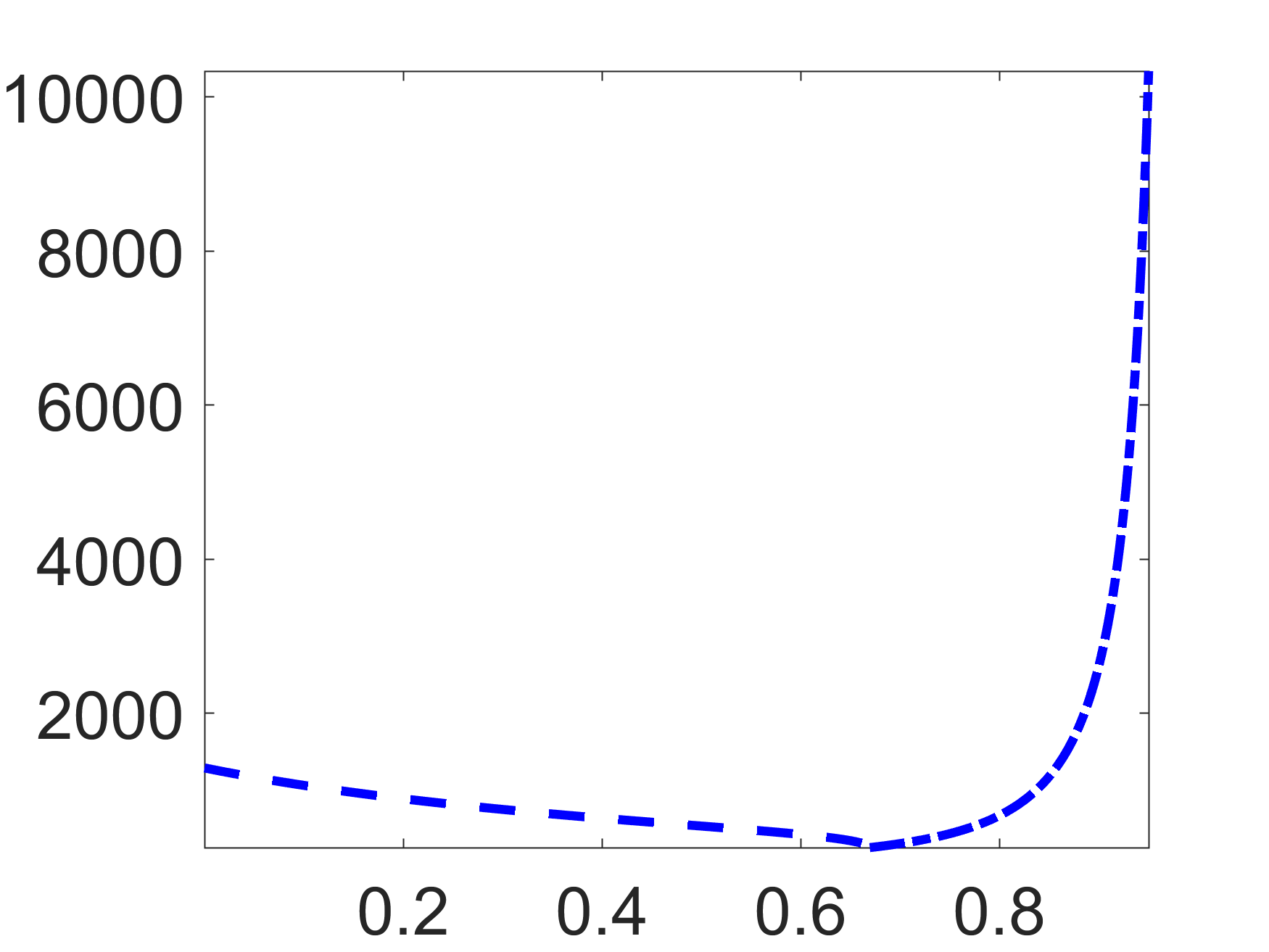}
 			\\[-0cm]
 			{\small  $\beta$}
 		\end{tabular}
 	\end{tabular}
 	\caption{The $\beta$-dependence of the function $v$ in~\eqref{eq.vBetaDef} for $\Lf=100$ and $\mf=1$.}
 	\label{fig.vbeta}
 \end{figure}
  Therefore,
 \begin{align*}
\nu(\beta)
&\;\ge\;
\dfrac{(\Lf+\mf)^2}{4\,\Lf\,\mf(1-\beta^2)(1-\sqrt{\beta})}\biggr\vert_{\beta=(\frac{\sqrt{\kappa}-1}{\sqrt{\kappa}+1})^2}
\;=\;
\dfrac{(\Lf+\mf)^2}{4\,\Lf\,\mf(1+\beta)(1+\sqrt{\beta})(1-\sqrt{\beta})^2}\biggr\vert_{\beta=(\frac{\sqrt{\kappa}-1}{\sqrt{\kappa}+1})^2}\\[0.cm]
&\;\ge\;
\dfrac{(\Lf+\mf)^2}{16\,\Lf\,\mf(1-\sqrt{\beta})^2}\biggr\vert_{\beta=(\frac{\sqrt{\kappa}-1}{\sqrt{\kappa}+1})^2}
\;=\;
\dfrac{(\kappa+1)^2 (\sqrt{\kappa}+1)^2}{64\kappa}
\;\ge\; \left(\dfrac{\kappa+1}{8}\right)^2
 \end{align*}
 {\color{black} which completes the proof of~\eqref{eq.tradOffHBa}. We next prove~\eqref{eq.tradOffHBb} for $\sigma = \alpha$. 
 	
 	We analyze the two cases $\alpha> 1/L$ and $\alpha\le 1/L$ separately. If $\alpha>1/L$, inequality~\eqref{eq.tradOffHBb} directly follows from inequality~\eqref{eq.tradOffHBa}
\begin{align*}
	\dfrac{J_\hb}{1 \, - \, \rho} \; \ge \; \sigma^2\left(\dfrac{\kappa+1}{8}\right)^2
	\;=\;
	\alpha^2\left(\dfrac{\kappa+1}{8}\right)^2
	\;\ge\; \left(\dfrac{\kappa}{8L}\right)^2.
\end{align*} 
Here, the first inequality is given by~\eqref{eq.tradOffHBa} and the second inequality holds since $\alpha> 1/L$.

Now suppose $\alpha\le 1/L$. The convergence rate of Polyak's method is given by $\max_i\hat{\rho}(\lambda_i)$, where
\begin{align*}
\hat{\rho}(\lambda) \; = \; 
\left\{
\ba{ll}
\sqrt{\beta}
&\text{if}\; (1-\sqrt{\beta})^2\le\alpha\lambda\le(1+\sqrt{\beta})^2
\\[0.1cm]
\frac{1}{2}|1+\beta-\alpha\lambda|+\frac{1}{2}\sqrt{\Delta}&\text{otherwise}
\ea
\right.
\end{align*}
and $\Delta \DefinedAs (1+\beta-\alpha\lambda)^2-4\beta$ (see the proof of Lemma~\ref{lem.TuningWholeSpec1}). Thus, for $\sigma=\alpha$, we have the trivial lower bound
\begin{align*}
\dfrac{ J}{1-\rho}
&\;\ge\;
\dfrac{ \hat{J}(\mf)}{1-\hat{\rho}(\mf)}
\;=\;
\dfrac{\alpha(1 \, + \, \beta)}
{ \mf\left(1\, - \, \beta\right) 
	\left(
	2(1 +  \beta) \, - \, \alpha\mf
	\right) \left(1-\hat{\rho}(\mf)\right)
} 
\nonumber\\[0.15 cm]
&\;\ge\;
p(\alpha,\beta)
\;\DefinedAs\;
\dfrac{\alpha}
{ 2\mf\left(1\, - \, \beta\right) 
	\left(1-\hat{\rho}(\mf)\right)
} 
\nonumber\\[0.15 cm]
&\; = \; 
\left\{
\ba{ll}
\dfrac{\alpha}
{ 2\mf\left(1\, - \, \beta\right) 
	\left(1-\sqrt{\beta}\right)
}, 
&\; \beta\in [(1-\sqrt{\alpha\mf})^2,\;1)
\\[0.45cm]
\dfrac{\alpha}
{ \mf\left(1\, - \, \beta\right) 
	\left(1-\beta+\alpha\mf-\sqrt{\Delta}\right)
} ,
&\; \beta\in [0,\;(1-\sqrt{\alpha\mf})^2).
\ea
\right. 
\end{align*}
Here, the first inequality follows from combining $J=\sum_{i} \hat{J}(\lambda_i)$ and $\max_i\hat{\rho}(\lambda_i)$, and the second inequality follows from $\alpha \mf \le \alpha \Lf\le 1$.
We next show that for any fixed $\alpha$, the function $p(\alpha,\cdot)$ attains its minimum at $\beta = (1-\sqrt{\alpha\mf})^2$. Before we do so, note that this fact allows us to use partial minimization with respect to $\beta $ and obtain
\begin{align*}
p(\alpha,\beta)
\;\ge\;
p(\alpha, (1-\sqrt{\alpha\mf})^2)
\;=\;
\dfrac{1}
{ 2\mf^2\left(2-\sqrt{\alpha\mf}\right) 
}
\;\ge\;
\dfrac{1}{4\mf^2} 
\;\ge\;
(\dfrac{\kappa}{2 L})^2
\end{align*}
which completes the proof of~\eqref{eq.tradOffHBb}.

For any fixed $\alpha$, it is straightforward to verify that  $p(\alpha,\beta)$ is increasing with respect to $\beta$ over  $ [(1-\sqrt{\alpha\mf})^2,\;1)$. Thus, it suffices to show that 
$p(\alpha,\beta)$ is decreasing with respect to $\beta$ over  $ [0,(1-\sqrt{\alpha\mf})^2)$. To simplify the presentation, let us define the new set of parameters
\begin{align*}
q
&\;\DefinedAs\;
s
\left(s+x-\delta\right)
,\quad
s
\;\DefinedAs\;
1-\beta
,\quad
x\;\DefinedAs\;
\alpha m
\\[0.15 cm]
\delta
&\;\DefinedAs\;
\sqrt{\Delta} \;=\; \sqrt{(1+\beta-\alpha\mf)^2-4\beta} \;=\; \sqrt{(s+x)^2-4x}.
\end{align*}
It is now  straightforward to verify that $p(\alpha,\beta) = \alpha/(mq)$ for $\beta\in [(1-\sqrt{\alpha\mf})^2,\;1)$. It thus follows that  $p(\alpha,\beta)$ is decreasing with respect to $\beta$ over  $ [0,(1-\sqrt{\alpha\mf})^2)$ if and only if
$q' = \mrd q/\mrd s\le0$ for $s \in (\sqrt{x}(2-\sqrt{x}),1]$. The derivative is given by
\begin{align*}
q'
\;=\;
\dfrac{1}{\delta}\left( (2s+x)\delta - 2s^2-3sx-x^2+4x \right).
\end{align*}
Thus, we have
\begin{align}\label{eq.tempPolSup2}
q'\;\le\; 0
&\;\iff\;
(2s+x)\delta\;\le\; 2s^2+3sx+x^2-4x.
\end{align}
It is easy to verify that both sides of the inequality in~\eqref{eq.tempPolSup2}, namely, $(2s+x)\delta$ and $2s^2+3sx+x^2-4x$ are positive for the specified range of $s\in (\sqrt{x}(2-\sqrt{x}),1]$. Thus, we can square both sides and obtain that
\begin{align*}
q'\;\le\; 0
&\;\iff\;
(2s+x)^2\delta^2\;\le\; (2s^2+3sx+x^2-4x)^2
\\[0.15 cm]
&\;\overset{(i)}{\iff}\;
(2s+x)^2\left((s+x)^2-4x\right) \;\le\; (2s^2+3sx+x^2-4x)^2
\\[0.15 cm]
&\;\overset{(ii)}{\iff}\;
8sx^2+4x^3\;\le\;16 x^2 
\;\iff\;
8s + 4x \;\le\; 16.
\end{align*}
where $(i)$ follows from the definition of $\delta$ and $(ii)$ is obtained by expanding both sides and rearranging the terms.
Finally, the inequality $8s + 4x \;\le\; 16$ clearly holds since $s\le 1$ and $x\le1$. This proves that  $p(\alpha,\cdot)$ attains its minimum at $\beta = (1-\sqrt{\alpha\mf})^2$.
}
 \end{proof}

	\begin{proof}[Proof of Theorem~\ref{th.lowerBoundJstar}]
		{\color{black}For the heavy-ball method, the result follows from combining Theorem~\ref{thm.tradeOffHB} and the inequality $1-\rho>{c}/{\sqrt{\kappa}}$. Next, we present three additional  lemmas that allow us to prove the result for Nesterov's method.
			
		The following lemma provides a lower bound on the function $\hat{J}(\mf)$ associated with Nesterov's method which depends on $\kappa$ and $\beta$.}

\begin{mylem}\label{lem.lowerBoundJhatm}
	For any strongly convex quadratic problem with condition number $\kappa>2$ and the smallest eigenvalue of the Hessian $\mf$, the  function $\hat{J}$  associated with Nesterov's accelerated method with any stabilizing pair of parameters $0<\alpha$, $0<\beta<1$, and $\sigma=1$ satisfies
	\begin{align}\label{eq.lowerBoundJhatm}
	\hat{J}(\mf) \ge \dfrac{\kappa^2}{24(1-\beta)\kappa+32\beta}.
	\end{align}
\end{mylem}	
\begin{proof}
	We first show that Nesterov's  method with $0<\alpha$ and $0<\beta<1$ is stable if and only if 
	\begin{align}\label{eq.StabilityNesterov}
	\mf<\frac{2\beta+2}{\alpha\,\kappa\,(2\beta+1)}.
	\end{align}  
	The rate of linear convergence is given by 
	$\rho = \max_{1\le i\le n}\; \hat{\rho}(\lambda_i)$, 
	where $\hat{\rho}(\lambda)$ is the largest absolute value of the roots of the characteristic polynomial 
	\begin{align*}
	\det(zI-\hat{A}) = z^2-(1+\beta)(1-\alpha\lambda)z + \beta(1-\alpha\lambda)
	\end{align*}
	associated with Nesterov's method and the eigenvalue $\lambda$ of the Hessian of the objective function $f$; \tc{black}{See~\eqref{eq.Ahat} for the form of $\hat{A}$.} For $\alpha>0$ and $0<\beta<1$, it can be shown that
	\begin{align}\label{eq.rhohatNester}
	\hat{\rho}(\lambda) = 
	\left\{
	\ba{ll}
	\sqrt{\beta(1-\alpha\lambda)}&\text{if}\quad \alpha\lambda\in((\frac{1-\beta}{1+\beta})^2, 1)\\
	\frac{1}{2}|(1+\beta)(1-\alpha\lambda)|+\frac{1}{2}\sqrt{(1+\beta)^2(1-\alpha\lambda)^2-4\beta(1-\alpha\lambda)}&\text{otherwise.}
	\ea
	\right.
	\end{align}
	The stability of the algorithm is equivalent to  $\hat{\rho}(\lambda_i)<1$ for all eigenvalues $\lambda_i$. For any positive stepsize $\alpha$ and parameter $\beta \in (0,1)$, it can be shown that the function $\hat{\rho}(\lambda)$ is quasi-convex and $\hat{\rho}(\lambda)=1$ if and only if $\lambda\in\{0,\frac{2\beta+2}{\alpha(2\beta+1)}\}$. This fact along with $0<\mf\le\lambda_i\le\Lf=\kappa\,\mf$ imply that $\hat{\rho}(\lambda_i)<1$ for all $\lambda_i\in[\mf,\Lf]$ if and only if  
	$\kappa\,\mf\le \frac{2\beta+2}{\alpha(2\beta+1)}$ which completes the proof of~\eqref{eq.StabilityNesterov}.

	For Nesterov's method, it is straightforward to show that the function $\hat{J}(\lambda)$ is quasi-convex over the interval $[0,\frac{2\beta+2}{\alpha(2\beta+1)}]$ and that it attains its minimum at $\lambda=1/\alpha$. Also, from~\eqref{eq.StabilityNesterov}, for $\kappa>2$ we obtain 
	\begin{align*}
		\mf
		\;\le\;
		\frac{2\beta+2}{\alpha\kappa(2\beta+1)}
		\;\le\;
		\dfrac{1}{\alpha}
	\end{align*}
and thus,
	\begin{align*}	
	\hat{J}(\mf) \, \ge \, \hat{J}(\frac{2\beta+2}{\alpha\kappa(2\beta+1)}) & \, = \,    \frac{\left(2\,\beta+1\right)\kappa^2\,\left(\kappa-2\,\beta+2\,\beta\,\kappa\right)}{4\,\left(\beta+1\right)\,\left(\kappa-1\right)\,\left(2\,\beta+\kappa+\beta\,\kappa-2\,\beta^2\,\kappa+2\,\beta^2\right)}
	\, \ge \, \frac{\kappa^2}{24\,(1-\beta) \kappa+ 32\beta}
	\end{align*}
	where the last inequality follows from the fact that $\beta \in (0, 1)$.
\end{proof}

The following lemma presents a lower bound on any accelerating parameter $\beta$ for Nesterov's method.
\begin{mylem}\label{lem.lowerboundBeta}
	For Nesterov's method,  under the conditions of Theorem~\ref{th.lowerBoundJstar},   there exist positive constants $c_3$ and $c_4$ such that for any $\kappa>c_3$, 
	\begin{align}\label{eq.upperBoundBeta}
	\beta \, > \, 1 \, - \, \frac{c_4}{\sqrt{\kappa}}.
	\end{align}
\end{mylem}
\begin{proof}
	For any $\alpha>0$ and $\beta \in (0,1)$, Nesterov's method converges with the rate $\rho = \max_{1 \, \le \, i \, \le \, n} \hat{\rho}(\lambda_i)$, where $\hat{\rho}(\lambda)$ is given by~\eqref{eq.rhohatNester}. We treat the two cases $(1-\beta)/(1+\beta)^2<\alpha \mf$ and $(1-\beta)/(1+\beta)^2 \ge \alpha \mf$ separately.
	For $(1-\beta)/(1+\beta)^2 < \alpha\mf$, we have
	\begin{align}
	(1-\beta)^2 
	\, \le \, 
	4(\dfrac{1-\beta}{1+\beta})^2
	\, < \, 4\alpha\mf 
	\, = \, 
	4\frac{\alpha\Lf}{\kappa}
	\, \le \, 
	\frac{8}{\kappa}
	\end{align}
	where the last inequality follows from~\eqref{eq.StabilityNesterov}. Therefore, we obtain $\beta \ge 1-\sqrt{8}/\sqrt{\kappa}$ as required. Now, suppose $(1-\beta)/(1+\beta)^2 \ge \alpha\mf$. The convergence rate $\rho$ satisfies
	\begin{align*}
	\rho \; \ge \; \frac{1}{2}(1+\beta)(1-\alpha\mf) \, + \, \frac{1}{2}\sqrt{(1+\beta)^2(1-\alpha\mf)^2-4\beta(1-\alpha\mf)} .
	\end{align*}
Thus, 
	\begin{align*}
	\rho^2 - \rho(1+\beta)(1-\alpha\mf) + \beta(1-\alpha\mf)>0
	\end{align*} 
	which yields a lower bound on $\beta$,
	\begin{align}
	\beta \, \ge \, \nu(\rho,\alpha\mf) \, \DefinedAs \, \frac{\rho(1-\alpha\mf-\rho)}{(1-\rho)(1-\alpha\mf)}.
	\end{align}
	In what follows, we establish a lower bound for $\nu$.
	For a fixed $\alpha\mf$, the critical point of $\nu(\rho)$ is given by $\rho_1 \DefinedAs 1-\sqrt{\alpha\mf}$, i.e., ${\partial \nu}/{\partial \rho} = 0$ for $\rho = \rho_1$. Furthermore,  the optimal rate from Table~\ref{tab:rates} and the condition on convergence rate in Theorem~\ref{th.lowerBoundJstar} for any $\kappa>c_1$ yield upper and lower bounds
	$\rho_3<\rho<\rho_2$,
where $\rho_2 \DefinedAs 1- {c_2}/{\sqrt{\kappa}}$ and $\rho_3 \DefinedAs 1- {2}/\sqrt{3\kappa+1}$.
	Thus, the lower bound on $\nu$ is given by
	\begin{align}\label{eq.boundalphamtemp1}
	\beta \; \ge \; \nu(\rho,\alpha\mf) \; \ge \; \min \, \{\nu(\rho_1,\alpha\mf),\nu(\rho_2,\alpha\mf), \nu(\rho_3,\alpha\mf)\}.
	\end{align}
	From the stability condition~\eqref{eq.StabilityNesterov}, we have 
	\begin{align}\label{eq.boundalphamtemp2}
		\alpha\mf \; < \; {2}/{\kappa}
	\end{align}
	 Furthermore, it can be shown that for any given $\rho \in (0,1)$ the function $\nu(\rho,\alpha\mf)$ is decreasing with respect to $\alpha\mf$. This fact combined with~\eqref{eq.boundalphamtemp1} and~\eqref{eq.boundalphamtemp2} yield
	\begin{align}
	\beta \; \ge \; \min \, \{\nu(\rho_1,\alpha\mf),\nu(\rho_2,{2}/{\kappa}),\nu(\rho_3,{2}/{\kappa})\}.
	\end{align}
	If we substitute for $\rho_1$. $\rho_2$, and $\rho_3$ their values as functions of $\kappa$ and use  $\alpha\mf<{2}/{\kappa}$, then the result follows immediately. In particular, 
	\begin{align*}
	\nu(\rho_1,\alpha\mf)  
	&\;=\;
	\dfrac{1-\sqrt{\alpha\mf}}{1+\sqrt{\alpha\mf}}
	\;\ge\;
	\dfrac{1-\sqrt{2/\kappa}}{1+\sqrt{2/\kappa}}
	\;=\; 
	\dfrac{\sqrt{\kappa}-\sqrt{2}}{\sqrt{\kappa}+\sqrt{2}}
	\ge 1- \dfrac{2\sqrt{2}}{\sqrt{\kappa}}\\[0.05cm]
	\nu(\rho_2, {2}/{\kappa}) 
	&\;=\; 
	1-\dfrac{(\frac{2}{c_2}+c_2)\sqrt{\kappa}-4 }{\kappa-2}
	\;\ge\;
	 1 - \dfrac{(\frac{2}{c_2}+c_2)}{\sqrt{\kappa}},
	 \quad\forall\kappa
	 \;\ge\;
	 (\frac{1}{c_2}+\frac{c_2}{2})^2\\[0.05cm]
	\nu(\rho_3, {2}/{\kappa})
	&\;=\;
	1-\dfrac{5\kappa-4\sqrt{3\kappa+1}+1}{(\kappa-2)\sqrt{3\kappa+1}}
	\;\ge\;
	 1-\dfrac{5}{\sqrt{\kappa}},
	 \quad\forall\kappa
	 \;\ge\;
	  9
	\end{align*}
	which completes the proof.
\end{proof}

{\color{black}
The next lemma provides a lower bound on $J_{\na}/(1-\rho)$ for Nesterov's method with $\sigma=\alpha\le 1/L$.
\begin{mylem}\label{lem.NesCorrection}
 Nesterov's accelerated method with any stabilizing pair of parameters $0<\alpha\le 1/L$ and $0<\beta<1$,  and $\sigma=\alpha$ satisfies
	\begin{align*}
	\dfrac{J_{\na}}{1-\rho}
	\;\ge\;
	\dfrac{1}{8}(\dfrac{\kappa}{L})^2.
	\end{align*}
\end{mylem}
\begin{proof}
The convergence rate of Nesterov's method is given by $\max_i\hat{\rho}(\lambda_i)$, where
\begin{align*}
\hat{\rho}(\lambda) = 
\left\{
\ba{ll}
\sqrt{\beta(1-\alpha\lambda)}&\text{if}\quad \alpha\lambda\in((\frac{1-\beta}{1+\beta})^2, 1)\\
\frac{1}{2}|(1+\beta)(1-\alpha\lambda)|+\frac{1}{2}\sqrt{\Delta}&\text{otherwise}
\ea
\right.
\end{align*}
and $\Delta\DefinedAs (1+\beta)^2(1-\alpha\lambda)^2-4\beta(1-\alpha\lambda)$; see equation~\eqref{eq.rhohatNester}.
Thus, we have the trivial lower bound
\begin{align}
\dfrac{J}{1-\rho}
&\;\ge\;
\dfrac{\hat{J}(\mf)}{1-\hat{\rho}(\mf)}
\;=\;
\dfrac{\alpha\left(1 \, + \, \beta(1 \, - \, \alpha\mf)\right)}
{  \mf
	\left( 
	1 \, - \, \beta(1 \, - \, \alpha\mf)
	\right)
	\left(
	2(1 \, + \, \beta) \, - \, (2\beta \, + \, 1)\alpha\mf
	\right) \left(1-\hat{\rho}(\mf)\right)
} 
\nonumber\\[0.15 cm]
&\;\ge\;
p(\alpha,\beta)
\;\DefinedAs\;
\dfrac{\alpha}
{ 4\mf\left( 
	1 \, - \, \beta(1 \, - \, \alpha\mf)
	\right)
	\left(1-\hat{\rho}(\mf)\right)
} 
\nonumber\\[0.15 cm]
&\; = \; 
\left\{
\ba{ll}
\dfrac{\alpha}
{ 4\mf\left(1\, - \, \beta(1-\alpha\mf)\right) 
	\left(1-\sqrt{\beta(1-\alpha\mf)}\right)
}, 
&\; \beta\in [ \gamma  ,\;1)
\\[0.55cm]
\dfrac{\alpha}
{ 2\mf\left(1\, - \, \beta(1-\alpha\mf)\right) 
	\left( 2- (1+\beta)(1-\alpha\mf) - \sqrt{\Delta}\right)
}, 
&\; \beta\in [0,\;\gamma)
\ea
\right. 
\label{eq.tempNesSup1}
\end{align}
where $\gamma \DefinedAs \dfrac{1-\sqrt{\alpha\mf}}{1+\sqrt{\alpha\mf}}$. Here, the first inequality can be obtained by combining $J=\sum_{i} \hat{J}(\lambda_i)$ and $\max_i\hat{\rho}(\lambda_i)$, and the second inequality follows from the fact that  $0< \alpha \mf \le 1$ and $0\le \beta<1$.
We next show that for any fixed $\alpha$, the function $p(\alpha,\cdot)$ attains its minimum at $\beta = \gamma$.
Before we do so, note that this fact allows us to do partial minimization with respect to $\beta $ and obtain
\begin{align*}
p(\alpha,\beta)
\;\ge\;
p(\alpha, \gamma)
\;=\;
\dfrac{1}
{ 4\mf^2\left(2-\sqrt{\alpha\mf}\right) 
}
\;\ge\;
\dfrac{1}{8\mf^2} 
\;\ge\;
\dfrac{1}{8}(\dfrac{\kappa}{L})^2.
\end{align*}

For any fixed $\alpha$, it is straightforward to verify that  $p(\alpha,\beta)$ is increasing with respect to $\beta$ over  $ [\gamma,\;1)$. Thus, it suffices to show that 
$p(\alpha,\beta)$ is decreasing with respect to $\beta$ over  $ [0,\gamma)$. 
To simplify the presentation, let us define
\begin{align*}
q
&\;\DefinedAs\;
(1-s)(2-x-s - \delta)
,\quad
x\;\DefinedAs\;
1-\alpha m
,\quad
s\;\DefinedAs\;
\beta x
\\[0.15 cm]
\delta
&\;\DefinedAs\;
\sqrt{\Delta} \;=\; \sqrt{(1+\beta)^2(1-\alpha\mf)^2-4\beta(1-\alpha\mf)} \;=\; \sqrt{(x+s)^2-4s}.
\end{align*}
It is now straightforward to verify that $p(\alpha,\beta) = \alpha/(2mq)$ for $\beta \in  [0,\gamma)$. It thus follows that  $p(\alpha,\beta)$ is decreasing with respect to $\beta$ over  $ [0,\gamma)$ if and only if
$q' = \mrd q/\mrd s\ge0$ for $s \in [0,(1-\sqrt{1-x})^2)$. The derivative is given by
\begin{align*}
q'
\;=\;
\dfrac{1}{\delta}\left( (x+2s-3)\delta + (1-s)(2-x-s) + \delta^2\right).
\end{align*}
Thus, we have
\begin{align}\label{eq.tempNesSup2}
q'\;\ge\; 0
&\;\iff\;
(1-s)(2-x-s) + \delta^2\;\ge\; (3-x-2s)\delta.
\end{align}
It is easy to verify that both sides of the inequality in~\eqref{eq.tempNesSup2}, namely, $(1-s)(2-x-s) + \delta^2$ and $(3-x-2s)\delta$ are positive for the specified range of $s\in [0,(1-\sqrt{1-x})^2)$. Thus, we can square both sides and obtain that
\begin{align*}
q'\;\ge\; 0
&\;\iff\;
\left((1-s)(2-x-s) + \delta^2\right)^2\;\ge\; (3-x-2s)^2\delta^2
\\[0.15 cm]
&\;\overset{(i)}{\iff}\;
\left((1-s)(2-x-s) + (x+s)^2-4s\right)^2 \;\ge\; (3-x-2s)^2\left((x+s)^2-4s\right)
\\[0.15 cm]
&\;\overset{(ii)}{\iff}\;
4(x-1)^2(2s+x+1)\;\ge\;0.
\end{align*}
where $(i)$ follows from the definition of $\delta$ and $(ii)$ is obtained by expanding both sides and rearranging the terms.
Finally, the inequality $4(x-1)^2(2s+x+1)\;\ge\;0$ trivially holds which completes the proof.
\end{proof}	
}

We are now ready to prove Theorem~\ref{th.lowerBoundJstar} for Nesterov's method.
Inequality~\eqref{eq.lowerboundJStara} directly follows from combining~\eqref{eq.lowerBoundJhatm} in	Lemma~\ref{lem.lowerBoundJhatm} and~\eqref{eq.upperBoundBeta} in Lemma~\ref{lem.lowerboundBeta}. To show inequality~\eqref{eq.lowerboundJStarb}, we treat the two cases $\alpha>1/L$ and $\alpha\le 1/L$ separately. If $\alpha>1/L$, then~\eqref{eq.lowerboundJStarb} directly follows from~\eqref{eq.lowerboundJStara}
\begin{align*}
	J_{\na} 
	\;=\;
	\alpha^2
	\dfrac{J_{\na}}{\sigma^2} 
	\;=\;
	\Omega(\dfrac{\kappa^{\frac{3}{2}}}{L^2}). 
\end{align*}
Now suppose $\alpha\le 1/L$. We can use Lemma~\ref{lem.NesCorrection} to obtain
\begin{align*}
	J_{\na}
	\;\ge\;
	(1-\rho)\dfrac{k^2}{8L^2}
	\;\ge\;
	\dfrac{c}{\sqrt{\kappa}}\dfrac{k^2}{8L^2}
	\;=\;
	\Omega(\dfrac{\kappa^{\frac{3}{2}}}{L^2}).
\end{align*}
Here, the first inequality follows from Lemma~\ref{lem.NesCorrection} and the second inequality follows from the acceleration assumption $\rho\le 1- c/\sqrt{\kappa}$.
This completes the proof.
\end{proof}

	\vspace*{-2ex}
\subsection{Consensus over $d$-dimensional torus networks}
	\label{app.networks}
	
	The proof of Theorem~\ref{cor:orders} uses the explicit expression for the eigenvalues of torus in~\eqref{eq:eigTorusExpression} to compute the variance amplification $\bar{J}=\sum_{i\neq0}\hat{J}(\lambda_i)$ for all three algorithms. Several technical results that we use in the proof are presented next.
	
	We  borrow the following lemma, which provides tight bounds on the sum of reciprocals of the eigenvalues of a $d$-dimensional torus network, from~\cite[Appendix~B]{bamjovmitpat12}. 

	\begin{mylem}
		\label{lem:sumRecip}
		The eigenvalues $ \lambda_i $ of the graph Laplacian of the $d$-dimensional torus $ \T_{\m}^d $ with $ {\m} \gg 1$ satisfy
\[
\sum_{0 \, \neq \, i \, \in \, \Z_{\m}^d}
\;
\dfrac{1}{\lambda_i}
\; = \;
\Theta(B({\m}))
\]
where the function $ B $ is given by
		\[
		B({\m})
		\; = \;
		\left\{
		\ba{lrcl}
		\dfrac{1}{d-2}
		\, (\m^d-\m^2), & d & \!\!\! \ne \!\!\!  & 2 
		\\[0.15cm]
		\m^d \log \, {\m},& d & \!\!\! = \!\!\! & 2 .
		\ea
		\right.
		\]
	\end{mylem} 
	We next use Lemma~\ref{lem:sumRecip} to establish an asymptotic expression for the variance amplification of the gradient descent algorithm for a $d$-dimensional torus. 
	
	\begin{mylem}
		\label{lem:GDboundTorus}
		For the consensus problem over a $d$-dimensional torus $ \T_{\m}^d $ with ${\m} \gg 1$, the performance metric $\bar{J}_{\gd}$ corresponding to gradient decent with the stepsize $\alpha = 2/(\Lf+\mf)$ satisfies
		\[
\bar{J}_{\gd} 
\; = \;
\Theta( B({\m}) )
\]
		where the function $ B $ is given in Lemma~\ref{lem:sumRecip}.
	\end{mylem}
	\begin{proof}
		Using the expression for the noise amplification of gradient descent from Theorem~\ref{th.varianceJhat}, we have
		\begin{align*}
		\bar{J}_{\gd} 
		&\,=\, 
		\sum_{ 0 \, \neq \, \binx \, \in \, \Z_{\m}^d}
		\dfrac{1}
		{\alpha \lambda_\binx (2 \, - \, \alpha \lambda_\binx )}\\
		&\,=\,
		\dfrac{1}{2\alpha}
		\sum_{ 0 \, \neq \, \binx \, \in \, \Z_{\m}^d}
		\dfrac{1}{\lambda_\binx} + \dfrac{1}
		{ \tfrac{2}{\alpha} \, - \,  \lambda_\binx}\\
		&\,=\,
		\dfrac{1}{2\alpha}
		\sum_{ 0 \, \neq \, \binx \, \in \, \Z_{\m}^d}
		\dfrac{1}{\lambda_\binx} + \dfrac{1}
		{ \lambda_{\max} \,+\, \lambda_{\min} \, - \,  \lambda_\binx}\\
		&\,\approx\,
		\dfrac{1}{\alpha}
		\sum_{ 0 \, \neq \, \binx \, \in \, \Z_{\m}^d}
		\dfrac{1}{\lambda_\binx} \,\approx\, 2d
		\sum_{ 0 \, \neq \, \binx \, \in \, \Z_{\m}^d}
		\dfrac{1}{\lambda_\binx}.
		\end{align*}
		The first approximation follows from the facts that the eigenvalues  satisfy 
		\[
		0 \; < \; 
		\lambda_\binx
		\; \le \; 
		\lambda_{\max} \, + \, \lambda_{\min}
		\; \approx \; 4d 
		\] 
		and that their distribution is asymptotically symmetric with respect to $ \lambda =2d $. The second approximation follows from 
		\[
		\alpha
		\; = \; 
		\dfrac{2}{\Lf \, + \, \mf}
		\; = \;
		\dfrac{2}{\lambda_{\max} \, + \, \lambda_{\min}}
		\; \approx \;
		\dfrac{1}{2d}. 
		\] 
		The bounds for the sum of reciprocals of $\lambda_i$ provided in Lemma~\ref{lem:sumRecip} can now be used to complete the proof. 
	\end{proof}

	The following lemma establishes a  relationship between the variance amplifications of Nesterov's method and gradient descent. This relationship allows us to compute tight bounds on $J_\na$ by splitting it into the sum of two terms. The first term  depends linearly on $J_\gd$ which is already computed in Lemma~\ref{lem:GDboundTorus} and the second term can be evaluated  separately using integral approximations for consensus problem on torus networks. This result holds in general for the scenarios in which the largest eigenvalue $\Lf=\Theta(1)$ is bounded and the smallest eigenvalue $\mf$ goes to zero causing the condition number $\kappa$ to go to infinity. 
	\begin{mylem}
		\label{th.variance-asymptotic}
		For a strongly convex quadratic problem with $\mf I \preceq Q \preceq \Lf I$ and condition number $\kappa \DefinedAs \Lf/\mf\ge\kappa_0$, the ratio between variance amplifications of Nesterov's algorithm and gradient descent with the parameters given in Table~\ref{tab:rates} satisfies the asymptotic bounds
\[
\dfrac{c_1}{\sqrt{\kappa}}
\;\le\;
\dfrac
{
	J_{\na} 
	\, - \,
	D
} 
{J_{\gd}}
\;\le\;
c_2, 
~~~~
D
\;\DefinedAs \;
\dfrac{2}{(3\beta+1)\,\alpha_\na^2}
\;\sum_{i\,=\,1}^n \;
\dfrac{1}{\lambda_i^2\,+\,\frac{1\,-\,\beta}{\alpha_\na\beta}\,\lambda_i}
\]
where $\kappa_0$, $c_1$, and $c_2$ are positive constants. Furthermore, depending on the distribution of  the eigenvalues of the Laplacian matrix, $D$ can take values between
	\begin{align}\label{eq.boundsOnDelta}
	\dfrac{c_3}{\kappa}
	\;\le\;
	\dfrac{D}{J_\gd}
	\;\le\;
	c_4\sqrt{\kappa}
	\end{align}
where $c_3$ and $c_4$ are positive constants.
	\end{mylem}
	\begin{proof}	
	We can split $\hat{J}_{\na} (\lambda)/\hat{J}_\gd (\lambda)$ into the sum of two decreasing homographic functions
	$
	\sigma_1(\lambda)\,+\,\sigma_2(\lambda),
	$ where $\sigma_1$ and $\sigma_2$ are defined in~\eqref{eq.Sigma1Sigma2}; see the proof of Proposition~\ref{prop.relationJhat}. Furthermore, for $\kappa \gg 1$, these functions attain their extrema over the interval $[\mf,\Lf]$ at 
	\be\label{eq.lowerUpperBoundsSigma}
	\sigma_1(\Lf)
	\;\approx\;
	\dfrac{9}{8\kappa}\,,
	~
	\sigma_1(\mf)
	\;\approx\;
	\dfrac{3\sqrt{3\kappa}}{8}\,,
	~~
	\sigma_2(\Lf)
	\;\approx\;
	\dfrac{9\sqrt{3}}{16\sqrt{\kappa}}\,,
	~
	\sigma_2(\mf)
	\;\approx\;
	\dfrac{3}{8}
	\ee	
where we have kept the leading terms.
		It is straightforward to verify that 
		\begin{align*}
		\sum_{ i=1 }^n \sigma_1(\lambda_i)\hat{J}_\gd(\lambda_i)
		\;=\;
		\tfrac{2}{(3\beta+1)\alpha_\na^2}
		\ds{\sum_{i\,=\,1}^n}
		\tfrac{1}{\lambda_i^2\,+\,\tfrac{1-\beta}{\alpha_\na\beta}\lambda_i}
		\;=\;
		D.
		\end{align*}
		This equation in conjunction with~\eqref{eq.lowerUpperBoundsSigma}, yield inequalities in~\eqref{eq.boundsOnDelta}. Moreover,  we obtain that 
		\begin{align*}
		\dfrac{J_\na - D}{J_\gd} 
		\;=\;
		\dfrac{\sum_{ i=1 }^n\sigma_2(\lambda_i)\hat{J}_\gd(\lambda_i)}{\sum_{ i=1 }^n\hat{J}_\gd(\lambda_i)}.
		\end{align*}
		This also implies that, asymptotically, 
		\begin{align*}
			\dfrac{J_\na - D}{J_\gd}
			&\;=\;
			O\left(
			\max_{\lambda\in[\mf,\Lf]}\sigma_2(\lambda)\right)
			\; = \;
			O(1)\\[0.15 cm]
			\dfrac{J_\na - D}{J_\gd}			
			&\;=\;
			\Omega\left(\min_{\lambda\in[\mf,\Lf]}\sigma_2(\lambda)\right)
			\;=\;
			\Omega(\dfrac{1}{\sqrt{\kappa}})
		\end{align*}
		which completes the proof.
\end{proof}

The next two lemmas provide us with asymptotic bounds on summations of the form $\sum_\binx {1}/{(\lambda_\binx^2 +\mu \lambda_\binx)}$, where $\lambda_i$ are the eigenvalues of the graph Laplacian matrix of a torus network. These bounds allow us to combine Lemma~\ref{lem:GDboundTorus} and Lemma~\ref{th.variance-asymptotic} to evaluate the variance amplification of Nesterov's accelerated algorithm.

\begin{mylem}  \label{lem.polarInt}
For an integer $q \gg 1$ and any positive $a=O(q^3)$,  we have
	\begin{align*}
		\sum_{0 \, \neq \, i \, \in \, \Z_q^d} \dfrac{1}{\norm{i}^4 + a\norm{i}^2}
		\;\approx\;
		q^{d-4}\int_{{1}/{q}}^{1}\dfrac{r^{d-1}}{r^4+wr^2}\,\mrd r
	\end{align*}
	where  $\omega={a}/{q^2}$.
\end{mylem}
\begin{proof}
The function $h(x) \DefinedAs \norm{x}^4 + \omega\norm{x}^2$ is strictly increasing over the positive orthant ($x\succ0$) and $h((1/q){\one})$ goes to 0 as $q$ goes to infinity where $\one\in\R^d$ is the vector of all ones. Therefore, using the lower and upper Riemann sum approximations, it is straightforward to show that
	$$
	\int \cdots\int_{\Delta\le \norm{x}\le 1}\dfrac{1}{h(x)}\mrd x_1\,\cdots\,\mrd x_d \approx \Delta^d
	\sum_{0\neq i\in\Z_q^d} \dfrac{1}{
	\left(\sum_{l=1}^{d}(\Delta i_l)^2\right)^2
	+
	\omega
	\sum_{l=1}^{d}(\Delta i_l)^2
	}
	$$
	where $\Delta = {1}/{q}$ is the incremental step in the Riemann approximation.
	Therefore, since $\omega = a\Delta^2$, we can write
	$$
	\sum_{0\neq i\in\Z_q^d} \dfrac{1}{\norm{i}^4 + a\norm{i}^2}
	\approx \Delta^{4-d}
	\int \cdots\int_{\Delta\le \norm{x}\le 1}\dfrac{1}{h(x)}\mrd x_1\,\cdots\,\mrd x_d.
	$$
	Finally,  we obtain the result by transforming the integral into a $d$-dimensional polar coordinate system, i.e.,
	\begin{align*}
		\int \cdots\int_{\Delta\le \norm{x}\le 1}\dfrac{1}{h(x)}\mrd x_1\,\cdots\,\mrd x_d
		&\,\approx\,
		\int_{\Delta}^{1}\dfrac{r^{d-1}}{r^4+\omega r^2}\,\mrd r.
	\end{align*}
\end{proof}	
	
	\begin{mylem}
	\label{lem:sumSecPol}
	Let $ \lambda_\binx $ be the eigenvalues of  the Laplacian matrix for the $d$-dimensional torus $ \T_{{\m}}^d $. In the limit of large $ {\m} $, for any $\mu=O({\m})$, we have
	\begin{align}\label{eq.lem.sumSecPol}
	\sum_{0 \, \neq \, \binx \, \in \, \Z_{{\m}}^d}
	\dfrac{1}{\lambda_\binx^2 + \mu\lambda_\binx}
	\; =\; 
	\Theta\left( \m^d \int_{\tfrac{1}{{\m}}}^{1}\dfrac{r^{d-1}}{r^4+\omega r^2}\,\mrd r\right)
	\end{align}
	where $\omega=\Theta(\mu)$.
\end{mylem} 
	\begin{proof}
	Let $ \zeta \,\DefinedAs\, \sum_{0\neq \binx\in\Z_{\m}^d}\dfrac{1}{\lambda_\binx^2+\mu\lambda_\binx} $, where  $\lambda_\binx=2\sum_{l=1}^{d}\left(1-\cos(\binx_l\dfrac{2\pi}{{\m}})\right)$ are the eigenvalues of the graph Laplacian matrix.  Since $ 1-\cos(\cdot-\pi)$ is an even function, for large ${\m} $,
	\begin{align*}
	\zeta
	 \,\approx\, 
	2^d \sum_{0\neq \binx\in\Z_q^d}\dfrac{1}{\lambda_\binx^2 + \mu\lambda_\binx}
	\end{align*} 
	where $ q = \lfloor {{\m}}/{2}\rfloor $. It is well-known that the function $ 1-\cos(x) $ can be bounded by quadratic functions as ${x^2}/{\pi^2} \le 1-\cos(x) \le x^2$ for any $ x\in[-\pi,\pi] $. Now, since for any $ \binx\in\Z_q^d $,  $ \binx_l \tfrac{2\pi}{{\m}}\in[0,\pi]$ for all $ l $, we can use these quadratic bounds to obtain
	\begin{align}\label{eq:temp1}
	\zeta
	\,\approx\,
	\m^4 \sum_{0\neq \binx\in\Z_q^d}\dfrac{1}{\norm{\binx}^4\,+\,c\mu \m^2\norm{\binx}^2}
	\end{align}
	where $c$ is a bounded constant. Finally, equation~\eqref{eq.lem.sumSecPol} follows from Lemma~\ref{lem.polarInt} where we let $a=c\mu \m^2$ and $q\approx {{\m}}/{2}$.
\end{proof}
	
	The following proposition characterizes the network-size-normalized asymptotic variance amplification of noisy consensus algorithms for $d$-dimensional torus networks. This result is used to prove Theorem~\ref{cor:orders}. 
	
\begin{myprop}\label{prop:orders}
	Let $ \Lap \in \R^{n\times n} $ be the graph Laplacian of the $d$-dimensional undirected torus $ \T_{\m}^d $ with $n = \m^d \gg 1$ nodes. For convex quadratic optimization problem~\eqref{eq.Lmin},  the network-size-normalized asymptotic variance amplification $ \bar{J}/n $ of the first-order algorithms on the subspace $ \one^\perp $ is determined by
	\vsp
	\begin{center}
		\begin{tabular}{l|@{\;}cccccc}
			& & {$ d=1 $} & {$ d=2 $} & {$ d=3 $} & {$ d=4 $} & {$ d=5 $}\vspace{0.0 in}
			\\[0.00cm]\hline\\&&&&&&\\[-1.15 cm]
			{\em Gradient} & & {$ \Theta(n) $}& {$ \Theta(\log \, n) $} & {$ \Theta(1) $} & {$ \Theta(1) $} & {$ \Theta(1) $}
			\\[0.15cm]
			{\em Nesterov} & & {$ \Theta(n^2) $}& {$ \Theta(\sqrt{n}\log \, n) $} & {$ \Theta(n^{1/6}) $} & {$ \Theta(\log \, n) $} & {$ \Theta(1) $}
			\\[0.15cm]
			{\em Polyak} & & {$ \Theta(n^2) $}& {$ \Theta(\sqrt{n} \, \log \, n) $} & {$ \Theta(n^{1/3}) $} & {$ \Theta(n^{1/4}) $} & {$ \Theta(n^{1/5}) $}.
		\end{tabular}
	\end{center}
\end{myprop}
\vsp
\begin{proof}
	We prove the result for the three algorithms separately.
	
\begin{enumerate}
	\item For gradient descent, the result follows from dividing the asymptotic bounds established in Lemma~\ref{lem:GDboundTorus} with the total number of nodes $ n=\m^d $.

	\item For  Nesterov's algorithm, we use the relation established in Lemma~\ref{th.variance-asymptotic} to write
\begin{subequations}\label{eq.tempCute1}
\begin{align}
\bar{J}_{\na}/n
\,-\,
\dfrac{c}{n}
\ds{\sum_{\binx}}
\;
\dfrac{1}{\lambda_\binx^2 +\mu \lambda_\binx} 
&\;=\;
O\left( \bar{J}_{\gd}/n\right)\\[0.15 cm]
\bar{J}_{\na}/n
\,-\,
\dfrac{c}{n}
\ds{\sum_{\binx}}
\;
\dfrac{1}{\lambda_\binx^2 +\mu \lambda_\binx}
&\;=\; 
\Omega\left(\, \bar{J}_{\gd}/(n\sqrt{\kappa})\right) 
\end{align}
\end{subequations} 
	where $c=2/\left((3\beta+1)\alpha_\na^2\right)\approx 9d^2/2$ and $ \mu = (1-\beta)/(\alpha_\na\beta) = \Theta({1}/{\sqrt{\kappa}}) = \Theta(\m^{-1})$; see equation~\eqref{eq.kappa_d}. 
	We can use Lemma~\ref{lem:sumSecPol} to compute the second term
	\begin{align}
	\dfrac{1}{n}\sum_{0 \, \neq \, \binx \, \in \, \Z_{\m}^d}
	\dfrac{1}{\lambda_\binx^2 + \mu\lambda_\binx}
	\; =\; 
	\Theta\left( \int_{\tfrac{1}{{\m}}}^{1}\dfrac{r^{d-1}}{r^4+ \omega r^2}\,\mrd r\right)
	\end{align}
	where $\omega=\Theta(\mu)=\Theta(\m^{-1})$. Evaluating the above integral for different values of $d\in \N$ and letting $\omega=\Theta(\m^{-1})$, it is straightforward to show that
	\begin{align*}
		\int_{\tfrac{1}{{\m}}}^{1}\dfrac{r^{d-1}}{r^4 \,+\,\omega r^2}\,\mrd r 
		\;=\;
		\left\{
		\ba{lrcl}
		\Theta(\m^2)& d&\!\!\!=\!\!\!&1\\[0.cm]
		\Theta({\m}\log\,{\m})& d&\!\!\!=\!\!\!&2\\[0.cm]
		\Theta(\sqrt{{\m}})& d&\!\!\!=\!\!\!&3\\[0.cm]
		\Theta(\log\,{\m})& d&\!\!\!=\!\!\!&4\\[0.cm]
		\Theta(1)& d&\!\!\!=\!\!\!&5.
		\ea
		\right.
	\end{align*}
	Finally, the result follows from the asymptotic values for $\bar{J}_\gd/n$ (shown in Part 1) and substituting for the second term on the left-hand-side of equation~\eqref{eq.tempCute1} from the above asymptotic values and using $n=\m^d$. We note that we used the following integrals to evaluate $\bar{J}_\na$,
	\begin{align*}
	\int\dfrac{1}{r^4+\omega r^2}\,\mrd r 
	&\,=\, -\dfrac{\tan^{-1}(\dfrac{r}{\sqrt{\omega}})}{\omega^{3/2}}\,-
	\,\dfrac{1}{r\omega}\\
	\int\dfrac{r}{r^4+\omega r^2}\,\mrd r 
	&\,=\,
	-\dfrac{\log{(r^2+\omega)}-2\log{(r)}}{2\omega}\\
	\int\dfrac{r^2}{r^4+\omega r^2}\,\mrd r 
	&\,=\, \dfrac{\tan^{-1}{(\dfrac{r}{\sqrt{\omega}})}}{\sqrt{\omega}}\\
	\int\dfrac{r^3}{r^4+\omega r^2}\,\mrd r
	&\,=\, \tfrac{1}{2}\log(r^2+\omega)\\
	\int\dfrac{r^4}{r^4+\omega r^2}\,\mrd r
	&\,=\, r-\sqrt{\omega}\tan^{-1}(\dfrac{r}{\sqrt{\omega}}).
	\end{align*}

	\item The result for the heavy-ball method directly follows from the first part of the proof, the relationship between variance amplifications of gradient descent and the heavy-ball method in Theorem~\ref{thm.RelationJhbgd}, and equation~\eqref{eq.kappa_d}.
\end{enumerate}
\end{proof}

We now use Proposition~\ref{prop:orders} to proof Theorem~\ref{cor:orders} as follows. 

\subsubsection*{Proof of Theorem~\ref{cor:orders}:}

As stated in~\eqref{eq.kappa_d}, the condition number satisfies
$
\kappa 
= 
\Theta(n^{2/d}) 
$
and the result follows from combining this asymptotic relation with those provided in Proposition~\ref{prop:orders}.

	\subsubsection*{Computational experiments}
\begin{figure*}
	\begin{tabular}{rcccc}
		\hspace{-.0cm}
		\begin{tabular}{c}
			\vspace{.4cm}
			\rotatebox{90}{$\bar{J}/n$}
		\end{tabular}
		&
		\hspace{-0.75cm}
		\begin{tabular}{c}
			\includegraphics[width=.23\textwidth]{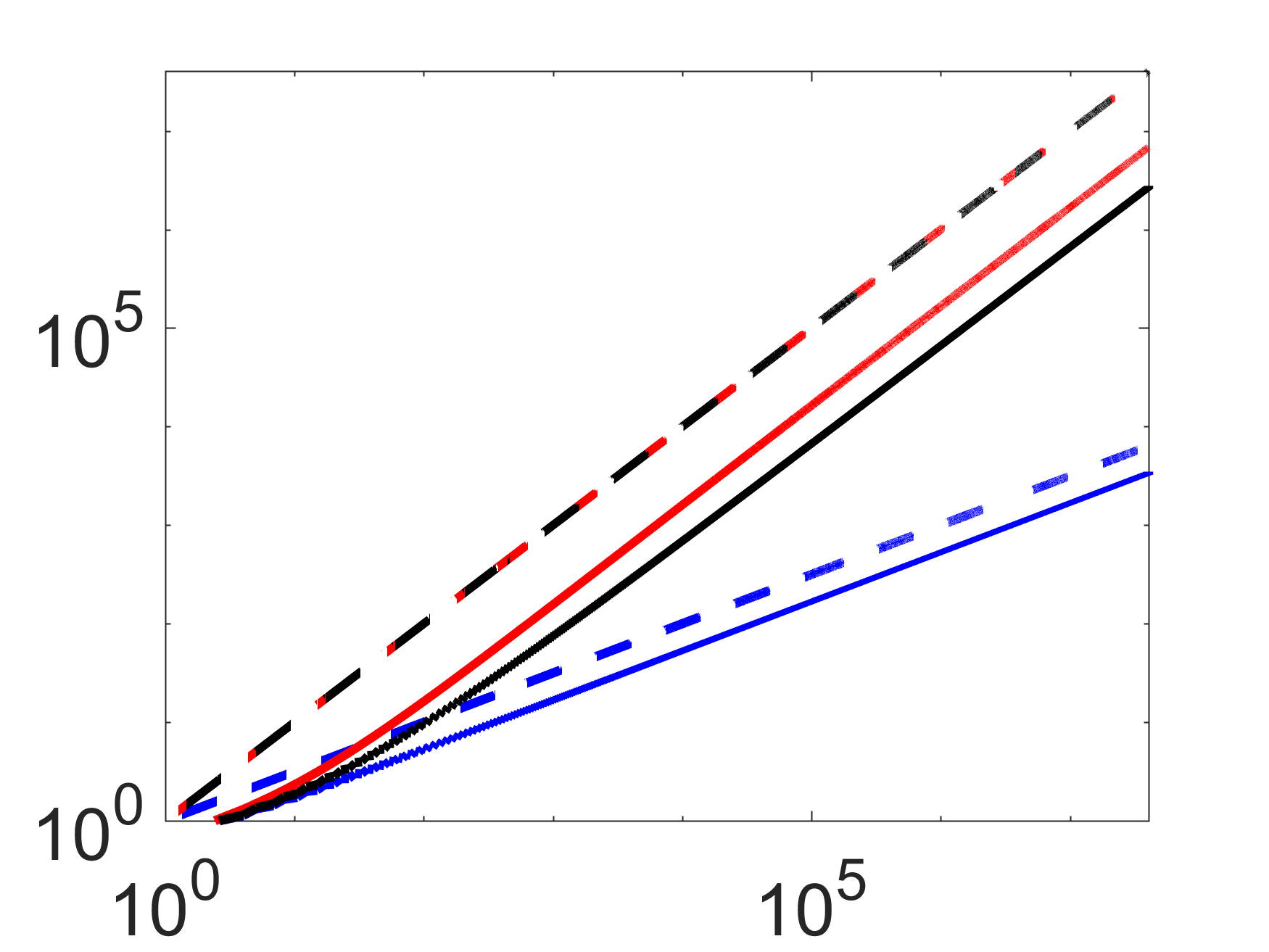}
			\\
			{$ \kappa $}
		\end{tabular}
		&
		\hspace{-1cm}
		\begin{tabular}{c}
			\includegraphics[width=.23\textwidth]{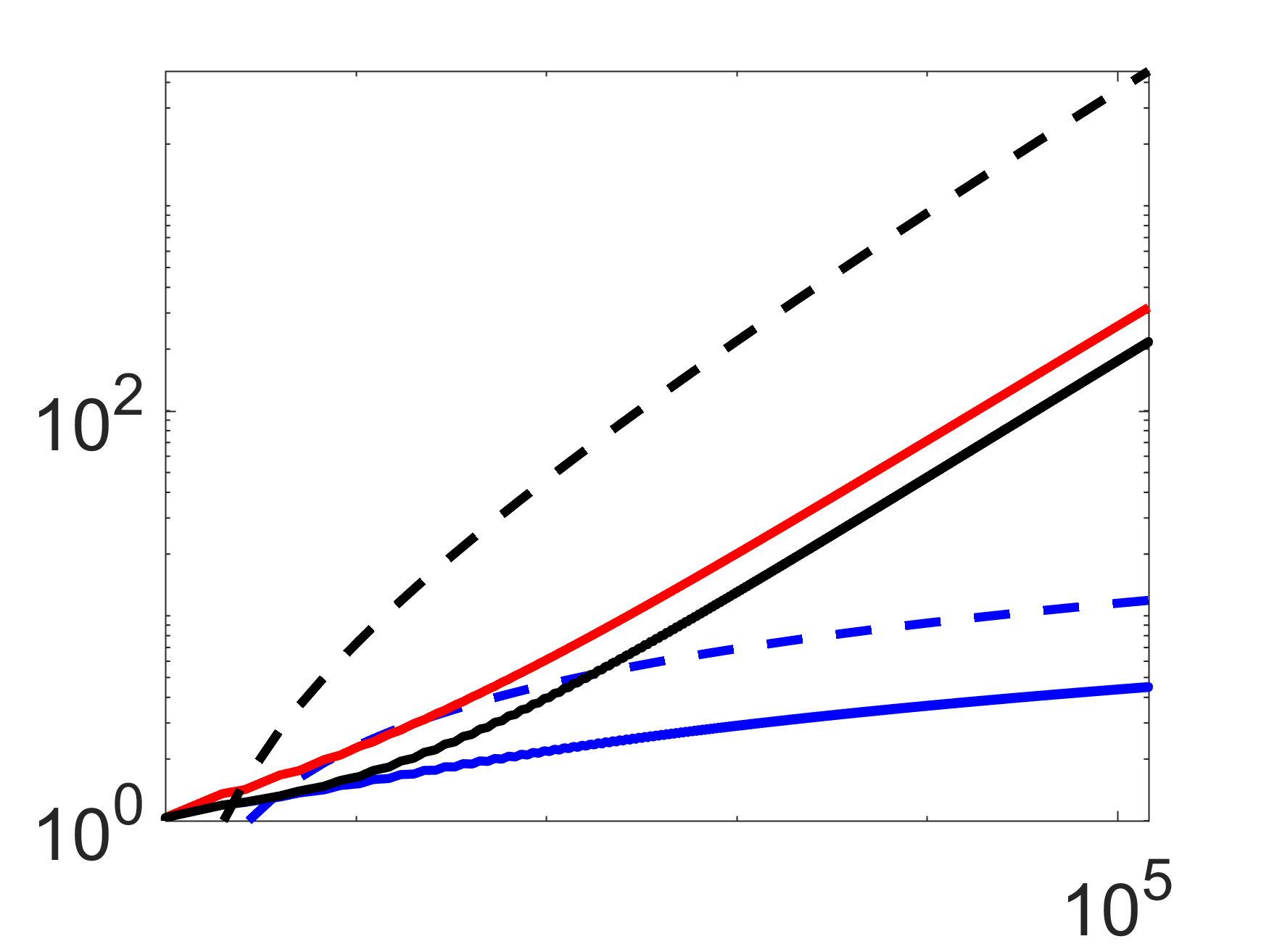}
			\\
			{$ \kappa $}
		\end{tabular}
		&
		\hspace{-1cm}
		\begin{tabular}{c}
			\includegraphics[width=.23\textwidth]{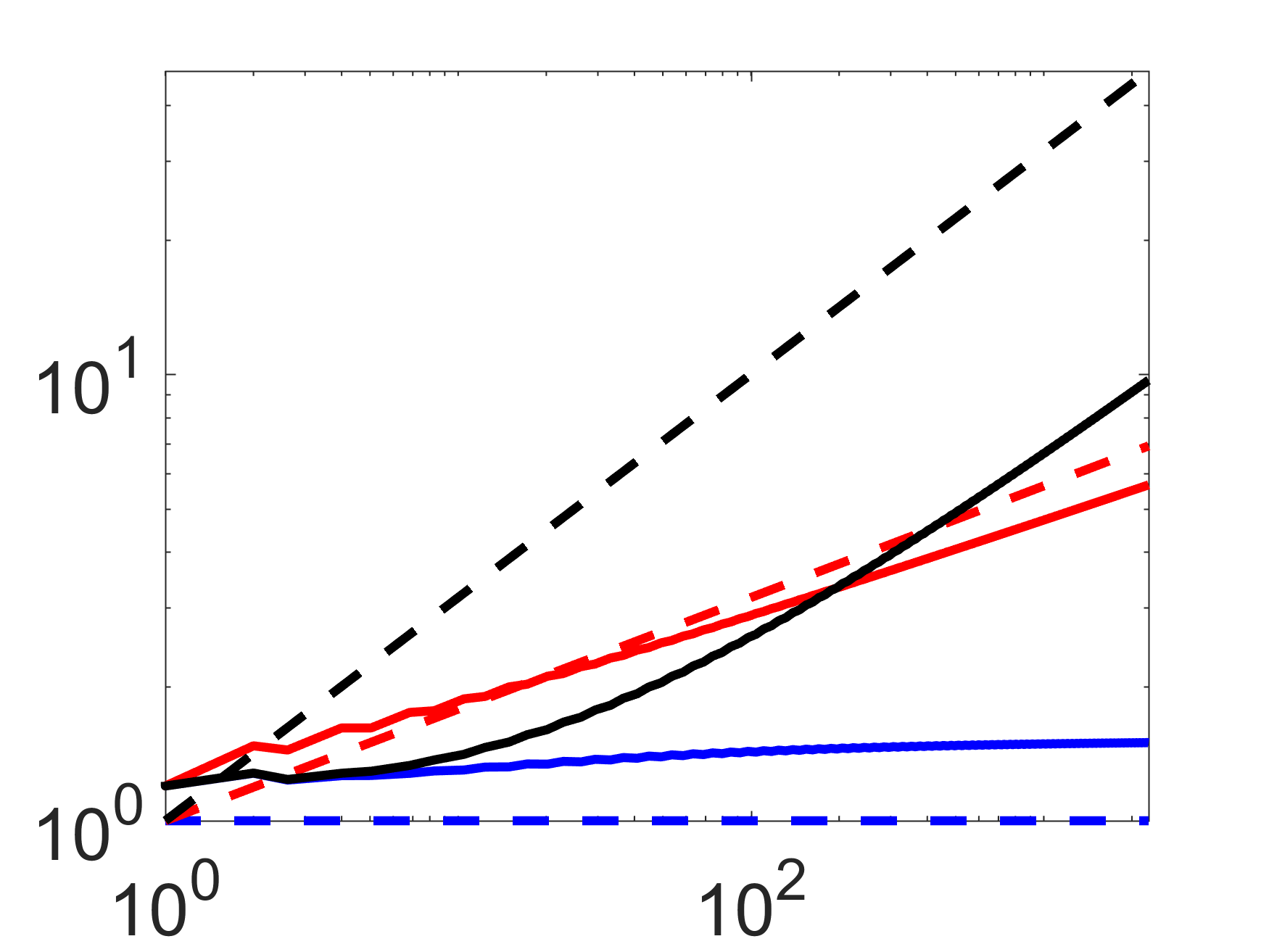}
			\\
			{$ \kappa $}
		\end{tabular}
		&
		\hspace{-1cm}
		\begin{tabular}{c}
			\includegraphics[width=.23\textwidth]{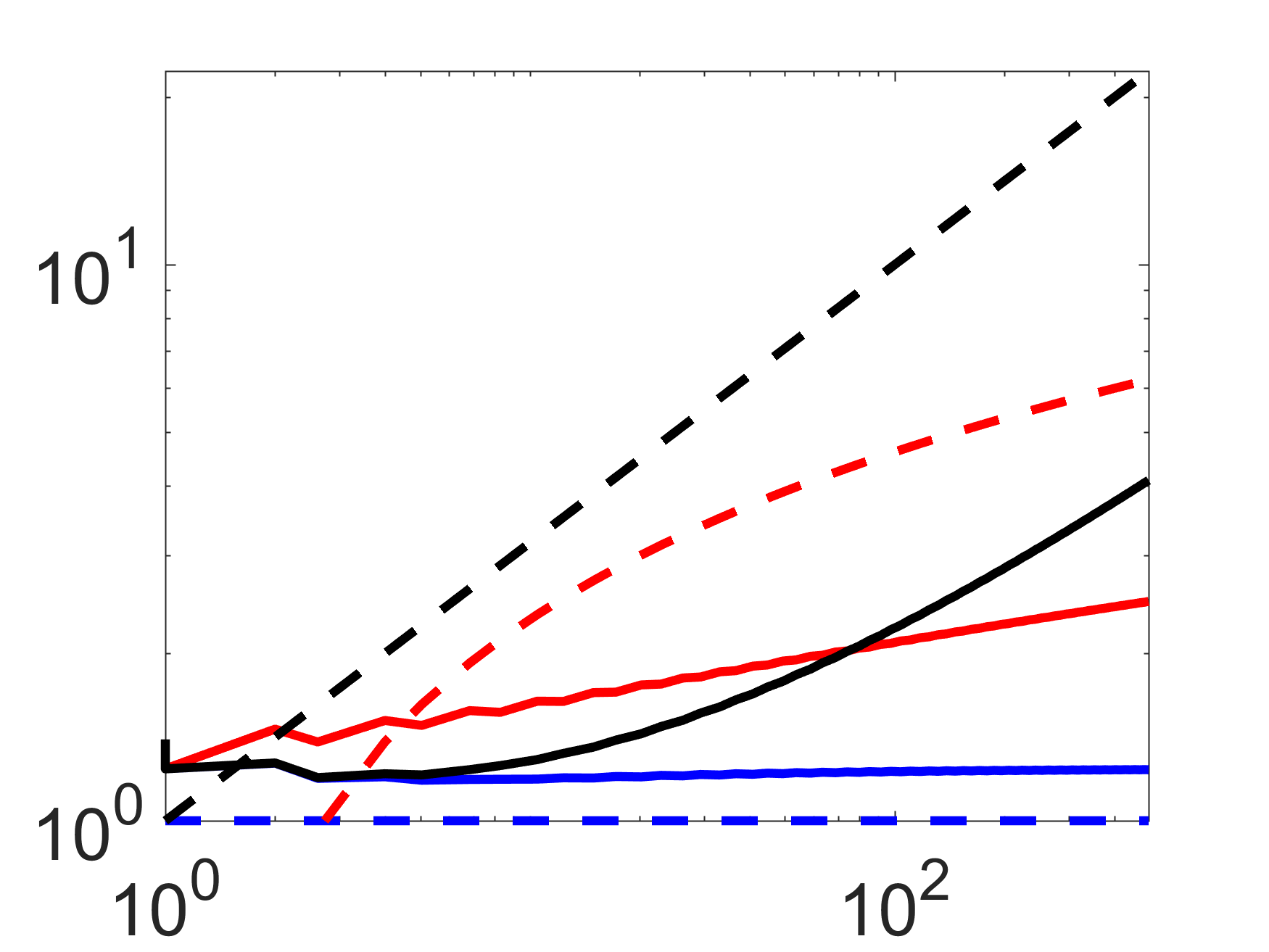}
			\\
			{$ \kappa $}
		\end{tabular}
		\\[-.3cm]
		&
		\hspace{-1.4 cm}
		\begin{tabular}{cl}
			\begin{tabular}{c}
				\subfloat[\label{fig3.1}]{}	
			\end{tabular} & \hspace{-0.6 cm}
			\begin{tabular}{c}\vspace{-0.43 cm}
				$ d=1 $
			\end{tabular}
		\end{tabular}
		&
		\hspace{-1.58cm}
		\begin{tabular}{cl}
			\begin{tabular}{c}
				\subfloat[\label{fig3.2}]{}	
			\end{tabular} & \hspace{-0.6 cm}
			\begin{tabular}{c}\vspace{-0.43 cm}
				$ d=2 $
			\end{tabular}
		\end{tabular}
		&
		\hspace{-1.58cm}
		\begin{tabular}{cl}
			\begin{tabular}{c}
				\subfloat[\label{fig3.3}]{}	
			\end{tabular} & \hspace{-0.6 cm}
			\begin{tabular}{c}\vspace{-0.43 cm}
				$ d=3 $
			\end{tabular}
		\end{tabular}
		&
		\hspace{-1.58cm}
		\begin{tabular}{cl}
			\begin{tabular}{c}
				\subfloat[\label{fig3.4}]{}	
			\end{tabular} & \hspace{-0.6 cm}
			\begin{tabular}{c}\vspace{-0.43 cm}
				$ d=4 $
			\end{tabular}
		\end{tabular}
	\end{tabular}
	\caption{The dependence of the network-size normalized performance measure $\bar{J}/n $ of the first-order  algorithms for $d$-dimensional torus $ \T_{\m}^d $ with $ n=\m^d $ nodes on condition number $ \kappa $. The blue, red, and black curves correspond to the gradient descent, Nesterov's method, and the heavy-ball method, respectively. Solid curves mark the actual values of $ \bar{J}/n $ obtained using the expressions in Theorem~\ref{th.varianceJhat} and the dashed curves mark the trends established in Theorem~\ref{cor:orders}.}
	\label{fig3}
\end{figure*}

To complement our asymptotic theoretical results, we compute the performance measure $\bar{J}$ in~\eqref{eq.Jbar} for the consensus problem over $d$-dimensional torus $ \T_{\m}^d $ with $n = \m^d$ nodes for different values of $ {\m} $ and $ d $. We use expression~\eqref{eq:eigTorusExpression} for the eigenvalues of the graph Laplacian $\Lap$ to evalute the formulae provided in Theorem~\ref{th.varianceJhat} for each algorithm. Figure~\ref{fig3} illustrates network-size normalized variance amplification $ \bar{J}/n $ vs. condition number $\kappa$ and verifies the asymptotic relations provided in Theorem~\ref{cor:orders}. It is noteworthy that, {\em even though our analysis is asymptotic in the condition number (i.e., it assumes that $ \kappa \gg 1 $), our computational experiments exhibit similar scaling trends for small values of $ \kappa $ as well.}

\end{document}

%% file: commands.tex
\newcommand{\enma}[1]   {\ensuremath{#1}}

\newcommand{\beq}{\begin{equation}}
\newcommand{\eeq}{\end{equation}}
\newcommand{\bseq}{\begin{subequations}}
\newcommand{\eseq}{\end{subequations}}
\newcommand{\beqn}{\begin{eqnarray}}
\newcommand{\eeqn}{\end{eqnarray}}
\newcommand{\ba}{\begin{array}}
\newcommand{\ea}{\end{array}}
\newcommand{\bct}{\begin{center}}
\newcommand{\ect}{\end{center}}
\newcommand{\btmz}{\begin{itemize}}
\newcommand{\etmz}{\end{itemize}}
\newcommand{\benum}{\begin{enumerate}}
\newcommand{\eenum}{\end{enumerate}}

\newcommand{\norm}[1]{\| #1 \|}                 

\newcommand{\diag}      {\enma{\mathrm{diag}}}

\newcommand{\trace}     {\enma{\mathrm{trace}}}

\newcommand{\matbegin}{
        \left[
}
\newcommand{\matend}{
        \right]
}

\newcommand{\tbo}[2]{
  \matbegin \begin{array}{c}
       #1 \\ #2
       \end{array} \matend }

\newcommand{\obt}[2]{
  \matbegin \begin{array}{cc}
       #1 & #2
       \end{array} \matend }

\newcommand{\tbt}[4]{
  \matbegin \begin{array}{cc}
       #1 & #2 \\ #3 & #4
       \end{array} \matend }

\newcommand{\tbth}[6]{
  \matbegin \begin{array}{ccc}
       #1 & #2 & #3\\ #4 & #5 & #6
       \end{array} \matend }

\newcommand{\be}{\begin{equation}}
\newcommand{\ee}{\end{equation}}

\newcommand{\cplxs}{ C\kern -.35em \rule{0.03 em}{.7 ex}~   }

\def\complex{\hbox{C\kern -.45em \rule{0.03 em}{1.5 ex}}~}

\newcommand{\bi}{\begin{itemize}}
\newcommand{\ei}{\end{itemize}}